\pdfoutput=1
\documentclass[twoside, 11pt]{article}
\usepackage{geometry}
\usepackage{xcolor} %

\usepackage{times}
\usepackage[utf8]{inputenc}
\usepackage[vietnamese,english]{babel}
\newif\ifdraftbib
\draftbibfalse
\ifdraftbib
  \AtBeginDocument{\renewcommand{\cite}[2][]{[\texttt{#2}\if\relax\detokenize{#1}\relax\else, #1\fi]}}
  \newcommand{\printbibliography}{}
\else
\usepackage[
    backend=biber,
    style=alphabetic,
    sorting=anyt,
    sortlocale=vi-VN,
    maxbibnames=99,     %
    maxalphanames=4,    %
    minalphanames=3     %
]{biblatex}
\fi

\usepackage{amsmath} %
\usepackage{amsthm}
\usepackage{amsfonts}
\usepackage{amssymb}
\usepackage{mathtools}
\usepackage[shortlabels]{enumitem}

\usepackage[bitstream-charter]{mathdesign}
\usepackage[T1]{fontenc}
\usepackage[new]{old-arrows}
\usepackage{graphicx}

\usepackage{fancyhdr}
\usepackage{sectsty}
\sectionfont{\fontsize{13}{15}\selectfont}

\usepackage[margin=1cm]{caption}
\usepackage{placeins} %
\usepackage[hidelinks]{hyperref}
\usepackage{cleveref}
\usepackage{aliascnt}

\numberwithin{equation}{section}

\newtheoremstyle{theorem}%
{3mm}%
{1mm}%
{\normalfont\itshape}%
{}%
{\normalfont\scshape}%
{.}%
{.5em}%
{\thmname{#1}\thmnumber{ #2}\thmnote{ (#3)}}%
\theoremstyle{theorem}

\newtheorem{theorem}{Theorem}
\numberwithin{theorem}{section}

\newaliascnt{lemma}{theorem}
\newtheorem{lemma}[lemma]{Lemma}
\aliascntresetthe{lemma}
\crefname{lemma}{Lemma}{Lemmas}
\Crefname{lemma}{Lemma}{Lemmas}

\newaliascnt{proposition}{theorem}
\newtheorem{proposition}[proposition]{Proposition}
\aliascntresetthe{proposition}
\crefname{proposition}{Proposition}{Propositions}
\Crefname{proposition}{Proposition}{Propositions}

\newaliascnt{corollary}{theorem}
\newtheorem{corollary}[corollary]{Corollary}
\aliascntresetthe{corollary}
\crefname{corollary}{Corollary}{Corollaries}
\Crefname{corollary}{Corollary}{Corollaries}

\newaliascnt{claim}{theorem}

\aliascntresetthe{claim}
\crefname{claim}{Claim}{Claims}
\Crefname{claim}{Claim}{Claims}

\newaliascnt{assumption}{theorem}

\aliascntresetthe{assumption}
\crefname{assumption}{Assumption}{Assumptions}
\Crefname{assumption}{Assumption}{Assumptions}

\newaliascnt{conjecture}{theorem}

\aliascntresetthe{conjecture}
\crefname{conjecture}{Conjecture}{Conjectures}
\Crefname{conjecture}{Conjecture}{Conjectures}

\newaliascnt{question}{theorem}

\aliascntresetthe{question}
\crefname{question}{Question}{Questions}
\Crefname{question}{Question}{Questions}

\newtheoremstyle{definition}%
{3mm}%
{1mm}%
{\normalfont\normalfont}%
{}%
{\normalfont\scshape}%
{.}%
{.5em}%
{\thmname{#1}\thmnumber{ #2}\thmnote{ (#3)}}%
\theoremstyle{definition}

\newaliascnt{definition}{theorem}
\newtheorem{definition}[definition]{Definition}
\aliascntresetthe{definition}
\crefname{definition}{Definition}{Definitions}
\Crefname{definition}{Definition}{Definitions}

\newaliascnt{notation}{theorem}

\aliascntresetthe{notation}
\crefname{notation}{Notation}{Notations}
\Crefname{notation}{Notation}{Notations}

\newaliascnt{example}{theorem}
\newtheorem{example}[example]{Example}
\aliascntresetthe{example}
\crefname{example}{Example}{Examples}
\Crefname{example}{Example}{Examples}

\newaliascnt{remark}{theorem}
\newtheorem{remark}[remark]{Remark}
\aliascntresetthe{remark}
\crefname{remark}{Remark}{Remarks}
\Crefname{remark}{Remark}{Remarks}
\theoremstyle{theorem}
\newtheorem{mainthm}{Theorem}

\crefname{mainthm}{Theorem}{Theorems}
\Crefname{mainthm}{Theorem}{Theorems}
\newaliascnt{mainprop}{mainthm}
\newtheorem{mainprop}[mainprop]{Proposition}
\aliascntresetthe{mainprop}
\crefname{mainprop}{Proposition}{Propositions}
\Crefname{mainprop}{Proposition}{Propositions}
\theoremstyle{definition}

\DeclareMathOperator{\SL}{SL}

\DeclareMathOperator{\Div}{div}

\DeclareMathOperator{\End}{End}
\DeclareMathOperator{\Hom}{Hom}

\newcommand{\pa}{\partial}

\newcommand{\C}{\mathbb C}

\newcommand\dqd{\foreignlanguage{vietnamese}{Đinh Quý Dương}}

\DeclareUnicodeCharacter{039F}{O}

\title{Stratification of opers with apparent singularities  \\
and isomonodromic bubbling}
\author{\dqd, Troy Figiel, Jörg Teschner}
\date{}

\begin{document}
\maketitle

\begin{abstract}
On a compact Riemann surface of genus $g \geq 2$, we study Simpson's  stratification of the de Rham moduli space of rank-2 flat connections.
Upon representing flat $\mathrm{SL}_2(\mathbb{C})$-connections by branched projective structures, or equivalently by opers with apparent singularities, the stratum containing a flat connection is determined by the minimal number of apparent singularities needed to represent it.
We construct explicit holomorphic families in which two simple apparent singularities collide and disappear, so that the family crosses into a smaller stratum; this realizes Simpson's closure relation between strata near generic points.
The analogous families of Higgs bundles are described by Baker--Akhiezer divisors on the spectral curve in which two points converge to a pair of points exchanged by the spectral involution.
We show that, before the collision, the two apparent singularities are joined by twin paths bounding a bubble, so that these families arise from the isomonodromic bubbling surgery discussed by Calsamiglia, Deroin and Francaviglia.
\end{abstract}

\tableofcontents

\section{Introduction}

Let $C$ be a compact Riemann surface of genus $g \geq 2$.
The Hitchin moduli spaces of Higgs bundles $(E, \phi)$ and the de Rham moduli spaces of flat connections $(E, \nabla)$ on $C$ 
have a natural stratification, 
defined by the Bia\l ynicki-Birula stratification on the Hodge moduli space of $\lambda$-connections.
The main goal of this paper is to explicitly describe families of Higgs bundles and flat connections, of rank 2 and with fixed $\det(E)$, 
that cross into smaller strata.
Our method is to parametrize Higgs bundles and flat connections by divisors 
on $C$ together with auxiliary parameters, 
namely lifts to spectral curves $\widetilde{C}$ and accessory parameters of apparent singularities, respectively.
The degeneration families then correspond to colliding simple points of the divisors
and have these points annihilate at the limit.

In the following, we introduce the setting in \Cref{subsec-intro-strata}, 
state the main results in \Cref{subsec-intro-results} 
and discuss relations to other subjects in \Cref{subsec-intro-discussion}.

\subsection{Natural stratification}\label{subsec-intro-strata}
\paragraph{Hodge bundles and strata}
Consider the Hodge moduli space $\mathcal{M}_{Hod}$ of rank-2 $\lambda$-connections $(\lambda, E, \nabla_\lambda)$ on $C$ with trivial determinant.
(For Higgs bundles alone, we also allow a fixed determinant $\Lambda$ of degree $1$; see \Cref{sect-preliminaries}.)
Its fibers over $\lambda = 0$ and $\lambda = 1$ are the Hitchin moduli space $\mathcal{M}_H$ of Higgs bundles and the de Rham moduli space $\mathcal{M}_{dR}$ of flat $\SL_2(\C)$-connections.
The $\C^\ast$-action, defined by 
$t.(\lambda, E, \nabla_\lambda) = (t\lambda, E, t\nabla_\lambda)$, 
has fixed points contained in the nilpotent cone of the Hitchin moduli space. 
We call such fixed points \textit{Hodge bundles}.
For all $x \in \mathcal{M}_{Hod}$, the limit $\underset{t \rightarrow 0}{\lim} t.x$ is a Hodge bundle.
The attractors associated to components of Hodge bundles define the Bia\l ynicki-Birula stratification of $\mathcal{M}_{Hod}$, 
the restriction of which to $\mathcal{M}_H$ and $\mathcal{M}_{dR}$ defines respectively the Bia\l ynicki-Birula stratification and Simpson's partial oper stratification on these moduli spaces \cite{simpson-iterated, collier-wentworth}.

In rank 2, the stratifications of $\mathcal{M}_H$ and $\mathcal{M}_{dR}$ 
are both governed by the maximal subbundles of the underlying bundles $E$, i.e. line subbundles of maximal degree.
Specifically, the largest open stratum consists of $(E, \phi)$ or $(E, \nabla)$ with $E$ stable,
and the strata with unstable underlying bundles are labelled by the degree of the maximal subbundle.
One can further refine the open stratum by the Segre stratification of the moduli space $\mathcal{N}$ of stable bundles.

A Higgs field $\phi$ (respectively, a holomorphic connection $\nabla$) on $E$ together with a line subbundle $L \hookrightarrow E$
induces a section of $KL^{-2}\det(E)$, 
which we denote by $s_L(\phi)$ (respectively, $s_L(\nabla)$).
When this section is nonzero, we denote its divisor of zeros by $D_L(\phi)$ (respectively, $D_L(\nabla)$).
In this paper, with $\Lambda = \det(E)$, we will label the strata by
\[ d = \deg(KL^{-2}\det(E)) = 2g-2-2\deg L + \deg(\Lambda), \qquad L \hookrightarrow E \text{ maximal}. \]
Note that $d \equiv \deg(\Lambda) \pmod 2$.
The strata with $d < 2g-2$ consist of Higgs bundles and flat connections with unstable underlying bundles. 
Simpson showed that, for $d < 2g-2$, the stratum of $\mathcal{M}_{dR}$ labelled by $d-2$ lies in the closure of the one labelled by $d$ \cite[Theorem~7.6]{simpson-iterated}.

\paragraph{Flat connections as branched projective structures.}
A rank-2 flat connection can be represented by a branched projective structure subordinate to $C$, or equivalently by an oper with apparent singularities (\Cref{subsec-projective-structures}).
A \emph{branched projective structure} is an atlas of local holomorphic maps from $C$ to $\mathbb{P}^1$ patched by M\"obius transformations. 
Its \emph{apparent singularities}, or branch points, are the critical points of these local maps,
and its \textit{developing map} from the universal cover of $C$ to $\mathbb{P}^1$ is defined by 
analytic continuation of one such local map.
An \emph{oper with apparent singularities} is an atlas of compatible local Schr\"odinger equations $(\partial_z^2 + q(z)) f = 0$, where $q$ has poles of order $2$ at the apparent singularities, with Laurent tails ensuring that the ratio of two solutions is holomorphic there.
For example, at a simple apparent singularity $x$,
\[ q(z) = -\frac{3}{4(z-z(x))^2} + \frac{\nu}{z-z(x)} + q_0 + \mathcal{O}(z - z(x)), \qquad \nu^2 + q_0 = 0, \]
and $\nu$ is called the \emph{accessory parameter} of $x$.
The two descriptions are related by taking the ratio $w = f_1/f_2$ of two independent local solutions.
This ratio is a local holomorphic map to $\mathbb{P}^1$ whose critical points are the apparent singularities, and a change of basis of solutions changes it by a M\"obius transformation; hence these local maps form the atlas of a branched projective structure.
Conversely, the oper is recovered from the atlas of the branched projective structure by taking Schwarzian derivatives of the local maps, $q = \frac{1}{2}\{w, z\}$.

A flat connection $(E, \nabla)$ together with a line subbundle $L \hookrightarrow E$ defines such a branched projective structure: parallel transport of the lines $L_x \hookrightarrow E_x$ gives local maps to $\mathbb{P}^1$ whose critical points, i.e.\ the apparent singularities, are the $2g-2-2\deg L$ points of $D_L(\nabla)$.
Hence the stratum of $\mathcal{M}_{dR}$ labelled by $d$ consists of the flat connections for which the minimal number of apparent singularities needed to represent them is $d$ (\Cref{par-PSL2-strata}).

\subsection{Main results}\label{subsec-intro-results}
We now state the main results: how colliding simple apparent singularities define a family that crosses to a smaller stratum. 
Fix a holomorphic projective structure on $C$, e.g.\ the uniformizing one. 
For $x_0 \in C$, let $(U, z)$ be a local coordinated neighborhood centred at $x_0$ given by a chart of this projective structure, 
and for $u \in z(U)$ let $x_\pm(u)$ be the points with $z(x_\pm(u)) = \pm u$.

\paragraph{Crossing to smaller partial oper strata}
\begin{mainprop}[Reduction of apparent singularities] \label{prop-A}
	Let $\mathcal{D}$ be an oper on $C$ 
    with $d-2$ simple apparent singularities $x_3, \dots, x_d$, where $d \leq 3g-3$, and let $x_0 \in C$ be generic.
	Then there exists a holomorphic family $\{\mathcal{D}_u\}_{u\in z(U)}$ with $\mathcal{D}_0 = \mathcal{D}$ such that, for $u \neq 0$, $\mathcal{D}_u$ has simple apparent 
    singularities at $x_+(u), x_-(u), x_3, \dots, x_d$, and the accessory parameters of $x_\pm(u)$ with respect to the projective chart $z$ are of the form
	\[ \nu_\pm(u) = \pm \frac{3}{4u} + \mathcal{O}(u). \]
	(See \Cref{prop-bb-limit-3}; the precise form of $\nu_\pm(u)$ is \eqref{eqn-form-B}.)
\end{mainprop}

\begin{mainthm}[Crossing to smaller strata]\label{thm-B}
	Suppose moreover that $d$ is even and that the projective monodromy of $\mathcal{D}$ is irreducible.
	Then the holomorphic family $\{\mathcal{D}_u\}$ of \Cref{prop-A} is defined by a holomorphic family $\{[E_u, \nabla_u]\}_u$ in $\mathcal{M}_{dR}$ 
	together with line subbundles $L_u \hookrightarrow E_u$ of degree
	\[ \deg L_u = g-1-\frac{d}{2} \quad \text{for } u \neq 0, \qquad \qquad \deg L_0 = g - \frac{d}{2}, \]
	and $D_{L_u}(\nabla_u) = x_+(u) + x_-(u) + x_3 + \dots + x_d$ for $u \neq 0$, $D_{L_0}(\nabla_0) = x_3 + \dots + x_d$.
	For $d < 2g-2$, $[E_u, \nabla_u]$ lies in the partial oper stratum labelled by $d$ for $u \neq 0$, 
	and in the partial oper stratum labelled by $d-2$ at $u = 0$.
	(See \Cref{prop-family-dR,cor-crossing-strata}.)
\end{mainthm}

We observe that \Cref{thm-B} gives explicit families realizing Simpson's closure relation near generic points of the smaller stratum.
On $\mathbb{P}^1$, a collision of two apparent singularities at a jump of the type of the underlying bundle was analyzed by Komyo and Saito for rank-2 parabolic connections with five regular singular points \cite[Theorem~5.5]{KomyoSaito19}; there too, the dual parameters of the colliding apparent singularities diverge.

\paragraph{Analogues for Higgs bundles}
\Cref{prop-A,thm-B} have exact analogues for Higgs bundles, in terms of Baker--Akhiezer divisors (\Cref{sect-degeneration-BA}).
If $(E, \phi)$ has a smooth spectral curve $\pi \colon \widetilde{C} \rightarrow C$, the points of $D_L(\phi)$ lift to the \emph{Baker--Akhiezer divisor} of $(L \hookrightarrow E, \phi)$ on $\widetilde{C}$ \cite{DT23}, which selects over each point one of the two sheets, i.e.\ one of the two eigenvalues of $\phi$ (\Cref{subsec-BA-divisors}).
In the analogue of \Cref{prop-A}, two points of the Baker--Akhiezer divisor on different sheets collide over $x_0$ and disappear; for $d < 2g-2$, 
the family crosses from the stratum labelled by $d$ to the one labelled by $d-2$ (\Cref{prop-Hitchin-bb}).
If the two points lie on the same sheet, they merge into a double point instead (\Cref{prop-Hitchin-db}); 
on the de Rham side, this corresponds to accessory parameters $\nu_\pm(u) = \mp \frac{1}{4u} + \mathcal{O}(1)$, which produce a double apparent singularity (\Cref{prop-double-as}).
For $d < 2g-2$, the limiting Hodge bundle of a family with a double point has a non-reduced divisor and is therefore \emph{wobbly}: its upward flow meets the nilpotent cone in more than one point \cite{HH21}.

It is clear that analogous results in higher rank hold for specific cases: 
one might degenerate Baker--Akhiezer divisors to have a summand of the form $\pi^\ast (p)$ for $p \in C$ to cross to a lower stratum.
We refer to \cite{HH21, DDX26} for the necessary background on Baker--Akhiezer divisors for higher rank cases. 

\paragraph{Relation to isomonodromic bubbling}
Bubbling is an isomonodromic surgery on branched projective structures \cite{CalsamigliaDeroinFrancaviglia2014}.
The input of the surgery is one branched projective structure and an embedded arc along which the developing map is injective.
We cut the surface along this arc 
and glue in a copy of $\mathbb{P}^1$ cut along the developed image of the arc.
The output is another branched projective structure with the same monodromy and two more apparent singularities at the end points of the arc, and is in general subordinate to a different complex structure.
The glued-in disc, a \emph{bubble}, is bounded by two \emph{twin paths}: paths between the two new branch points with the same developed image.

\begin{mainthm}[Bubbling]\label{thm-C}
	For sufficiently small $u \neq 0$, 
    the branched projective structure corresponding to $\mathcal{D}_u$ of \Cref{prop-A} 
    is the output of a bubbling surgery 
    that creates branch points that correspond to the apparent singularities $x_+(u)$ and $x_-(u)$.
    Specifically, this bubbling surgery takes as an input the branched projective structure $\sigma_u$, 
    which is obtained by removing the bubble bounded by the twin paths and identifying these twin paths, 
    and has simple apparent singularities at $x_3, \dots, x_d$.
	(See \Cref{thm-bubble}.)
\end{mainthm}

\begin{figure}[ht]
  \centering
  \includegraphics[width=0.8\textwidth]{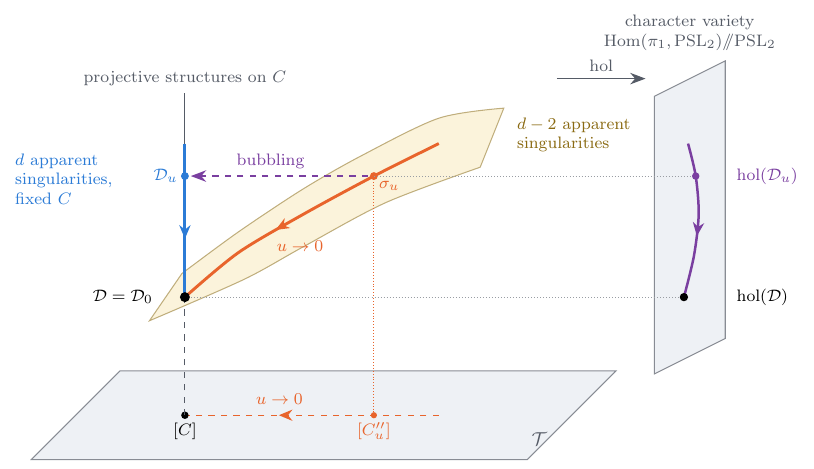}
  \caption{
The two families $\mathcal{D}_u$ on the fixed surface $C$ and $\sigma_u$ on $C''_u$ of \Cref{thm-bubble} over the Teichm\"uller space $\mathcal{T}$.
The inputs $(C''_u, \sigma_u)$ and outputs $\mathcal{D}_u$ of the bubbling operation lie on the same isomonodromic flow. 
These two families both have $\mathcal{D}_0$ as their limits.
}
  \label{fig-debubbling}
\end{figure}

\subsection{Discussion}
\label{subsec-intro-discussion}
We now briefly discuss the relation of the main objects and results of this paper with other subjects.

\paragraph{Apparent singularities as classical limits of degenerate fields}
Opers with apparent singularities arise in the classical limit of conformal field theory \cite[Section~3]{Te17}.
Consider Virasoro conformal blocks on the Riemann sphere with primary fields at marked points $p_1, \dots, p_N$, degenerate fields $V_{-1/2b}$ at points $x_1, \dots, x_d$, and one further degenerate field $V_{-b/2}$ at a variable point $y$.
The null vectors of the degenerate fields imply that the conformal block satisfies second-order Belavin--Polyakov--Zamolodchikov (BPZ) equations in $y$ and in each $x_r$ \cite{BPZ}.
In the classical limit $b \rightarrow 0$, with the ansatz
$e^{-S/b^2}\chi(y)$ for conformal blocks, the BPZ equation in $y$ 
degenerates into an oper equation for $\chi$ with simple apparent singularities at $x_1, \dots, x_d$.

The accessory parameters are the derivatives $\nu_r = -\partial_{x_r} S$ of the classical conformal block $S$, and the BPZ equations at the $x_r$ become exactly the constraints that make the $x_r$ apparent singularities.
Furthermore, the residues $\partial_{p_s} S$ at the marked points are the Hamiltonians of the isomonodromic deformation \cite{Te17}.
For higher rank, opers with apparent singularities and the conditions on their accessory parameters are studied in \cite{DM07, DDX26}, and one may expect them to arise in the classical limit of $\mathfrak{sl}_n$ Toda theories.

\paragraph{Collisions as fusion}
The outcome of the collision of two insertions of degenerate fields 
is described by the fusion rule in conformal field theories.
The degenerate field $V_{-1/2b}$ is labelled by the spin-$\frac12$ representation of $\mathfrak{sl}_2$, 
and its fusion with itself follows the Clebsch--Gordan rule 
\[ \frac12 \otimes \frac12 = 0 \oplus 1\]  
\cite{BPZ}.
The two channels are the vacuum representation, which corresponds to inserting no field at all, and the next degenerate representation $V_{-1/b}$, labelled by spin $1$, whose classical limit is a double apparent singularity.

Our results on the annihilation of two apparent singularities (\Cref{prop-A}) and on the formation of a double apparent singularity (\Cref{prop-double-as}) can be seen as the classical limit of this fusion rule, in the vacuum and the spin-$1$ channel respectively.

\paragraph{Neighborhoods of strata}
\Cref{prop-A,thm-B,thm-C} suggest a local picture of $\mathcal{M}_{dR}$ near a smaller stratum.
Let $\mathcal{D}$ be as in \Cref{prop-A}, with $d-2$ simple apparent singularities $x_3, \dots, x_d$, and suppose that it represents a flat connection in the stratum labelled by $d-2$.
We expect that every flat connection close to it can be represented by a branched projective structure whose apparent singularities consist of $d-2$ points close to $x_3, \dots, x_d$, together with pairs of nearby points, each pair joined by twin paths bounding a small bubble.
The family $\{\mathcal{D}_u\}_u$ of \Cref{prop-A} is of this form, with the single pair $x_\pm(u)$ (\Cref{thm-C}).

On the other hand, if the complex structure is allowed to vary, the pair can be removed: 
by \Cref{thm-C}, for small $u \neq 0$, $\mathcal{D}_u$ is obtained by bubbling a branched projective structure $\sigma_u$ on $C'_u$ 
with two fewer apparent singularities and the same monodromy.
Hence, for even $d < 2g-2$, the monodromy of $\mathcal{D}_u$ needs $d$ apparent singularities on $C$ by \Cref{thm-B}, but at most $d-2$ on $C'_u$.
This suggests describing a neighborhood of a stratum in terms of the space of branched projective structures with $d-2$ apparent singularities over Teichm\"uller space, 
which one may call a \emph{universal Bia\l ynicki-Birula stratum}. 
See \Cref{fig-debubbling} for an illustration.

\paragraph{Discontinuity of tangent spaces at smaller strata}
Let $[E_0, \nabla_0]$ lie in the stratum labelled by $d-2$, where $d < 2g-2$, let $L_0 \hookrightarrow E_0$ be its maximal subbundle, and for $x_0 \in C$ let $L_0' = L_0(-x_0)$.
Simpson proved the closure relation by deforming $(E_0, \nabla_0)$, together with the injection $L_0' \rightarrow L_0 \hookrightarrow E_0$ vanishing at $x_0$, 
into a family $(E_\tau, \nabla_\tau)$ with line subbundles $L_\tau \hookrightarrow E_\tau$ of degree $\deg L_0'$, which lies in the larger stratum labelled by $d$ for $\tau \neq 0$ \cite[Theorem~7.6]{simpson-iterated}, see also \cite[Proposition~2.2]{Huang26}.
In terms of apparent singularities, this is the mechanism of \Cref{prop-A,thm-B}: 
the section induced by $L_0' \rightarrow E_0$ has zero divisor $D_{L_0}(\nabla_0) + 2x_0$, 
and if $x_0$ is not a point of $D_{L_0}(\nabla_0)$, then for $\tau \neq 0$ the double point $2x_0$ splits into two simple apparent singularities, which collide at $x_0$ as $\tau \rightarrow 0$.

For $d \geq 4$, Huang chooses $x_0$ to be a point of $D_{L_0}(\nabla_0)$, so that three apparent singularities collide at $x_0$ \cite{Huang26}.
He shows that along this family that limits to $[E_0, \nabla_0]$, the tangent spaces along the Lagrangian leaves associated to Hodge bundles display a discontinuity at the collision limit \cite[Theorem~1.1]{Huang26}.
Hence these Lagrangian leaves do not form a $C^1$ foliation of the stable locus of $\mathcal{M}_{dR}$, which answers the foliation part of Simpson's question \cite[Question~7.4]{simpson-iterated} negatively; the closedness part holds in rank $2$ \cite{DS24}.
Our results suggest an intuitive picture of this collision (\Cref{fig-three-collision}): 
two simple apparent singularities, joined by twin paths bounding a small bubble, collide with a third apparent singularity, which is itself joined by twin paths to another apparent singularity.

\begin{figure}[ht]
  \centering
  \includegraphics[width=0.9\textwidth]{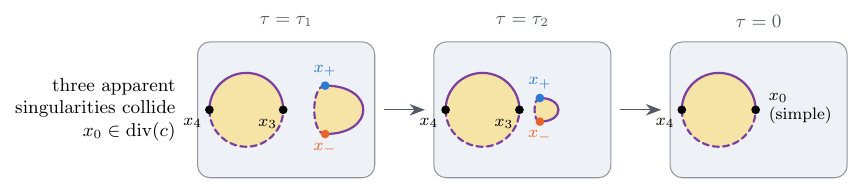}
  \caption{
Collision of three apparent singularities in Huang's deformation \cite{Huang26}
}
  \label{fig-three-collision}
\end{figure}

\paragraph{Beilinson--Drinfeld approach to the geometric Langlands correspondence}
The geometric Langlands correspondence (GLC), recently proved by Gaitsgory--Raskin et al.~\cite{GLC},
was pioneered by Drinfeld \cite{Dr83} and Beilinson--Drinfeld \cite{BD}.
The special case Beilinson--Drinfeld proved belongs to that of holomorphic opers (i.e. those \textit{without} apparent singularities):
given a complex semisimple group $G$ and a holomorphic oper $\sigma$ on $C$ with gauge group being the Langlands dual $\check{G}$,
using techniques from conformal field theories they constructed the corresponding Hecke eigensheaf $\mathcal{F}_\sigma$ on the moduli stack $\mathrm{Bun}_G$ of $G$-bundles on $C$ \cite{BD}, see also \cite{Fr07}.
In particular, $\mathcal{F}_\sigma$ can be realized by a holonomic system of eigenvalue equations of differential operators called \textit{quantum Hitchin Hamiltonians} that act on $K^{1/2}_{\mathrm{Bun}_G}$.

From this point of view, one might be motivated to understand the construction of the Hecke eigensheaves $\mathcal{F}_\sigma$ 
for a general local system $\sigma$ in the proximity of holomorphic opers.
Frenkel and Teleman extended the equivalence of Beilinson and Drinfeld on the locus of opers to the formal neighborhood of this locus \cite{FT13}. 
In contrast, for $\check{G} = \mathrm{PSL}_2(\C)$ and odd $g$, our results suggest a concrete description of an analytic neighborhood of the locus of holomorphic opers.
Namely, we expect that a generic flat connection in a neighborhood of opers can be represented 
by an oper with $3g-3$ apparent singularities, 
forming $\frac{3g-3}{2}$ pairs of nearby points, each joined by twin paths bounding a small bubble.
Debubbling all of these bubbles gives a holomorphic oper with the same monodromy on a nearby Riemann surface.

\paragraph{Acknowledgement} We thank Dima Arinkin, Pengfei Huang and Motohico Mulase for helpful discussions on related topics. 

\vspace{-10pt}
\paragraph{Tool and computational resource disclosure}
The explicit families of projective structures and an approximate calculation evidencing a relation to isomonodromic bubbling were worked out during the PhD studies of the first and second author with the third author. 
These results were worked out without AI and recorded in the first author's thesis \cite{dinh-thesis}.

By using the AI tool Claude Fable 5 in the summer 2026, the authors realized that our approximate developing map (cf. page 100 in \cite{dinh-thesis}) is M\"obius equivalent to the Joukowsky function $J(w) = w + w^{-1}$ 
and, after a blow up, is the limit on the exceptional divisor $\mathbb{E} \cong \mathbb{P} ^1$ of the family of developing maps (cf. \Cref{prop-family-developingmaps-1}).
The authors then used AI assistance in arriving at the proof of \Cref{thm-bubble} as presented in Sections \ref{sect-family-blown-up} -- \ref{sect-existence-bubble} and in finalizing the paper, in particular making the figures. 

We have checked all statements and proofs of the paper and take full responsibility for them.

\section{Preliminaries}\label{sect-preliminaries}
Let $C$ be a compact Riemann surface of genus $g \geq 2$, and $K$ denotes its canonical line bundle.
In this section, we recall the Bia\l ynicki-Birula stratification of the Hodge moduli space of $\lambda$-connections \cite{simpson-hodge-filtration, simpson-iterated}, following \cite{collier-wentworth, HH21}.
This Hitchin moduli space of Higgs bundles and the de Rham moduli space of flat connections are respectively the fibers $\lambda = 0$ and $\lambda = 1$ of the Hodge moduli space. 
Restricting the Bia\l ynicki-Birula stratification to these fibers then define the Bia\l ynicki-Birula stratification of the Hitchin moduli space and Simpson partial oper stratification of the de Rham moduli space.

Throughout this paper, we will consider rank-2 Higgs bundles and flat bundles with fixed determinants $\Lambda$.
Our convention is that, whenever we consider these as special cases of $\lambda$-connections, we let $\Lambda = \mathcal{O}_C$.
On the other hand, when we consider holomorphic bundles and Higgs bundles in and of themselves as in \Cref{subsec-BA-divisors}, \Cref{sect-degeneration-BA} and \Cref{subsec-moduli-bundles}, we can let $\Lambda$ be a fixed line bundle of degree $\deg(\Lambda) \in \{0, 1\}$.

\subsection{$\lambda$-connections and the Bia\l ynicki-Birula stratification}\label{subsubsec-natural-stratification}
\label{subsec-moduli-Higgs}
\paragraph{Moduli spaces}
Throughout our paper, a $\lambda$-connection is a triple $(\lambda, E, \nabla_\lambda)$, 
where $\lambda \in \C$, $E$ is a holomorphic rank-$2$ bundle with $\det(E) \cong \mathcal{O}_C$, 
and $\nabla_\lambda : E \rightarrow E \otimes K$ satisfies the twisted Leibniz rule 
\[ \nabla_\lambda(f s) = \lambda \; \partial f \otimes s + f \nabla_\lambda (s) \]
for local holomorphic functions $f$ and local sections $s$ of $E$, and induces a fixed $\lambda$-connection on $\det(E)$.
A $\lambda$-connection with $\lambda = 0$ is an $\SL_2(\C)$-Higgs bundle $(E, \phi)$ with $\phi \in H^0(\End_0(E) \otimes K)$, a traceless endomorphism of $E$ twisted by holomorphic differentials.
A $\lambda$-connection with $\lambda = 1$ is an $\SL_2(\C)$-holomorphic connection $(F, \nabla)$.

We denote by $\mathcal{M}_H(\Lambda)$ the moduli space of stable $\mathrm{SL}(2, \mathbb{C})$-Higgs bundles $(E, \phi)$ with $\det(E) \cong \Lambda$, 
and by $\mathcal{M}_{dR}$ the moduli space of irreducible $\mathrm{SL}(2, \mathbb{C})$-holomorphic connections $(E, \nabla)$ with $\det(E) \cong \mathcal{O}_C$.
Denote by $\mathcal{M}_{Hod}$ the moduli space of $\lambda$-connections $(\lambda, E, \nabla_\lambda)$ with $\det(E) \cong \mathcal{O}_C$.
The map $\underline{\lambda} : \mathcal{M}_{Hod} \rightarrow \C$ has fibers $\underline{\lambda}^{-1}(0) = \mathcal{M}_H(\mathcal{O}_C)$ and $\underline{\lambda}^{-1}(1) = \mathcal{M}_{dR}$.
Without loss of generality, we let $\deg(\Lambda) \in \{0, 1\}$ and write $\mathcal{M}_H \equiv \mathcal{M}_H(\Lambda)$ when the choice of $\Lambda$ does not matter. 

\paragraph{The $\C^\ast$-action and Hodge bundles}
The Hodge moduli space $\mathcal{M}_{Hod}$ admits a natural $\mathbb{C}^\ast$-action via 
\[
	t . [\lambda, E, \nabla_\lambda] \coloneqq [t\lambda, E, t\nabla_\lambda], \qquad t \in \C^\ast.
\]
On $\underline{\lambda}^{-1}(0) = \mathcal{M}_H(\mathcal{O}_C)$ and, in general, on $\mathcal{M}_H(\Lambda)$, 
the $\mathbb{C}^\ast$-action is $t.[E, \phi] = [E, t\phi]$.

Recall the Hitchin map 
\[ h: \mathcal{M}_H(\Lambda) \longrightarrow H^0(C, K^2), \qquad \qquad [E, \phi] \longmapsto \det(\phi),  \]
is a proper completely integrable system.
We say that a fixed point $\alpha \in \mathcal{M}_H^{\mathbb{C}^\ast}$, which is contained in the nilpotent cone $h^{-1}(0)$, is a Hodge bundle.
Given $x \in \mathcal{M}_{Hod}$, the limit $\underset{t \rightarrow}{\lim} \, t.x$ exists and is a Hodge bundle.
With $\mathcal{N}$ denoting the moduli space of rank-2 stable bundles with determinant $\Lambda$, one can write 
\begin{equation}\label{fixed-points}
	\mathcal{M}_H^{\C^\ast} = \mathcal{N} \sqcup \left( \sqcup_{d} N_d \right).
\end{equation}
Here, $\mathcal{N}$ parametrizes Higgs bundles of the form $(E, 0)$ with $E$ stable and of determinant $\Lambda$, and $N_d$ parametrizes Higgs bundles $(E_M, \phi_D)$ of the form 
\begin{align}\label{C*-fixed}
	&E_M = M \oplus M^{-1} \Lambda,
	& &\phi_{D} = \begin{pmatrix} 0 & 0 \\ s_D & 0 \end{pmatrix},
	& &s_D \in H^0(K M^{-2} \Lambda) \cong \mathcal{O}_C(D),
	& &D \in C^{(d)}.
\end{align}
Note that $d \equiv \deg(\Lambda)$ and that stability of $(E_M, \phi_D)$ requires $d < 2g-2$.

\paragraph{Upward flows and Bia\l ynicki-Birula stratification}
For a Hodge bundle $\alpha$, let
\[ \alpha^+ \coloneqq \left\{ x \in \mathcal{M}_{Hod} \, \big| \, \underset{t \rightarrow 0}{\lim} t.x = \alpha \right\} \]
--the upward flow of $\alpha$ in $\mathcal{M}_{Hod}$, and
\[ \alpha^+_\lambda \coloneqq \alpha^+  \cap \underline{\lambda}^{-1}(\lambda). \]
For a component $X \in \pi_0(\mathcal{M}_H^{\mathbb{C}^\ast})$, let
\[ X^+ \coloneqq \left\{ x \in \mathcal{M}_{Hod} \big| \underset{t \rightarrow 0}{\lim} t.x \in X \right\} \]
--the attractor of $X$ in $\mathcal{M}_{Hod}$, and
\[ X^+_\lambda \coloneqq X^+ \cap \underline{\lambda}^{-1}(\lambda). \]
We use the same notations $\alpha^+_0$ and $X^+_0$ respectively for the upward flow of $\alpha$ and attractor of $X$ in $\mathcal{M}_H(\Lambda)$ for general $\Lambda$. 
In particular, for $\alpha \in N_0$ which consists of $2^{2g}$ points, 
$\alpha^+_0$ and $\alpha^+_1$ correspond to a Hitchin section and space of holomorphic opers respectively.

\begin{theorem}[\cite{simpson-iterated, collier-wentworth}]\label{thm-CW}
	Let $\alpha$ be a stable Hodge bundle, contained in $X \in \pi_0(\mathcal{M}_H^{\mathbb{C}^\ast})$. 
	Then $\alpha^+_0 \subset \mathcal{M}_H$ and $\alpha^+_1 \subset \mathcal{M}_{dR}$ are holomorphic Lagrangian submanifolds, isomorphic as varieties with $\mathbb{C}^\ast$-action to the positive weighted subspace of $T_\alpha \mathcal{M}_H$.
	In particular, for $G = \mathrm{SL}(2, \mathbb{C})$, $\alpha^+_0$ and $\alpha^+_1$ are isomorphic to $\C^{3g-3}$. 
\end{theorem}

The decompositions  
\[ \mathcal{M}_{Hod} = \mathcal{N}^+ \sqcup \left( \sqcup_d N_d^+ \right),
\qquad \qquad \mathcal{M}_{H} = \mathcal{N}^+_0 \sqcup \left( \sqcup_d (N_d)^+_0 \right) \] 
define the Bia\l ynicki-Birula stratification on the Hodge moduli spaces and Hitchin moduli spaces, respectively.
The intersection of the former with $\mathcal{M}_{dR} \equiv \underline{\lambda}^{-1}(1)$ defines the Simpson partial oper stratification
\begin{equation}\label{BB-stratifications}
	\mathcal{M}_{dR} = \mathcal{N}^+_1 \sqcup \left( \sqcup_d (N_d)^+_1 \right).
\end{equation}

In all of these cases, the largest dimensional open stratum
\[ \mathcal{N}^+, \qquad \qquad \mathcal{N}^+_0, \qquad \qquad \mathcal{N}^+_1, \]
whose closures are the whole respective moduli spaces, parametrize $\lambda$-connections with stable underlying bundles.
The fibrations 
\[ \mathcal{N}^+_0 \longrightarrow \mathcal{N}, 
\qquad \qquad \mathcal{N}^+_1 \longrightarrow \mathcal{N} \]
are respectively the cotangent bundle $T^\ast \mathcal{N}$ and the twisted cotangent bundle of $\mathcal{N}$, i.e.
for $\alpha = E \in \mathcal{N}$,
\[ \alpha^+_0 = \{ (E, \phi) \in \mathcal{M}_H \}, 
\qquad \qquad \alpha^+_1 = \{ (E, \nabla) \in \mathcal{M}_{dR} \}. \]

\paragraph{Divisors from triples $(L \hookrightarrow E, \phi)$ and $(L \hookrightarrow E, \nabla)$}
Consider a rank-2 holomorphic bundle $E$ together with a line subbundle $L \hookrightarrow E$, i.e. an extension
\begin{equation} \label{bundle-extension}
	0 \rightarrow L \rightarrow E \rightarrow L^{-1}\Lambda \rightarrow 0.
\end{equation}
Given a Higgs field $\phi$ or a holomorphic connection $\nabla$ on $E$, consider the morphism
\begin{subequations}\label{eqn-section}
\begin{align}
	s_{L}( \phi): L \hookrightarrow E \overset{\phi}{\rightarrow} E \otimes K \rightarrow L^{-1} \Lambda K, \label{lower-left-component} \\
	s_{L}(\nabla): L \hookrightarrow E \overset{\nabla}{\rightarrow} E \otimes K \rightarrow L^{-1} K. \label{lower-left-dR-0}
\end{align}
\end{subequations}
One can check that the latter is also a $\mathcal{O}_C$-linear morphism,
and so these morphisms correspond to sections of $K L^{-2} \Lambda$ and $K L^{-2}$ respectively. 
Concretely, in local frames adapted to $L \hookrightarrow E$, one can write
\[ \phi = \begin{pmatrix}
 \ast & \ast \\ s_L(\phi) & \ast	
\end{pmatrix},
\qquad \qquad \qquad
\nabla = \partial + \begin{pmatrix}
	\ast & \ast \\ s_L(\nabla) & \ast
\end{pmatrix}.
\]
It is shown in \cite{DT23, Dinh24} that the extension class $[E]_L \in H^1(C, L^2 \Lambda^{-1})$ defined by \eqref{bundle-extension} pairs with these sections via Serre duality to
\[ \langle [E]_L, s_L(\phi) \rangle = 0, 
\qquad \qquad \langle [E]_L, s_L(\nabla) \rangle = \deg(L).  \]

Note that if $\det(\phi)$ is a quadratic differential with simple zeroes, i.e. the corresponding Hitchin fiber is smooth, (respectively, $\nabla$ is irreducible), 
then $s_L(\phi)$ (respectively, $s_L(\nabla)$) is not identically vanishing.
In this case, we denote by $D_L(\phi)$ and $D_L(\nabla)$ the corresponding zero divisors of these sections. 
The embedding of the subbundle $L$ into $E$ is unique up to scaling, which rescales $s_L$, so these zero divisors depend only on $L$.
In \Cref{subsec-BA-divisors} we also consider injections $L \overset{i}{\rightarrow} E$ that may vanish at some points, and then write $s_i$ and $D_i$.

\paragraph{Destabilizing subbundles.}
Let $E$ be strictly unstable, i.e. $E$ admits a unique destabilizing line subbundle $L_E \hookrightarrow E$ with $\deg(L_E^{-2}\Lambda) < 0$.
It follows that a Higgs bundle $(E, \phi)$ is stable if and only if $s_{L_E}(\phi) \neq 0$.

\begin{proposition}\label{prop-rank2-limits}
	Let $E$ be strictly unstable with destabilizing line subbundle $L_E \hookrightarrow E$.
	\begin{enumerate}[(i)]
		\item Let $[E, \phi] \in \mathcal{M}_H(\Lambda)$ and $\alpha = [E_M, \phi_D] \in N_d$ for some $d < 2g-2$.
		Then $[E, \phi] \in \alpha^+_0$ 
		if and only if $M$ is the unique destabilizing line subbundle of $E$ and 
		\[ D_{M}(\phi) = D. \]

		\item Let $[E, \nabla] \in \mathcal{M}_{dR}$ and $\alpha = [E_M, \phi_D] \in N_d$ for some $d < 2g-2$.
		Then $[E, \nabla] \in \alpha^+_1$ 
		if and only if $M$ is the unique destabilizing line subbundle of $E$ and 
		\[ D_{M}(\nabla) = D. \]
	\end{enumerate}
\end{proposition}

For rank $2$, Simpson proves that $(N_d)^+_1 \subset \overline{(N_{d+2})^+_1}$ for $d \leq 2g-6$ \cite[Theorem~7.6]{simpson-iterated}.
An explicit description of this closure relation and the relation to bubbling surgery is the subject of our paper. 

\paragraph{Very stable and wobbly Hodge bundle}
We say that a stable Hodge bundle $\alpha \in X \in \pi_0(\mathcal{M}_H^{\mathbb{C}^\ast})$ is \textit{very stable} if its upward flow $\alpha^+$ intersects the nilpotent cone $h^{-1}(0)$ exactly at $\alpha$; otherwise, we say that $\alpha$ is \textit{wobbly}.
These terminology and notion were first introduced by Drinfeld--Laumon \cite{Lau88} and Donagi--Pantev \cite{DP09} respectively for the case $X$ is the moduli space of stable bundles--a bundle is called \textit{wobbly} if it admits nonzero nilpotent Higgs fields and \textit{very stable} otherwise--and generalized to other components of Hodge bundles by Hausel--Hitchin \cite{HH21}. 

\begin{theorem} Let $\alpha$ be a stable Hodge bundle.
\begin{enumerate}[(i)]
		\item The Lagrangian leaf $\alpha^+_0$ is closed in $\mathcal{M}_H$ if and only if $\alpha$ is very stable \cite{PPe19, HH21}. 
		For $\alpha = [E_M, \phi_D] \in N_d$, this holds if and only if $D$ is reduced \cite{HH21}.
		\item The Lagrangian leaf $\alpha^+_1$ is closed in $\mathcal{M}_{dR}$ \cite{DS24}, but the leaves do not form a $C^1$ foliation \cite{Huang26}. 
\end{enumerate}
\end{theorem}
We remark that part (ii) answers the two parts of Simpson's question \cite[Question~7.4]{simpson-iterated}.
A similar description for $\alpha \in \mathcal{N}$ in terms of divisors of quadratic differentials is given in \cite{Dinh24wobbly}.

\subsection{Moduli spaces of rank-2 bundles}
\label{subsec-moduli-bundles}
Given a rank-2 holomorphic bundle $E$,
we say a subbundle $M$ of $E$ is a maximal subbundle of $E$ if $\deg(M) \geq \deg(L)$ for all other subbundles $L$ of $E$.
Clearly $E$ is stable if its Segre invariant 
\begin{equation*}
	s(E) \coloneqq \deg(\Lambda M^{-2}) 
	= \underset{L \hookrightarrow E}{\min} \; \deg(\Lambda L^{-2}),
\end{equation*} 
where $M$ is a maximal subbundle of $E$, is positive.
It follows by a theorem of Nagata \cite{Nagata} that $s(E) \leq g$.
As $s(E) \equiv \deg(\Lambda)$ mod $2$, we have 
\begin{equation}
	\max \: s(E) = 
	\begin{cases}
		g  & \text{for } g \equiv \deg(\Lambda) \pmod 2 \\
		g-1 & \text{for } g \equiv \deg(\Lambda)-1 \pmod 2 
	\end{cases}.
\end{equation}

The moduli space $\mathcal{N} \equiv \mathcal{N}_\Lambda$ of stable bundles with fixed determinant $\Lambda$ can be decomposed into 
\[ \mathcal{N} = \underset{s}{\sqcup} \, \mathcal{N}_s, 
\qquad \qquad \qquad 1 \leq s \leq g, \quad s \equiv \deg(\Lambda), \]
where $\mathcal{N}_s$ is the open Segre stratum defined by 
$\mathcal{N}_s = \{ [E] \in \mathcal{N} \mid s(E) = s \}$.
Then $\mathcal{N}_s$ is an irreducible algebraic variety of dimension $2g + s -2$.
These define a nested stratification called Segre stratification \cite{Lange-Narasimhan}.

It is known that if $s(E) = g$, then the set of its maximal subbundles is of dimension $1$ \cite{Lange-Narasimhan}.
On the other hand, a generic bundle $E$ with $s(E) = g-1$ only has a finite number of maximal subbundles,
and a generic bundle $E$ with $s(E) < g-1$ has only one maximal subbundle.
If however $E$ is strictly unstable, i.e. $s(E) < 0$,
then $E$ has a unique maximal subbundle $L_E$.
Regardless of the value of $s(E)$, 
if $M$ is a maximal subbundle of $E$ then its embedding into $E$ is unique up to a scaling, i.e. $h^0(M^{-1} E) = 1$.

We note that the Segre invariants also label the components of Hodge bundles with underlying unstable bundles. 
The Segre invariants therefore induce a finer stratification on $\mathcal{M}_H$ and $\mathcal{M}_{dR}$ by respectively refining the open stratum $\mathcal{N}^+_0$ and $\mathcal{N}^+_1$.
{%
\paragraph{Closed strata.}\label{par-closed-strata}
For $[E, \nabla] \in \mathcal{M}_{dR}$ we have $\Lambda = \mathcal{O}_C$, so $2g-2+s(E) = \min_{L \hookrightarrow E} \deg(KL^{-2})$, and the refined stratification of $\mathcal{M}_{dR}$ is labelled by this number $d$:
\begin{itemize}
	\item if $d < 2g-2$, i.e.\ $s(E) < 0$, then $E$ is strictly unstable and $[E, \nabla] \in (N_d)^+_1$ by \Cref{prop-rank2-limits}(ii);
	\item if $d = 2g-2$, then $E$ is strictly semistable;
	\item if $d > 2g-2$, then $E$ is stable and $[E, \nabla]$ lies in the Segre stratum $(\mathcal{N}_s)^+_1 \coloneqq \{ [E, \nabla] \in \mathcal{N}^+_1 \mid E \in \mathcal{N}_s \}$ with $s = d - (2g-2)$.
\end{itemize}
For an even integer $d \geq 0$, we call
\begin{align}
	\mathcal{M}^{\leq d}_{dR} 
	&\coloneqq \left\{ [E, \nabla] \in \mathcal{M}_{dR} \;\middle|\; E \text{ admits a line subbundle } L \text{ with } \deg(KL^{-2}) \leq d \right\} \nonumber \\
	&= \left\{ [E, \nabla] \in \mathcal{M}_{dR} \;\middle|\; s(E) \leq d - (2g-2) \right\} \label{eqn-closed-strata}
\end{align}
the closed stratum labelled by $d$.
It is closed, since the Segre invariant is lower semicontinuous in families, and $\mathcal{M}^{\leq d-2}_{dR} \subset \mathcal{M}^{\leq d}_{dR}$.
For $d < 2g-2$, it is the union $\bigsqcup_{d' \leq d} (N_{d'})^+_1$ of partial oper strata, which equals $\overline{(N_d)^+_1}$ by \cite[Theorem~7.6]{simpson-iterated}.
For $d \geq 2g-2$, it consists of all partial oper strata $(N_{d'})^+_1$ together with the connections on semistable bundles $E$ with $s(E) \leq d-(2g-2)$.
For $d > 2g-2$, its intersection with the open stratum $\mathcal{N}^+_1$ is the closure of $(\mathcal{N}_s)^+_1$, $s = d-(2g-2)$, in $\mathcal{N}^+_1$, by the nestedness of the Segre stratification \cite{Lange-Narasimhan}, since $\mathcal{N}^+_1 \rightarrow \mathcal{N}$ is an affine bundle.
}%

\subsection{Spectral correspondence and integrable structure}\label{subsec-spectral}
The Hitchin fibration $h: \mathcal{M}_H \rightarrow H^0(K^2)$ associates to the isomorphism class of a Higgs bundle $(E, \phi)$ the quadratic differential $q = \det(\phi)$.
A generic Hitchin fiber is isomorphic to an abelian variety, namely the Prym variety of the associated spectral curve.
This endows $\mathcal{M}_H$ with the structure of an algebraic integrable system \cite{Hit87b}.

\paragraph{Spectral curves.}
We recall that associated to a quadratic differential $q = \det(\phi)$
is a spectral curve $\widetilde{C}_{q}$ embedded in the total space of $T^\ast C$.
The spectral curve encodes the eigen-values of the Higgs field:
concretely, if $z$ is the coordinate of an open set $U \subset C$,
$v$ the fiber coordinate of the restriction of $T^\ast C$ to $U$
and $\phi (z) = \begin{pmatrix} a(z) & b(z) \\ c(z) & - a(z) \end{pmatrix}$ locally,
then locally $\widetilde{C}_q$ is defined by
\begin{align}
	v^2 + q(z) = v^2  - a(z)^2 - b(z) c(z) = 0.
	\label{spectral-curve-eqn}
\end{align}
Here $v$ is the tautological section of $T^\ast C$, and $v^2 + q$ is the characteristic polynomial of $\phi$.
The morphism $\widetilde{C}_q \overset{\pi}{\rightarrow} C$ induced by $T^\ast C \rightarrow C$ is a $2:1$ covering that branches at the zeroes of $\det(\phi)$.
The involution $\sigma$ of $\widetilde{C}_q$ interchanges points
corresponding to the eigenvalues $v = \pm (-q(z))^{1/2}$ of $\phi(z)$.
We say a quadratic differential $q$ and its associated spectral curve $\widetilde{C}_q$ are non-degenerate if the zeroes of $q$ are all simple.
In this case, $\widetilde{C}_q$ is a smooth compact Riemann surface of genus $\tilde{g} = 4g-3$,
and in particular $\pi^\ast(K)$ has a canonical section defined by $v$ that vanishes precisely at the ramification divisor $R_q$ of $\widetilde{C}_q$.
A spectral curve $\widetilde{C}_q$ is called degenerate when $q$ has zeroes of non-trivial multiplicity;
the most degenerate case is $q=0$, and we call the fiber $h^{-1}(0)$ the nilpotent cone.

\paragraph{Spectral line bundles.}
Let $(E, \phi) \in \mathcal{M}_H(\Lambda)$ be a Higgs bundle with
non-degenerate spectral curve $\widetilde{C}_q \overset{\pi}{\rightarrow} C$.
The cokernel of the map
$ \pi^\ast(E K^{-1}) \overset{\pi^\ast(\phi) - v}{\longrightarrow} \pi^\ast(E) $
is a line bundle $\mathcal{L}_{(E,\phi)}$ on $\widetilde{C}_q$
which we will call the spectral line bundle of $(E,\phi)$.
{%
At a point $\tilde{p} = (z, v)$ of $\widetilde{C}_q$ which is not a ramification point, the image of $\pi^\ast(\phi) - v$ is the eigen-line of $\phi(z)$ with eigen-value $-v$, so the fiber of $\mathcal{L}_{(E,\phi)}$ at $\tilde{p}$ is the quotient of $E_z$ by this eigen-line.
}%
The spectral correspondence \cite{Hit87a, BNR}
is a correspondence between Higgs bundles and their spectral line bundles,
namely that the underlying bundle can be recovered by the direct image
$$ E \simeq \pi_\ast \mathcal{L}_{(E, \phi)},
\qquad \qquad \deg \left( \mathcal{L}_{(E,\phi)} \right) = 2g-2 +\deg(\Lambda), \qquad \qquad \phi = \pi_\ast v. $$
It follows that the fiber $h^{-1}(q)$ of the Hitchin fibration
$h: \mathcal{M}_H(\Lambda) \rightarrow H^0(C, K^2)$ over a non-degenerate $q$
is canonically isomorphic to
\begin{equation}\label{L-L-sigma}
	\mathrm{Prym}^{2g - 2 +\deg(\Lambda)}(\widetilde{C}_q, \Lambda) \coloneqq \left\{ \mathcal{L} \in \mathrm{Pic}^{2g - 2 +\deg(\Lambda)}(\widetilde{C}_q)
	\ \middle| \ \mathcal{L} \otimes \sigma^\ast\mathcal{L} \simeq \pi^\ast(K \Lambda) \right\},
\end{equation}
which is a torsor over the Prym variety of $\widetilde{C}_q \overset{\pi}{\rightarrow} C$.

\section{Baker--Akhiezer divisors of Higgs bundles and apparent singularities of projective connections}\label{sect-BA-apparent}

\subsection{Baker--Akhiezer divisors}\label{subsec-BA-divisors}
In this subsection we recall the notion of Baker--Akhiezer divisors on non-degenerate spectral curves.

The input data to define these divisors
consist of a Higgs bundle $(E,\phi)$ with non-degenerate associated spectral curve and an injection of a line bundle $L$ to $E$.
While the terminology for these divisors is recent \cite{DT23} and inspired by the literature on integrable systems,
the usage of these divisors is not entirely new.
Hitchin in his original paper \cite{Hit87a} already characterized Higgs bundles with underlying unstable bundles in terms of these divisors,
and the more recent work \cite{HH21} of Hausel-Hitchin on very stable Higgs bundles $\alpha$ and Fourier--Mukai transforms of $\alpha^+$
also made extensive use of related divisors.

\subsubsection{Definitions and basic properties}
Let $(E, \phi) \in \mathcal{M}_H(\Lambda)$ be a Higgs bundle with non-degenerate spectral curve $\widetilde{C} \overset{\pi}{\rightarrow} C$.
Consider an injection $L \overset{i}{\rightarrow} E$ from a line bundle $L$ into $E$,
which a priori is not necessarily an embedding and can vanish at some point on $C$.
We will use the notation $L \overset{i}{\hookrightarrow} E$ 
for an injection which has no zero, i.e. an embedding that makes $L$ into a subbundle of $E$.

\begin{definition}\label{def-BA-divisor-formal}
	The \textit{Baker--Akhiezer divisor}%
	\footnote{In local frames adapted to $\pi^\ast L$, the line bundle $\pi^\ast(K^{-1}) \mathcal{L}_{(E,\phi)}$ of eigenvectors of $\pi^\ast(\phi)$ has sections of the form $\left(\begin{smallmatrix} v + a(z) \\ c(z) \end{smallmatrix} \right)$,
	reminiscent of the Baker--Akhiezer function in the theory of integrable systems \cite{DT23}, \cite[Remark~2.3]{Dinh24wobbly}.}
	$\widetilde{D}_i(\phi)$ of the data $(L \overset{i}{\rightarrow} E, \phi)$ is the zero divisor of the morphism of line bundles on $\widetilde{C}$
	\begin{equation} \label{BA-div-def-composition}
		\pi^\ast(L) \overset{\pi^\ast i}{\longrightarrow} \pi^\ast(E) \longrightarrow \mathcal{L}_{(E, \phi)}.
	\end{equation}
\end{definition}
We note that this definition of Baker--Akhiezer divisors carries straightforwardly to higher rank cases \cite{HH21, DDX26}.

If $L \overset{i}{\rightarrow} E$ vanishes at $p \in C$, then $\pi^\ast(p) < \widetilde{D}_i(\phi)$ (this holds even if $p$ is a branch point).
In case $i$ is an injection, similar to \eqref{eqn-section} we can consider the composition
\begin{equation*}
	s_i(\phi): L \overset{i}{\longrightarrow} E \overset{\phi}{\longrightarrow} E \otimes K \longrightarrow L^{-1} \Lambda K,
\end{equation*}
where the last morphism is induced from the quotient of $i$; 
if $i$ is an embedding this coincides with \eqref{lower-left-component}.
We regard $s_i(\phi)$ as a section of $K L^{-2} \Lambda$ and denote by $D_i(\phi)$ its zero divisor.

\begin{proposition}\cite{DT23}
	\label{prop-basic-properties-BA-divisor}
	Let $(E, \phi)$ be a Higgs bundle with smooth spectral curve and 
	$L \overset{i}{\rightarrow} E$ an injection. 
	Then $\pi^\ast(p) < \widetilde{D}_i(\phi)$ if and only if $i$ vanishes at $p$ and 
	\begin{equation*}
		D_i(\phi) = \mathrm{Nm}(\widetilde{D}_i(\phi)).
	\end{equation*}
\end{proposition}

\paragraph{Inverse construction}
The construction of Baker--Akhiezer divisors can be inverted: given the data $\left(q, \widetilde{D} \right)$
where $q$ is a quadratic differential with simple zeroes and $\widetilde{D}$ an effective divisor on the spectral curve $\widetilde{C}_q \overset{\pi}{\rightarrow} C$,
we can construct the data $(L \rightarrow E, \phi)$ that defines $\widetilde{D}$ as its Baker--Akhiezer divisor.
The solutions to this inverse problem are not unique, since twisting one solution with a line bundle defines another.
Solutions that define Higgs bundle in the same moduli space $\mathcal{M}_H( \Lambda)$ differ exactly by twisting by a square-root of the trivial line bundle.

\begin{proposition}\cite{DT23}
	\label{prop-inverse-BA}
	Let $\widetilde{D}$ be an effective divisor on a non-degenerate spectral curve 
	$\widetilde{C}_q \overset{\pi}{\rightarrow} C$.
	Let $L$ and $\Lambda$ be line bundles on $C$ satisfying 
	$\mathcal{O}_C(D) \simeq K L^{-2} \Lambda$ 
	where $D = \mathrm{Nm}(\widetilde{D})$.
	Then there exists a triple $(L \overset{i}{\rightarrow} E, \phi)$ 
	unique up to scaling of $i$ such that $\widetilde{D}_i(\phi) = \widetilde{D}$
	and $\det(E) \cong \Lambda$.
	In particular, $i$ is an embedding if and only if 
	$\widetilde{D}$ contains no summand of the form $\pi^\ast(p)$ for $p \in C$.
\end{proposition}

\paragraph{Concrete formulation in terms of the matrix form of Higgs fields.}
Suppose that $L \overset{i}{\hookrightarrow} E$ is an embedding
and that in the local frames adapted to $i$ we have
$\phi =
\left( \begin{smallmatrix} a(z) & b(z) \\ c(z) & -a(z) \end{smallmatrix} \right)$,
where $c(z)$ is the local expression of $s_i(\phi)$ defined in \eqref{lower-left-component}.
Note that $s_i(\phi) \neq 0$: otherwise $q = -a^2$ and the zeroes of $q$ would have multiplicity.
At a point $\tilde{p} \in \widetilde{C}$, the image of $\pi^\ast \phi - v$
coincides with $\pi^\ast L\mid_{\tilde{p}}$
if and only if $c(p) = 0$ for $p = \pi(\tilde{p})$
and $\tilde{p}$ is defined by the eigenvalue $v = -a(z(p))$.
Concretely, let $D_i(\phi) = \sum_{k=1}^d x_k$ be the zero divisor of $s_i(\phi)$, counted with multiplicity.
At each $x_k$, equation \eqref{spectral-curve-eqn} reduces to $v^2 - a(z(x_k))^2 = 0$.
If $x_k$ is not a branch point, then the two points in $\pi^{-1}(x_k)$ are unambiguously labeled by $v = \pm a(z(x_k))$;
in this case let $\tilde{x}_k$ be the point defined by $v = -a(z(x_k))$.
If $x_k$ is a branch point then let $\tilde{x}_k$ be the  ramification point $\pi^{-1}(x_k)$.
Then
\begin{equation}
	\widetilde{D}_i(\phi) = \sum_{k=1}^d \tilde{x}_k.
	\label{BA_divisor_explicit_def}
\end{equation}
In particular, the sheet of $\widetilde{C}$ on which a point of $\widetilde{D}_i(\phi)$ lies is determined by the diagonal entry $a$ of $\phi$ at the corresponding zero of $s_i(\phi)$.

\subsubsection{Wobbly bundles and $Q$-special divisors}\label{subsubsec-wobbly-Q-special}
Following Laumon \cite{Lau88}, a vector bundle $E$ is \textit{very stable} if it does not admit nonzero nilpotent Higgs fields;
we say that $E$ is \textit{wobbly} if it is semistable but not very stable.
{%
For an effective divisor $\mathbf{x}$ on $C$, let $Q_{\mathbf{x}} \subset H^0(K^2)$ be the space of quadratic differentials vanishing on $\mathbf{x}$, as in \Cref{sect-degeneration-apparent-singularities}.
}%
{%
For unstable underlying bundles, Hausel--Hitchin characterize very stability of the limiting Hodge bundle by reducedness of the Baker--Akhiezer divisor of the destabilizing subbundle \cite{HH21} (cf. \Cref{subsubsec-natural-stratification}).
}%
It turns out that the condition $Q_{\mathrm{Nm}(\widetilde{D})} \neq 0$ is the proper analogue of the condition on $\widetilde{D}$ to be non-reduced when the underlying bundle is semistable.

\begin{theorem}[{\cite[Theorem~1.2]{Dinh24wobbly}}]\label{thm-wobbly-Q-special}
	Let $\widetilde{C}_q \overset{\pi}{\rightarrow} C$ be a non-degenerate spectral curve
	and $\mathcal{L}$ a line bundle on $\widetilde{C}_q$.
	Assume that $E = \pi_\ast \mathcal{L}$ is semistable, and let $\det(E) = \Lambda$.
	Then $E$ is wobbly if and only if
	\begin{equation*}
		\mathcal{L} \simeq \pi^\ast(L) \otimes \mathcal{O}_{\widetilde{C}_q}(\widetilde{D})
	\end{equation*}
	where $L$ is a line bundle on $C$ and $\widetilde{D}$ is an effective divisor on $\widetilde{C}_q$ with $Q_{\mathrm{Nm}(\widetilde{D})} \neq 0$.
	Furthermore, if $\widetilde{D}$ has no summand of the form $\pi^\ast(P)$
	for some effective divisor $P$ on $C$, then
	$E$ admits an embedding $L \overset{i}{\hookrightarrow} E$
	and there is a canonical bijection between the corresponding space of nilpotent Higgs fields
	\begin{equation*}
		H^0\left( C, K L^2 \Lambda^{-1} \right) \simeq
		\left\{ \phi \in H^0(C, \mathrm{End}(E)\otimes K) \text{ nilpotent},
		\ker(\phi) = i(L) \right\}
	\end{equation*}
	and $Q_{\mathrm{Nm}(\widetilde{D})}$.
\end{theorem}

{%
By \Cref{prop-inverse-BA} and \Cref{prop-basic-properties-BA-divisor}, $\widetilde{D}$ in \Cref{thm-wobbly-Q-special} is the Baker--Akhiezer divisor of $(L \overset{i}{\rightarrow} E, \pi_\ast v)$, and $\mathrm{Nm}(\widetilde{D}) = D_i(\pi_\ast v)$.
For divisors of degree $3g-3$, the condition $Q_{\mathbf{x}} \neq 0$ is precisely $Q$-specialness in the sense of \Cref{sect-degeneration-apparent-singularities}, since the expected dimension of $Q_{\mathbf{x}}$ is then $0$; this is the condition excluded for the collision site $2x_0 + \mathbf{x'}$ by \eqref{condition-neighborhood} when $d = 3g-3$.
}%

\subsection{Holomorphic connections, projective connections, projective structures and opers}\label{subsec-projective-structures}
\subsubsection{Projective connections and projective structures}\label{subsubsec-projective-structures}
In the following, we recall the notions of projective connections and projective structures, 
two geometric objects that naturally realize monodromy representations in $\mathrm{PSL}_2(\C) \cong \mathrm{PGL}_2(\C)$.

\begin{definition}
	A projective connection on $C$ is a pair $(P, s)$ where
	$P$ is a flat $\mathrm{PSL}_2(\C)$-bundle on $C$ with $\mathbb{P}^1$-fibers
	and $s: C \rightarrow P$ is a global holomorphic section which is not parallel w.r.t. the flat structure of $P$.
	The points where $ds = 0$ are called apparent singularities;
	the order of an apparent singularity is the order of the zero of $ds$ at that point.
\end{definition}

We denote by $\Div((P,s))$ the divisor of apparent singularities, counted with multiplicity,
of a projective connection $(P,s)$.
Two projective connections $(P_1, s_1)$ and $(P_2, s_2)$ are isomorphic
if there exists an isomorphism $P_1 \rightarrow P_2$ of holomorphic fiber bundles
that commutes with $s_1$ and $s_2$.
Clearly if $(P_1, s_1)$ is isomorphic to $(P_2, s_2)$, then $\Div((P_1, s_1)) = \Div((P_2, s_2))$.

\begin{definition}
	A projective structure is a maximal atlas $\left\{ U_\alpha \right\}_{\alpha \in \mathcal{I}}$ of $C$
	together with local holomorphic functions $\{ w_\alpha: U_\alpha \rightarrow \C\}$
	the values of which are related by Möbius transformations, 
	i.e. for all $x\in U_\alpha \cap U_\beta$,
	\begin{align}\label{PSL_2-action-projective-structure}
		&w_\beta (x) = \frac{a_{\beta \alpha} w_\alpha (x) + b_{\beta \alpha}}{c_{\beta \alpha} w_\alpha (x) + d_{\beta \alpha}}, 
		&\begin{pmatrix} a_{\beta \alpha} & b_{\beta \alpha} \\ c_{\beta \alpha} & d_{\beta \alpha} \end{pmatrix} \in \SL_2(\C).
	\end{align}
	Apparent singularities are where $d w_\alpha = 0$;
	the order of an apparent singularity is the order of the zero of $d w_\alpha$ at that point.
\end{definition}
Note that $w_\alpha$ can serve as a local coordinate on $U_\alpha$ if and only if $U_\alpha$ contains no apparent singularity.
{%
In the language of opers, a projective structure in this sense, i.e.\ possibly with apparent singularities, is a \emph{$\mathrm{PSL}_2(\C)$-oper with apparent singularities} \cite{Dinh24}; it is a holomorphic oper, or a \emph{holomorphic projective structure}, when it has no apparent singularity.
Following \cite{CalsamigliaDeroinFrancaviglia2014}, the same object is also called a \emph{branched projective structure}, whose \emph{branch points} are its apparent singularities; we use this name when discussing bubbling, where the underlying complex structure is not fixed.
}%
It is clear that a projective structure is determined once a local holomorphic function $w_\alpha: U_\alpha \rightarrow \C$ is determined.
Given a point $x$ and the analytic continuation $[\gamma].w_\alpha$ of $w_\alpha$ along a closed path $\gamma$ containing $x$,
the composition along $\gamma$ of the Möbius transformations 
maps $([\gamma].w_\alpha)(x)$ back to $w_\alpha(x)$.
In other words, the composition along a closed path of the Möbius transformations
is the inverse of the $\mathrm{PSL}_2(\C)$-action defined by analytic continuation.
The holomorphic function $\widehat{w}_\alpha: \widehat{C} \rightarrow \mathbb{P}^1$,
where $\widehat{C}$ is the universal cover of $C$,
defined by analytically continuing $w_\alpha$ 
is called a developing map of the projective structure.

Given a projective connection $(P,s)$,
by letting $s$ be represented by local holomorphic functions on charts of a sufficiently refined atlas of $C$, 
one can define local holomorphic functions the values of which are related by constant $\mathrm{PSL}_2(\C)$-valued transition functions.
A maximal atlas together with these local holomorphic functions define a projective structure.
Conversely, from the local holomorphic functions of a projective structure one can define a section of a flat bundle with $\mathbb{P}^1$-fibers,
the constant $\mathrm{PSL}_2(\C)$-valued transition functions of which are the Möbius transformations.
In other words, a projective structure is equivalent to an isomorphism class of projective connections.

\subsubsection{Opers with apparent singularities}\label{subsubsec-SL-operators}

\begin{definition}
	Let $N$ be a line bundle on $C$, defined via transition functions $(N)_{\alpha \beta}$ over a coordinate covering $\mathcal{U} = \{ (U_\alpha, z_\alpha) \}_{\alpha \in \mathcal{I}}$ of $C$. 
	An \emph{oper} $\mathcal{D}$ on $N$ for the coordinate covering $\mathcal{U}$ is a collection of meromorphic differential operators $\{ \mathcal{D}_\alpha = \partial_{z_\alpha}^2 + q_\alpha(z_\alpha)\}_{\alpha \in \mathcal{I}}$ 
	such that $f_\alpha(z_\alpha)$ is a solution to $\mathcal{D}_\alpha$ 
	if and only if $(N)_{\beta \alpha} f_\alpha$ is a solution to $\mathcal{D}_\beta \mid_{U_\alpha \cap U_\beta}$. 
	The points at which $\left( \frac{f_{\alpha,1}}{f_{\alpha,2}} \right)' = 0$, where $f_{\alpha,1}$ and $f_{\alpha,2}$ are two linearly independent solutions of $\mathcal{D}_\alpha$, 
	are called apparent singularities,
	with the order of the apparent singularity defined to be the order of the zero of $\left( \frac{f_{\alpha,1}}{f_{\alpha,2}} \right)'$. 
	Two opers on $N$ for two coordinate coverings are considered equivalent if their union is also an oper. 
	An oper on $N$ is an equivalence class of opers on $N$ for different coverings.
\end{definition}
{%
Opers in this sense are allowed to have apparent singularities, and we will say \emph{opers with apparent singularities} when we want to stress this.
An oper without apparent singularities is called \emph{holomorphic}; these are the opers in the usual sense \cite{BD, Fr07}.
The same objects are called \emph{$SL$-operators} by Iwasaki \cite{iwasaki91, iwasaki92}; Yoshida relates them to meromorphic projective connections and shows that every irreducible representation $\pi_1(C) \rightarrow \mathrm{PSL}_2(\C)$ is the monodromy of one with at most $3g-3$ apparent singularities \cite[Theorem~B]{Yoshida97}.
}%

We denote by $\Div(\mathcal{D})$ the divisor formed by apparent singularities, counted with multiplicity,
of the oper $\mathcal{D}$.

\begin{remark}
	It was shown in \cite{iwasaki91} that an oper exists on a line bundle $N$ if and only if $\deg(N) = 1-g$, 
	such as $N \cong K^{-1/2}$. 
	Such an oper then can be regarded as a differential operator $N \rightarrow N K^2$ whose principal symbol is 1 and subprincipal symbol is 0 \cite{Fr07}.
\end{remark}

It is clear from the definition of opers on $N$ that,
if $N'$ is a flat line bundle with constant transition functions,
the same collection of local differential operators 
would define opers on both $N$ and $N \otimes N'$.
Hence the specific line bundle on which these local differential operators act is not of interest to us.
From now on, we will fix $N = K^{-1/2}$ and by an oper we will mean an oper on $K^{-1/2}$.

We now elaborate on the explicit forms of an oper before showing its equivalence to the notions of projective connections and projective structures.

\paragraph{Transformation rules of local differential operators.}
Similar to projective connections and projective structures,
an oper is determined once we know its local form $\mathcal{D}_\alpha = \pa_{z_\alpha}^2 + q_\alpha(z_\alpha)$ 
over one coordinate open set $\left( U_\alpha, z_\alpha \right)$.
This is because on the overlap $U_\alpha \cap U_\beta$,
\begin{align}\label{pre2_SL_operator_transit}
	2 q_\beta(z_\beta(z_\alpha)) (z_\beta'(z_\alpha) )^2 = 2 q_\alpha(z_\alpha) - \left\{z_\beta, z_\alpha \right\}
\end{align}
where $\left\{ g(z), z \right\} \coloneqq \frac{g'''}{g'} - \frac{3}{2} \left( \frac{g''}{g'} \right)^2$ is the Schwarzian derivative of a function $g(z)$. 
This follows from $\left\{ w(z_\alpha) , z_\alpha \right\} = 2 q_\alpha (z_\alpha)$, where $w \coloneqq \frac{f_{\alpha,1}}{f_{\alpha,2}}$ is the ratio of two linearly independent solutions to $\mathcal{D}_\alpha$ \cite{W15}, and the transformation rule of a Schwarzian derivative upon a change of coordinates
{%
(cf. \cite[(3.1)--(3.2)]{Dinh24})%
}%
. 

There are two important consequences of \eqref{pre2_SL_operator_transit} that we will exploit heavily in this work.
One is that if an oper $\mathcal{D}$ has no apparent singularity on $U_\alpha$ and hence the ratio $w_\alpha$ of two linearly independent local solutions can serve as a local coordinate, 
then $\mathcal{D}$ takes the simple form $\partial_{w_\alpha}^2$ in terms of the local coordinate $w_\alpha$.  
Another is that in case the change of coordinate $z_\alpha \mapsto z_\beta(z_\alpha)$ is a Möbius transformation,
which occurs if and only if $\{z_\beta, z_\alpha\} = 0$,
it follows from \eqref{pre2_SL_operator_transit} that $\{q_\alpha(z_\alpha)\}$ glue into a quadratic differential whose local forms are $q_\alpha(z_\alpha) dz_\alpha^2$.

Hence if the coordinate covering $\{(U_\alpha, w_\alpha)\}$ is a holomorphic projective structure, i.e. one that has no apparent singularity, 
then the collection of local differential operators $\{ \pa_{w_\alpha}^2 \}$
defines an oper.
Any other collection of local differential operators $\{ \pa_{w_\alpha}^2 + q_\alpha(w_\alpha) \}$
defines an oper if and only if
$\{ q_\alpha(z_\alpha) dz_\alpha^2 \}$ glue into a meromorphic quadratic differential.

\paragraph{Local forms near apparent singularities.}
 It follows from \eqref{pre2_SL_operator_transit} that, given a local coordinate $z$, the oper takes the form $D_z = \partial_z^2 + \frac{1}{2}\left\{ w(z), z \right\}$ for $w$ being the ratio of two linearly independent solutions. 
{%
The Laurent tail of $q(z) = \frac{1}{2}\left\{ w(z), z \right\}$ at an apparent singularity $x$ depends on its order, i.e. the order of the zero of $w'$ at $x$ \cite[Section~3.1]{Dinh24}, \cite[Chapter~5]{dinh-thesis}:
}%
\begin{subequations}
	\begin{align}
		q (z) =  -  \frac{3}{4(z - z(x))^2}
		+  \frac{\nu_{z}}{z - z(x)}
		+ q_{z, 0}
		+ \mathcal{O}(z - z(x))  &
		\label{SL-operator_tail-1} \\
		\text{with } \nu_{z}^2 + q_{z, 0} = 0 &
		\label{SL-operator_residue-constraint-1}
	\end{align}
\end{subequations}
if $x$ is a simple apparent singularity, and
\begin{subequations}
\begin{align}
	 q (z) = - \frac{2}{(z - z(x))^2} + \frac{\mu_{z}}{z - z(x)} + q_{z, 0} + q_{z, 1} (z - z(x)) + \mathcal{O}( (z - z(x))^2 ) &
	\label{SL-operator_tail-2} \\
	\text{ with } \mu_{z}^3 +  4 q_{z, 0} \mu_{z} + 4 q_{z, 1} = 0 &
	\label{SL-operator_residue-constraint-2}
\end{align}
\end{subequations}
if $x$ is a double apparent singularity.
Although the specific coefficients of the Laurent tails \eqref{SL-operator_tail-1} and \eqref{SL-operator_tail-2}
depend on the coordinate $z$ via the expansion of $w(z)$,
the polynomial constraints \eqref{SL-operator_residue-constraint-1} and \eqref{SL-operator_residue-constraint-2}
they satisfy are invariant upon a change of coordinates.
These Laurent tails ensure that even though solutions to $\mathcal{D}_\alpha$ are singular and might
have non-trivial monodromy around these singularities,
their ratios are holomorphic there.

\paragraph{Correspondence to projective structures.}
The two equivalent
{%
descriptions
}%
are related
by taking the ratio $w_\alpha = f_{\alpha,1}/f_{\alpha,2}$ of two linearly
independent solutions to $\mathcal{D}_\alpha$
and by noting that $q_\alpha(z_\alpha) = \frac{1}{2} \left\{ w_\alpha(z_\alpha), z_\alpha \right\}$.
Hence a projective structure is equivalent to an isomorphism class of opers,
with the positions and order of apparent singularities matched.  
The projective monodromy representation $\pi_1(C) \rightarrow \mathrm{PSL}_2$ 
that are inherent in the notions of a projective connection and projective structure 
is realized in an oper via the ratios $w_\alpha$ of local solutions.
Note that these ratios of local solutions are holomorphic at apparent singularities,
and hence the projective monodromy representations do not detect these singularities.

\begin{example}
	On a compact Riemann surface $C$ of genus $g\geq 2$, the most distinguished projective connection is induced by the uniformization theorem, which realizes $C$ as a quotient of the upper-half plane
	and equips on it a distinguished maximal coordinate atlas $\{\left( U_\alpha, x_\alpha \right) \}$ from the upper-half plane.
	We call this the uniformizing projective structure of $C$.
	Since nowhere $dx_\alpha = 0$, the projective structure $\{\left( U_\alpha, x_\alpha \right) \}$ has no apparent singularity. 
	The corresponding oper takes the form $\partial_{x_\alpha}^2$ in each $U_\alpha$. 
	This uniformizing projective structure is distinguished in the sense that the projective coordinates $x_\alpha$ all take values in the upper half-plane.
	
	It follows from \eqref{pre2_SL_operator_transit} that if $q$ is a holomorphic quadratic differential, then $\mathcal{D}_\alpha \coloneqq \partial_{x_\alpha}^2 + q(x_\alpha)$ glue into an oper that has no apparent singularity. 
	These form the space of holomorphic projective connections,
	i.e. those that have no apparent singularities, which is an affine space modeled over $H^0(K^2) \cong \mathbb{C}^{3g-3}$.
	The isomorphism class of each such oper is equivalent
	to a maximal coordinate atlas that is also a projective structure.
\end{example}

\subsubsection{Holomorphic connections and monodromy representations to $\pmb{\SL_2(\C)}$}\label{subsubsec-holomorphic-connections}
Let $F^s$ be a smooth rank-2 bundle on $C$ with a fixed isomorphism 
between its determinant line bundle $\det(F^s)$ and the trivial line bundle. 
We will let $F^s$ be the underlying smooth objects for the following holomorphic objects.

An $\SL_2(\C)$-holomorphic connection $\nabla$ on an $\SL_2(\C)$-holomorphic bundle $F$ 
is a $\C$-linear map $\nabla: F \rightarrow FK$ 
satisfying the Leibniz rule $\nabla(fs) = \partial f\otimes s + f \nabla s$ 
for any local holomorphic function $f$ and local section $s$ of $F$,
such that the induced connection on $\det(F^s)$ is the trivial connection $\partial: f \mapsto f(z) dz$. 
Since there is no $(2,0)$-form on a Riemann surface, 
$(F, \nabla)$ is automatically a flat connection 
and gives rise to a monodromy representation 
$\check{\rho}_{(F,\nabla)}: \pi_1 \rightarrow \SL_2(\C)$ via developing local parallel frames. 
In this case, we say $(F, \nabla)$ realizes the monodromy representation $\check{\rho} = \check{\rho}_{(F,\nabla)}$.

By changing to local frames of $F$ that consist of local parallel sections 
and using a new set of transition functions with respect to these local flat frames,
one can define a flat bundle $F^\nabla$ with constant $\SL_2(\C)$-valued transition functions,
the composition of which along a closed path $\gamma$ is $\check{\rho}_{(F,\nabla)}([\gamma])$.
Since the transition functions of $F^\nabla$ are constant, 
$\partial$ is a well-defined holomorphic connection on $F^\nabla$.
The changes of local frames define an isomorphism $F \overset{f}{\rightarrow} F^\nabla$ 
such that $\nabla = f^{-1} \circ \partial \circ f$,
and we say $(F, \nabla)$ and $(F^\nabla, \partial)$ are isomorphic as holomorphic connections. 
In general, for two $\SL_2(\C)$-holomorphic connections $(F_1, \nabla_1)$ and $(F_2, \nabla_2)$,
we say they are isomorphic if there exists a smooth automorphism 
of the underlying bundle $F^s$ relating $\nabla_1$ and $\nabla_2$.

On the other hand, given a monodromy representation $\check{\rho}: \pi_1(C) \rightarrow \SL_2(\C)$,
it is straightforward to define a flat rank-2 bundle $F^{\check{\rho}}$ with constant $\SL_2(\C)$-valued transition functions 
that realizes $\check{\rho}$.
Clearly a holomorphic connection $(F,\nabla)$ realizes $\check{\rho}$ 
if and only if $(F,\nabla)$ is isomorphic to $(F^{\check{\rho}}, \partial)$ as holomorphic connections.

\paragraph{Moduli spaces.}
{%
Recall from \Cref{subsubsec-natural-stratification} that
}%
$\mathcal{M}_{dR}$ is the moduli space of irreducible $\SL_2(\C)$-holomorphic connections,
i.e. ones that leave no holomorphic subbundle of the underlying $\SL_2(\C)$-holomorphic bundle invariant.
It is known that $\mathcal{M}_{dR}$ is a smooth complex analytic space of complex dimension $6g-6$ \cite{simpson-moduli2} \cite{collier-wentworth}.

Consider the set $\Hom(\pi_1, \SL_2(\C))$ of homomorphisms from $\pi_1$ to $\SL_2(\C)$.
It has the structure of an affine variety, and admits an action from $\SL_2(\C)$ by conjugation.
We define the $\SL_2(\C)$-representation variety by
\begin{align*}
	\mathcal{R}_{\SL_2(\C)} = \Hom^{irr}(\pi_1, \SL_2(\C)) // \SL_2(\C)
\end{align*}
where $\Hom^{irr}(\pi_1, \SL_2(\C))$ is the subspace of irreducible monodromy representations 
and the double slash indicates invariant theoretic quotient \cite{simpson-moduli2} \cite{W15}.
It is known that $\mathcal{R}_{\SL_2(\C)}$ is an irreducible affine variety of complex dimension $6g-6$.
Taking the monodromy representation of an irreducible $\SL_2(\C)$-holomorphic connection gives an
 irreducible monodromy representation in $\SL_2(\C)$,
and this  defines a homeomorphism
$\mathcal{M}_{dR} \rightarrow \mathcal{R}_{\SL_2(\C)}$ \cite{W15}.

\subsubsection{From holomorphic connections to projective connections}
\paragraph{Lifts of projective monodromy representations.}
Given a projective monodromy representation $\rho: \pi_1 \rightarrow \mathrm{PSL}_2(\C)$,
we say $\check{\rho}: \pi_1 \rightarrow \SL_2(\C)$ is a lift of $\rho$ if
$\rho$ is equal to the composition of $\check{\rho}$ with the projection $\SL_2(\C) \rightarrow \mathrm{PSL}_2(\C)$.
Concretely, it means that for generators $A_{i= 1, ...,g}$, $B_{j=1, ...,g}$ of $\pi_1$ representing a basis of cycles of $C$,
one can find $\check{\rho}(A_i), \check{\rho}(B_j) \in \SL_2(\C)$ 
that projects to $\rho(A_i), \rho(B_j) \in \mathrm{PSL}_2(\C)$
and satisfy
\begin{equation}\label{lift-condition}
	\left[ \check{\rho}(A_g), \check{\rho}(B_g) \right] ... \left[ \check{\rho}(A_1), \check{\rho}(B_1) \right] = 1.
\end{equation}

A projective monodromy representation $\rho$, if it has a lift to $\SL_2(\C)$, 
has in total $2^{2g}$ lifts modulo conjugation:
these lifts correspond to the freedom to choose the sign $\{ \pm\}$ for the lifts of $\rho(A_i)$ and $\rho(B_j)$ to $\SL_2(\C)$,
since these signs cancel after taking the commutator in \eqref{lift-condition}.
On the other hand, two monodromy representations $\check{\rho}_1, \check{\rho}_2: \pi_1 \rightarrow \SL_2(\C)$
are lifts of the same projective monodromy representation
if and only if
$\check{\rho}_1([\gamma]) = \pm \check{\rho}_2([\gamma])$ for any generator $\gamma \in \{A_1, ..., A_g, B_1, ..., B_g\}$.
In practice, we will care about lifts of projective monodromy representations modulo conjugation.

\begin{lemma}\label{lemma-twist-monodromy}
	Two $\SL_2(\C)$-holomorphic connections $(F_1, \nabla_1)$, $(F_2, \nabla_2)$ 
	have monodromy representations $\check{\rho}_1, \check{\rho}_2: \pi_1 \rightarrow \SL_2(\C)$
	that are lifts of the same projective monodromy representation up to conjugation
	if and only if $F_1 \cong F_2 \otimes N$ for some line bundle $N$ with $N^2 \cong \mathcal{O}_C$. 
\end{lemma}
\begin{proof}
	The condition $\check{\rho}_1([\gamma]) = \pm \check{\rho}_2([\gamma])$ for any closed path $\gamma$
	is equivalent to the fact that the two flat bundles $F_1^{\nabla_1}$ and $F_2^{\nabla_2}$
	are such that $F_1^{\nabla_1} \cong F_2^{\nabla_2} \otimes N$ 
	where $N$ is a flat line bundle with constant transition functions valued in $\{\pm 1 \}$.
	This occurs if and only if $N^2 \cong \mathcal{O}_C$.
\end{proof}

\paragraph{Projectivization.}
Given a holomorphic connection $(F, \nabla)$,
we can projectivize the $\SL_2(\C)$-flat bundle $F^\nabla$ to obtain a $\mathrm{PSL}_2(\C)$-flat bundle $\mathbb{P}(F^\nabla)$, 
the $\mathbb{P}^1$-fibers of which are the projectivization of the fibers of $F^\nabla$.
The constant $\mathrm{PSL}_2(\C)$-valued transition functions of $\mathbb{P}(F^\nabla)$ are induced by the 
constant $\SL_2(\C)$-valued transition functions of $F^\nabla$ via the projection $\SL_2(\C) \rightarrow \mathrm{PSL}_2(\C)$.
Let $L \hookrightarrow F$ be a subbundle that is not invariant by $\nabla$,
and $L \cong L^\nabla \hookrightarrow F^\nabla$ the corresponding subbundle induced by the isomorphism $F \rightarrow F^\nabla$.
Then $\mathbb{P}(L^\nabla)$ defines a section of $\mathbb{P}(F^\nabla)$ and hence a projective connection.
Conversely, if $(P,s)$ is a projective connection with projective monodromy representation which has a lift to $\SL_2(\C)$,  
then $(P,s)$ is equivalent to $(\mathbb{P}(F^\nabla), \mathbb{P}(L^\nabla))$ for some holomorphic connection $(F, \nabla)$ and subbundle $L$ of $F$.

\paragraph{Irreducible projective monodromy representations.}
We say a projective monodromy representation $\rho: \pi_1 \rightarrow \mathrm{PSL}_2(\C)$ is reducible 
if $\rho(\pi_1) \subset \mathrm{PSL}_2(\C) \equiv Aut(\mathbb{P}^1)$ has a fixed point on $\mathbb{P}^1$,
and irreducible if it is not reducible.
Clearly if $\rho$ has a lift to $\SL_2(\C)$,
it is irreducible if and only if its lift is irreducible.

Similar to how irreducible monodromy representations in $\SL_2(\C)$ are in 1-1 correspondence 
with irreducible flat $\SL_2(\C)$-connections,
irreducible projective monodromy representation
are in 1-1 correspondence with 
flat $\mathrm{PSL}_2(\C)$-bundles with $\mathbb{P}^1$-fibers that have no global parallel section.
Furthermore, if an irreducible projective monodromy representation $\rho$ has a lift $\check{\rho}$ 
in $\SL_2(\C)$ which is realized by an irreducible holomorphic connection $(F, \nabla)$,
then $P = \mathbb{P}(F^\nabla)$ realizes $\rho$.

\paragraph{Inducing apparent singularities.}
Given an $\SL_2(\C)$-holomorphic connection $(F,\nabla)$ and a subbundle $L$ of $F$,
the composition
\begin{equation}\label{lower-left-dR}
	s_{L}(\nabla): L \hookrightarrow F \overset{\nabla}{\rightarrow} F  K \rightarrow {L}^{-1} K,
\end{equation}
where the last arrow is induced by the quotient $F \rightarrow L^{-1}$ of the embedding $L \hookrightarrow F$, 
is nonzero if $L$ is not invariant by $\nabla$.
Hence in particular if $(F,\nabla)$ is irreducible then for all subbundles $L$ of $F$ we have $s_{L}(\nabla) \neq 0$.
A priori, unlike \eqref{lower-left-component} which is $\mathcal{O}_C$-linear,
\eqref{lower-left-dR} is only $\mathbb{C}$-linear since it involves $\nabla$. 
However, since $FK \rightarrow L^{-1}K$ is induced from the quotient map of $L \hookrightarrow F$,
overall \eqref{lower-left-dR} is $\mathcal{O}_C$-linear.

The following proposition shows that the zero divisor of \eqref{lower-left-dR} is the loci of apparent singularities.
In this sense, \eqref{lower-left-dR} is the analogue of \eqref{lower-left-component}, 
and apparent singularities are the analogues of the projection to $C$ of Baker--Akhiezer divisors.

\begin{proposition}
	Suppose $(F, \nabla)$ is a holomorphic connection and $L \hookrightarrow F$ a subbundle not invariant by $\nabla$.
	Then the divisor of apparent singularities $\Div( (\mathbb{P}(F^\nabla), \mathbb{P}(L^\nabla)))$
	of the projective connection $(\mathbb{P}(F^\nabla), \mathbb{P}(L^\nabla))$
	coincides with the zero divisor of $s_{L}(\nabla)$.
\end{proposition}
\begin{proof}
{%
		See \cite[Section~4.1]{Dinh24} and \cite[Chapter~5]{dinh-thesis}.
}%
\end{proof}

Let $(F_1, \nabla_1)$ and $(F_2, \nabla_2)$ be two irreducible $\SL_2(\C)$-holomorphic connections
with monodromy representations that project to the same projective monodromy up to conjugation.
By \Cref{lemma-twist-monodromy},
$F_1 \cong F_2 \otimes N$ where $N^2 \cong \mathcal{O}_C$.
If $L_2$ is a subbundle of $F_2$, then there exists a subbundle $L_1 \cong L_2 \otimes N$ of $F_1$
such that $s_{L_1}(\nabla_1)$ is identified with 
$s_{L_2}(\nabla_2)$ via the isomorphism $KL_1^{-2} \cong KL_2^{-2}$.
In particular, the projective connections 
$(\mathbb{P}(F_1^{\nabla_1} ), \mathbb{P}(L_1^{\nabla_1}) )$ and $(\mathbb{P}(F_2^{\nabla_2} ), \mathbb{P}(L_2^{\nabla_2}) )$
are isomorphic.

We say two data $(L_1 \hookrightarrow F_1, \nabla_1)$ and $(L_2 \hookrightarrow F_2, \nabla_2)$ 
are isomorphic if there is an isomorphism $L_1 \overset{\sim}{\rightarrow} L_2$ 
that commutes with an isomorphism $F_1 \overset{\sim}{\rightarrow} F_2$  
which makes $(F_1, \nabla_1)$ isomorphic to $(F_2, \nabla_2)$.
The projective connections $(\mathbb{P}(F_1^{\nabla_1}), \mathbb{P}(L_1^{\nabla_1}))$ 
and $(\mathbb{P}(F_2^{\nabla_2}), \mathbb{P}(L_2^{\nabla_2}))$ 
clearly are isomorphic if and only if $(L_1 \hookrightarrow F_1, \nabla_1)$ 
and $(L_2 \hookrightarrow F_2, \nabla_2)$ are isomorphic. 

\begin{corollary}\label{cor-isom-projective-connections}
	Two projective connections 
	with the same irreducible projective monodromy representations up to conjugation
	that have lifts to $\SL_2(\C)$
	and the same divisors of apparent singularities
	of degree $< 2g-2$
	are isomorphic.
\end{corollary}
\begin{proof}
	Two such projective connections are the projectivization of a holomorphic connection $(F,\nabla)$ and subbundles $L_1$, $L_2$ of $F$ 
	such that $s_{L_1}(\nabla)$, $s_{L_2}(\nabla)$ both vanish at $\mathbf{x}$ and are identified
	via the isomorphism $KL_1^{-2} \cong \mathcal{O}_C(\mathbf{x}) \cong KL_2^{-2}$. 
	But $\deg(\mathbf{x}) < 2g-2$, therefore $L_1$ and $L_2$ destabilize $F$ and must be the same subbundle of $F$.
\end{proof}

{%
\paragraph{Stratification of $\mathrm{PSL}_2(\C)$-flat bundles.}\label{par-PSL2-strata}
Let $\mathcal{M}_{dR}(C, \mathrm{PSL}_2(\C))$ be the moduli space of irreducible flat $\mathrm{PSL}_2(\C)$-bundles on $C$ whose monodromy lifts to $\SL_2(\C)$.
By \Cref{lemma-twist-monodromy}, projectivization identifies it with the quotient of $\mathcal{M}_{dR}$ by twisting with $2$-torsion line bundles $N$.
Twisting maps subbundles $L$ of $F$ to subbundles $L \otimes N$ of $F \otimes N$ of the same degree, hence preserves the strata of \Cref{par-closed-strata}; we denote by $\mathcal{M}_d$ the image of the stratum labelled by $d$.
Since $\deg \Div(\mathbb{P}(F^\nabla), \mathbb{P}(L^\nabla)) = \deg(KL^{-2}) = 2g-2-2\deg L$, the projective connections realizing a point of $\mathcal{M}_d$ have at least $d$ apparent singularities, with equality exactly when $L$ is a maximal subbundle of $F$.
In particular, the closed stratum $\mathcal{M}^{\leq d}_{dR}$ consists of the flat connections that can be represented by an oper with at most $d$ apparent singularities.
}%

\subsubsection{From holomorphic connections to opers}
\paragraph{Associated opers.}
Given the initial data $\left(L \hookrightarrow F, \nabla \right)$, we want to have a concrete construction of an oper with apparent singularities that corresponds up to equivalence to the projective connection $(\mathbb{P}(F^\nabla), \mathbb{P}(L^\nabla) )$.
To this end, suppose $\nabla = \partial_z + \begin{pmatrix} a(z) & b(z) \\ c(z) & -a(z) \end{pmatrix}$ in certain local frames adapted to $L$.
Consider the local differential operator $\partial_z^2 + q(z)$ where
\begin{align}\label{potential-after-merom-gauge}
	q(z) = - b(z) c(z) - \left( a(z) - \frac{c'(z)}{2c(z)} \right)^2 - \left( a(z) - \frac{c'(z)}{2c(z)} \right)'.
\end{align}
It is straightforward to check that $q(z)$ is invariant upon a change of
local frames adapted to $L$
w.r.t.
{%
which
}%
the local $1$-forms $a(z)$, $b(z)$ and $c(z)$ are defined.
One can also check that, upon a change of coordinates and transformations of $1$-forms
\begin{align*}
	&a_w(w) = a(z(w)) z'(w), &
	&b_w(w) = b(z(w)) z'(w), &
	&c_w(w) = c(z(w)) z'(w),
\end{align*}
{%
$q$ transforms as in \eqref{pre2_SL_operator_transit}.
}%
We conclude that the collection of local differential operators
$\partial_z^2 + q(z)$ defines an oper, the isomorphism class of which depends only on the isomorphism class of the data $\left(L \hookrightarrow F, \nabla \right)$ \cite{Dinh24}.

Now suppose $z=0$ is a simple zero of $c(z)$. 
Then the Laurent expansion of \eqref{potential-after-merom-gauge}
\begin{equation}\label{pre2_expansion_of_q}
	q(z) = \frac{-3}{4}\frac{1}{z^2} + \left[  a(0) -   \frac{c''(0)}{4c'(0)} \right]\frac{1}{z}   -  a(0)^2 +   \frac{a(0) c''(0)}{2c'(0)} -   \left( \frac{c''(0)}{4c'(0)} \right)^2 + \mathcal{O}(z)
\end{equation}
makes $z=0$ an apparent singularity.
While apparent singularities can be regarded as the analogues of the projection to $C$ of Baker--Akhiezer divisors, 
the accessory parameter
\begin{equation}\label{accessory-parameter}
	a(0) -  \frac{c''(0)}{4c'(0)}
\end{equation}
can be regarded as the analogue of the coordinate $-a(0)$ of the point in the Baker--Akhiezer divisor projecting to $z=0$.
The Laurent expansion of $q(z)$ also makes a double zero of $s_L(\nabla)$ a double apparent singularity.
This enables us to make the following definition.

\begin{definition}\label{def-associated-SL-operator}
	Given an irreducible $\SL_2(\C)$-holomorphic connection $(F,\nabla)$ and a subbundle $L \hookrightarrow F$,
	the associated oper $\mathcal{D}_{(L\hookrightarrow F, \nabla)}$
	 takes the local form $\partial_z^2 + q(z)$ where $q(z)$ is defined in \eqref{potential-after-merom-gauge}.
	The divisor of apparent singularities $\Div(\mathcal{D}_{(L\hookrightarrow F, \nabla)})$ coincides with the zero divisor of $s_L(\nabla)$.
\end{definition}

\begin{proposition}
	Let $(F,\nabla)$ be an irreducible $\SL_2(\C)$-holomorphic connection 
	and $L \hookrightarrow F$ a subbundle.
	Then the associated oper $\mathcal{D}_{(L \hookrightarrow F, \nabla)}$ 
	is equivalent to the isomorphism class 
	of the projective connection $(\mathbb{P}(F^\nabla), \mathbb{P}(L^\nabla))$.	
\end{proposition}
\begin{proof}
{%
		See \cite[Proposition~4.1]{Dinh24} and \cite[Chapter~5]{dinh-thesis}.
}%
\end{proof}

\begin{remark}
	Twisting the data $(L \hookrightarrow F, \nabla)$
	by a square-root of $\mathcal{O}_C$
{%
induces equivalent opers.
}%
	There exist in total $2^{2g}$ such data, which differ from each other by a twist
	by a square-root of $\mathcal{O}_C$,
	that define the same oper up to equivalence.
\end{remark}

\paragraph{Terminology.}
By \Cref{subsubsec-projective-structures} and the proposition above, a projective structure on $C$ possibly with apparent singularities, an isomorphism class of projective connections, and an oper are equivalent descriptions of the same object, a $\mathrm{PSL}_2(\C)$-oper with apparent singularities in the terminology of \cite{Dinh24}; we simply call it an \emph{oper}, or an \emph{oper with apparent singularities} for emphasis.
When its monodromy lifts to $\SL_2(\C)$, it is given by data $(L \hookrightarrow F, \nabla)$, up to twisting by a square root of $\mathcal{O}_C$, which we call an \emph{$\SL_2(\C)$-oper with apparent singularities}.
We denote it by $\mathcal{D}$ and use its local form $\partial_z^2 + q(z)$ when computing in a local coordinate $z$.

This terminology matches the partial oper stratification \eqref{BB-stratifications}.
Let $d = \deg \Div(\mathcal{D}) = 2g-2-2\deg L$.
If $d \leq 2g-2$, i.e.\ $\deg L \geq 0$, the associated graded Higgs bundle $(L \oplus L^{-1}, s_L(\nabla))$ is semistable, so $(L \hookrightarrow F, \nabla)$ is a partial oper in the sense of \cite[Definition~2.4]{simpson-iterated}.
If moreover $d < 2g-2$, then $L$ is the destabilizing subbundle of $F$, and by \Cref{prop-rank2-limits}(ii) the flat bundle $(F, \nabla)$ lies in the leaf of $(N_d)^+_1$ over $[E_L, \phi_{\Div(\mathcal{D})}]$.
Hence, for $d < 2g-2$, the stratum $(N_d)^+_1$ is the space of $\SL_2(\C)$-opers with $d$ apparent singularities, and the divisor of apparent singularities determines the leaf; for $d = 0$ this is the stratum of holomorphic opers.
For $d > 2g-2$, the data $(L \hookrightarrow F, \nabla)$ is not a partial oper, and the stratum of $(F, \nabla)$ is determined by the minimal number of apparent singularities (\Cref{par-PSL2-strata}).

\paragraph{Varying apparent singularities.}
In general, there exist different opers having different apparent singularities that realize the same projective monodromy.
Indeed, an irreducible $\SL_2(\C)$-holomorphic connection $(F, \nabla)$ together with different subbundles of $F$ induce different isomorphism classes
of opers with different apparent singularities.
However, in case the number of apparent singularities, counted with multiplicity, is less than $2g-2$,
there exists a unique subbundle that is maximal and destabilizes
$F$.
Hence in this case, given a projective monodromy representation with lift to $\SL_2(\C)$, 
the isomorphism class of the opers realizing it is unique (cf. \Cref{cor-isom-projective-connections}).

\section{Degeneration of Baker--Akhiezer divisors of Higgs bundles}
\label{sect-degeneration-BA}

In this section we describe the analogues for Higgs bundles  
of the families in \Cref{sect-degeneration-apparent-singularities}.
Specifically, we discuss families of Higgs bundles in $\mathcal{M}_H \equiv \mathcal{M}_H(\Lambda)$, the underlying bundles of which admit subbundles of the same degree, that limit to a Higgs bundle with the underlying bundle admitting a subbundle of higher degree. 
These types of families are understood most easily in terms of the corresponding Baker--Akhiezer divisors,
so we start from this perspective.

\subsection{Reduction of the degree of Baker--Akhiezer divisors}\label{subsec-reduction-degree-BA}

Let $(E,\phi) \in \mathcal{M}_H$ be a Higgs bundle
with associated non-degenerate spectral curve 
$\widetilde{C} \overset{\pi}{\rightarrow} C$ with involution $\sigma$, 
$L \hookrightarrow E$ a subbundle and $\widetilde{D}$ the Baker--Akhiezer divisor of $\left( L \hookrightarrow E, \phi \right)$. 
Recall from \Cref{subsec-BA-divisors} that 
$\widetilde{D}$ does not contain any summand being the pull-back of a divisor on $C$ and satisfies $L^2 
\cong K\Lambda \otimes \mathcal{O}_C(-\mathrm{Nm}(\widetilde{D}))$,
and the spectral line bundle of $(E, \phi)$ is isomorphic to $
	\pi^\ast(L) \otimes \mathcal{O}_{\widetilde{C}}(\widetilde{D})$.

Consider a sufficiently small, simply connected open set $U \subset C$ that does not contain a branch point of $\widetilde{C}$.
Consider an effective divisor $\tilde{x}_+ + \tilde{x}_-$ on $\widetilde{C}$
where $\tilde{x}_\pm$ lie in distinct components of $\pi^{-1}(U)$
and are such that $\tilde{x}_+ \neq \sigma(\tilde{x}_-)$.
In other words, for $x_\pm = \pi(\tilde{x}_\pm)$
we have $x_+ \neq x_-$.
By \Cref{prop-inverse-BA}, $\widetilde{D} + \tilde{x}_+ + \tilde{x}_-$ is the Baker--Akhiezer divisor 
of some Higgs bundle in $\mathcal{M}_H$ induced by some subbundle which is of degree $\deg(L) -1$ and is a square-root of $K \Lambda \otimes \mathcal{O}_C(-\mathrm{Nm}(\widetilde{D} + \tilde{x}_+ + \tilde{x}_-))$.
For each pair $(x_+, x_-)$, there are $2^{2g}$ such square-roots.
Varying $(x_+, x_-)$ then
induces $2^{2g}$ families of line bundles of degree $\deg(L)-1$.
If $U$ is sufficiently small, these families define distinct subsets
in $\mathrm{Pic}^{\deg(L)-1}(C)$, 
each of which, given a point $x_0 \in U$, 
contains upon closure a unique square-root of $K\Lambda \otimes \mathcal{O}_C(-\mathrm{Nm}(\widetilde{D}) - 2x_0)$ in the limit $x_+ \rightarrow x_0 \leftarrow x_-$.
One particular family would contain $L(-x_0)$ upon closure.
Let $L_{x_+, x_-}$ be the square-root of $K \Lambda \otimes \mathcal{O}_C(-\mathrm{Nm}(\widetilde{D} + \tilde{x}_+ + \tilde{x}_-))$
contained in that family,
i.e.
\begin{subequations}
	\begin{equation}\label{limit-subbundle}
		\underset{x_+ \rightarrow x_0 \leftarrow x_-}{\lim} [L_{x_+, x_-}] = [L(-x_0)] \in \mathrm{Pic}^{\deg(L)-1}(C)
\end{equation}
for any $x_0 \in U$.

Let $(L_{x_+, x_-} \hookrightarrow E_{x_+, x_-}, \phi_{x_+, x_-})$ be the data that induces $\widetilde{D} + \tilde{x}_+ + \tilde{x}_-$ as its Baker--Akhiezer divisor.
The spectral line bundle of $(E_{x_+, x_-}, \phi_{x_+, x_-})$ is isomorphic to $\pi^\ast(L_{x_+, x_-}) \otimes \mathcal{O}_{\widetilde{C}}(\widetilde{D} + \tilde{x}_+ + \tilde{x}_-)$.
It follows from \eqref{limit-subbundle} that as $x_+ \rightarrow x_0 \leftarrow x_-$,
this family of line bundles limits to $\pi^\ast(L) \otimes \mathcal{O}_{\widetilde{C}}(\widetilde{D})$, 
which is the spectral line bundle of $(E, \phi)$.
In other words,
	\begin{equation}
		\underset{x_+ \rightarrow x_0 \leftarrow x_-}{\lim} (E_{x_+, x_-}, \phi_{x_+, x_-}) = (E, \phi) \in \mathcal{M}_H
	\end{equation}
	for any $x_0 \in U$.
\end{subequations}
This is clear in terms of Baker--Akhiezer divisors:
for any $x_0 \in U$, $\widetilde{D} + \pi^\ast(x_0)$ 
is the Baker--Akhiezer divisor induced by $(E,\phi)$ via the injection $L(-x_0) \rightarrow L \hookrightarrow E$ which vanishes at $x_0$.
The following proposition summarizes our discussion.

\begin{proposition}\label{prop-limit-BA-divisor}
	Let $(E, \phi) \in \mathcal{M}_H$ with 
	associated non-degenerate spectral curve $\widetilde{C} \overset{\pi}{\rightarrow} C$, 
	and $L \hookrightarrow E$ a line subbundle with Baker--Akhiezer divisor $\widetilde{D}$. 
	Let $U \subset C$ be a simply connected open set that contains no branch point of $\widetilde{C}$ and no point of $\mathrm{Nm}(\widetilde{D})$, 
	and let $U' \coloneqq \left\{ (x_+, x_-) \in U^2 \mid  x_+ \neq x_- \right\}$.
	Then there exist a holomorphic family $\{(E_{x_+, x_-}, \phi_{x_+, x_-})\}_{(x_+, x_-) \in U'}$ in $\mathcal{M}_H$ 
	and line subbundles $L_{x_+, x_-} \hookrightarrow E_{x_+, x_-}$ of degree $\deg(L) - 1$, depending holomorphically on $(x_+, x_-)$, 
	such that 
	\begin{enumerate}[(a)]
		\item for all $(x_+, x_-) \in U'$, the Baker--Akhiezer divisor of $(L_{x_+, x_-} \hookrightarrow E_{x_+, x_-}, \phi_{x_+, x_-} )$ is $\widetilde{D} + \tilde{x}_+ + \tilde{x}_-$, 
		where $\tilde{x}_\pm \in \pi^{-1}(x_\pm)$ lie in two fixed distinct components of $\pi^{-1}(U)$; 
		\item for all $x_0 \in U$, 
		$\underset{x_+ \rightarrow x_0 \leftarrow x_-}{\lim} [L_{x_+, x_-}] = [ L(-x_0)]$ in $\mathrm{Pic}^{\deg(L)-1}(C)$
		and $\underset{x_+ \rightarrow x_0 \leftarrow x_-}{\lim}(E_{x_+, x_-}, \phi_{x_+, x_-}) = (E, \phi)$ in $\mathcal{M}_H$.
	\end{enumerate}
\end{proposition}

In
{%
\Cref{sect-degeneration-apparent-singularities}
}%
where we analogously analyze the collision limit of apparent singularities of opers,
we will find it convenient for calculation 
to specialize to the case where $U$ is equipped with a coordinate $z$ 
with $z(x_0) = 0$ for some $x_0\in U$ and $z(x_\pm(u)) = \pm u$.
It is instructive to now spell out the specialization of \Cref{prop-limit-BA-divisor} to this case
in order to later see the analogy between colliding points
in the projection of Baker--Akhiezer divisors and apparent singularities.

\begin{proposition}\label{prop-Hitchin-bb}
	Let $(E, \phi) \in \mathcal{M}_H$ with 
	associated non-degenerate spectral curve $\widetilde{C} \overset{\pi}{\rightarrow} C$, 
	and let $L \hookrightarrow E$ be a line subbundle with Baker--Akhiezer divisor $\widetilde{D}$. 
	Given $x_0 \in C$, neither a branch point of $\widetilde{C}$ nor a point of $\mathrm{Nm}(\widetilde{D})$, 
	let $(U, z)$ be a coordinate neighborhood of $x_0$,
	where $z(x_0)=0$, $U$ is simply connected and contains no branch point of $\widetilde{C}$ and no point of $\mathrm{Nm}(\widetilde{D})$.
	For $u \in z(U)$, let $x_\pm(u)$ be the points with $z(x_\pm(u)) = \pm u$.
	Then there exist a holomorphic family $\{ (E_u, \phi_u)\}_{u \in z(U)}$ in $\mathcal{M}_H$ with spectral curve $\widetilde{C}$ and $(E_0, \phi_0) = (E, \phi)$, 
	and a holomorphic family of line bundles $\{L_u\}_{u \in z(U)}$ of degree $\deg(L) - 1$ with $L_0 \cong L(-x_0)$, 
	such that, for all $u \neq 0$, $L_u \hookrightarrow E_u$ is a line subbundle and 
	the Baker--Akhiezer divisor of $(L_u \hookrightarrow E_u, \phi_u)$ is $\widetilde{D} + \tilde{x}_+(u) + \tilde{x}_-(u)$, 
	where $\tilde{x}_\pm(u) \in \pi^{-1}(x_\pm(u))$ lie in two fixed distinct components of $\pi^{-1}(U)$.
\end{proposition}

Note that if $\deg(\widetilde{D}) < 2g-2$, these families limit to smaller Bia\l ynicki-Birula strata, i.e. the strata of Higgs bundles with increasingly unstable underlying bundles.
In the notation of the stratification defined by Hodge bundles in \Cref{subsubsec-natural-stratification}, 
by \Cref{prop-rank2-limits}(i), the limit $(E, \phi)$ lies in $(N_{\deg(\widetilde{D})})^+_0$, while for $u \neq 0$ the Higgs bundle $(E_u, \phi_u)$ of \Cref{prop-Hitchin-bb} lies in $(N_{\deg(\widetilde{D})+2})^+_0$ if $\deg(\widetilde{D}) + 2 < 2g-2$.
The limits of these families compactify the open dense locus consisting of Higgs bundles with stable underlying bundles into the whole Prym variety.

\subsection{Double point in Baker--Akhiezer divisors}\label{subsec-double-point-BA}
In the above discussion, we have analyzed families of Baker--Akhiezer divisors
whose limits contain a summand of the form $\pi^\ast(x_0)$
for some $x_0 \in C$. 
The following proposition is the counter-part of \Cref{prop-Hitchin-bb} for the case where the limit contains a double point that projects to $2x_0$.
In this case, 
unlike in \Cref{prop-Hitchin-bb}, 
the underlying bundle of the corresponding Higgs bundle at the limit
admits a subbundle of the same degree as in the family.
\begin{proposition}\label{prop-Hitchin-db}
	Let $(E, \phi) \in \mathcal{M}_H$ with the associated non-degenerate spectral curve $\widetilde{C} \overset{\pi}{\rightarrow} C$ 
	and let $L \hookrightarrow E$ be a line subbundle
	such that $s_L(\phi)$ has a double zero at $x_0 \in C$ which is not a branch point of $\widetilde{C}$. 
	Let $\widetilde{D}$ be the Baker--Akhiezer divisor of $(L \hookrightarrow E, \phi)$
	and $\tilde{x}_0$ be the point with multiplicity 2 in $\widetilde{D}$
	with $\pi(\tilde{x}_0) = x_0$.
	Let $(U,z)$ be a coordinate neighborhood of $x_0$, where $z(x_0) = 0$, $U$ is simply connected and contains no branch point of $\widetilde{C}$.
	For $u \in z(U)$, let $x_\pm(u)$ be the points with $z(x_\pm(u)) = \pm u$.
	Then there exist a family of Higgs bundles $\{ (E_u, \phi_u)\}_{u \in z(U)}$ and a family of line bundles $\{L_u\}_{u \in z(U)}$ of the same degree as $L$ 
	such that 
	\begin{enumerate}[(i)]
		\item $[L_0] = [ L ]$ in $\mathrm{Pic}^{\deg(L)}(C)$
		and $(E_0, \phi_0) = (E, \phi)$ in $\mathcal{M}_H$;
		\item for all $u \in z(U)$, $E_u$ admits $L_u$ as a subbundle;
		\item for all $u \neq 0$, the Baker--Akhiezer divisor of $(L_u \hookrightarrow E_u, \phi_u)$ is $\widetilde{D} - 2\tilde{x}_0 + \tilde{x}_+(u) + \tilde{x}_-(u)$, 
		where $\tilde{x}_\pm(u) \in \pi^{-1}(x_\pm(u))$ lie in the component of $\pi^{-1}(U)$ containing $\tilde{x}_0$.
	\end{enumerate}
\end{proposition}
\begin{proof}
	The proposition follows from the discussion preceding \Cref{prop-limit-BA-divisor} by requiring
	$\tilde{x}_\pm$ now to be contained in the same component of $\pi^{-1}(U)$ containing $\tilde{x}_0$.
\end{proof}

\section{Degeneration of apparent singularities of projective structures}\label{sect-degeneration-apparent-singularities}

\subsection{Meromorphic quadratic differentials and opers with apparent singularities}
{%
Throughout this section, $\mathcal{D}$ denotes an oper with apparent singularities, and $\mathrm{div}(\mathcal{D})$ denotes the divisor of its apparent singularities, a double apparent singularity being counted twice.
}%
In this and the next sections, we will fix a holomorphic projective structure $\{(U_\alpha, z_\alpha)\}_{\alpha \in \mathcal{I}}$
and work with opers on this maximal coordinate atlas.
In these coordinate charts, an oper $\mathcal{D}$
with apparent singularities $x_1$, ..., $x_d$
 takes the form $\{ \mathcal{D}_{\alpha} =  \partial_{z_\alpha}^2 +  q_\mathcal{D} (z_\alpha) \}_{\alpha \in \mathcal{I}}$ 
 where $\{q_\mathcal{D}(z_\alpha) dz_\alpha^2 \}$ 
 define a meromorphic quadratic differential $q_\mathcal{D}$
 that has poles precisely at the apparent singularities $x_1$, ..., $x_d$.
{%
The Laurent tails of $q_\mathcal{D}$ at the apparent singularities are given by \eqref{SL-operator_tail-1}--\eqref{SL-operator_residue-constraint-2}.
}%

\paragraph{Building blocks of meromorphic quadratic differentials.}
For each $x \in C$, there exists quadratic differentials that have poles at $x$ and is holomorphic everywhere else. 
We will be interested in those that have simple and double poles at $x$.
We denote the space of all such meromorphic quadratic differentials as $\Omega_x^{\otimes 2}$ and $\Omega_{2x}^{\otimes 2}$ respectively.
Clearly $\Omega_x^{\otimes 2} \cong H^0\left(K_C^2 \otimes \mathcal{O}_C(x) \right)$ and $\Omega_{2x}^{\otimes 2} \cong H^0 \left(K_C^2 \otimes \mathcal{O}_C(2x) \right)$. 
In particular, 
elements of $\Omega_x^{\otimes 2}$ ($\Omega_{2x}^{\otimes 2}$) that are holomorphic at $x$ 
and hence on all of $C$ correspond to sections of $K^2_C \otimes \mathcal{O}_C(x)$ (respectively, $K^2_C \otimes \mathcal{O}_C(2x)$) 
that admit $x$ as a zero of multiplicity at least 1 (respectively, 2).

In the following we discuss, given a local coordinate $z$ of $x$,
certain quadratic differentials that have simple and double pole at $x$
with specific Laurent tails in the expansion w.r.t. $z$.
We will think of these specific quadratic differentials as building blocks, 
out of which we can construct all other elements of $\Omega_x^{\otimes 2}$ and $\Omega_{2x}^{\otimes 2}$
in a way such that we can control their Laurent tails w.r.t. $z$.

To this end, first note that since $h^0\left(K_C^2 \otimes \mathcal{O}_C(x) \right) = 3g - 2$, an element in $\Omega_x^{\otimes 2}$ is fixed up to scaling upon fixing $3g-3$ out of its $4g-3$ zeroes in total.
Similarly, an element in $\Omega_{2x}^{\otimes 2}$ is fixed up to scaling upon fixing $3g-2$ out of its $4g-2$ zeroes.
To characterize the meromorphic quadratic differentials that have non-trivial simple and double poles at $x$, let us first recall 
the notion of $Q$-generic divisors
{%
\cite{dinh-thesis}, \cite[Section~3.2]{Dinh24}%
}%
.
We say that an effective divisor $\mathbf{x}$ on $C$ is $Q$-generic if the space
$Q_\mathbf{x} = \{ q \in H^0(C, K^2) \mid \mathbf{x} \leq \mathrm{div}(q) \} \cup \{ 0 \}$
of quadratic differentials vanishing at $\mathbf{x}$, with multiplicity counted,
is of expected, i.e. minimal, dimension $\max\{0, 3g-3 - \deg(\mathbf{x})\}$.
We say that $\mathbf{x}$ is $Q$-special otherwise.
{%
For a reduced $Q$-generic divisor, opers with simple apparent singularities at its points are characterized by their accessory parameters, up to adding elements of $Q_\mathbf{x}$ \cite[Propositions~3.2 and 3.3]{Dinh24}; the proof uses ``building blocks'' of quadratic differentials, an idea first used by Iwasaki \cite{iwasaki91}.
}%
Then an element of $\Omega_x^{\otimes 2}$ ($\Omega_{2x}^{\otimes 2}$) that has a non-trivial simple (respectively, double) pole at $x$ is characterized by the fact that 
if an effective divisor of degree  $3g-3$ (respectively, $3g-2$) is contained in its zero divisor (and hence determines it up to scaling), 
then this divisor is $Q$-generic. 
It is almost obvious that 
by varying the pole $x$ in its neighborhood 
we can construct a family of quadratic differentials that is holomorphically parameterized by this neighborhood,
while keeping intact the freedom to choose the zeroes.
\Cref{prop-merom-q-existence-1} formalizes this intuition.

\begin{proposition}\label{prop-merom-q-existence-1}	
	Let $U \subset C$ be an open subset, $i\in \{1, 2\}$, $d = 3g-4 + i$, and
	$\mathbf{x}^d: U \rightarrow C^{[d]}$, $x \mapsto \mathbf{x}^d(x) = x_1(x) + ... + x_d(x)$,
	be a holomorphic map such that, 
	for all $x\in U$ and $r \in \{1 , ..., d\}$,
	$x_r(x) \neq x$ 
	and $\mathbf{x}^d(x)$ is $Q$-generic.
	Then there exists a family of meromorphic quadratic differentials $\left\{ q^{(i)}_{x, \mathbf{x}^d} \right\}_{x \in U}$ 
	holomorphically parameterized by $U$ such that 
		\begin{enumerate}[(a)]
			\item $q^{(i)}_{x, \mathbf{x}^d}$ has a simple (double) pole at $x$ if $i=1$ (respectively, $i=2$) and is holomorphic on $C\setminus \{x\}$;
			\item $\mathbf{x}^d(x)$ is contained in the support of the zero divisor of $q^{(i)}_{x, \mathbf{x}^d}$.
		\end{enumerate}
\end{proposition}
\begin{proof}%
{%
See \cite[Proposition~6.1]{dinh-thesis}.%
}%
\end{proof}

The families parameterized by $U$ that satisfies conditions (a) and (b) in \Cref{prop-merom-q-existence-1} are far from unique: 
scaling one such family by a function holomorphic on $U$ produces another.
Given a local coordinate $z$ on $U$, however, we can look to control the Laurent tails of these quadratic differentials w.r.t. $z$.
The resulting ``normalized'' families would be unique,
but we would be able to impose less zeroes.
  
\begin{proposition}	\label{building-blocks-2}
	Let $(U, z)$ be an open coordinated subset of $C$. \\
	1. Let $\mathbf{x}^{3g-3}: U \rightarrow C^{[3g-3]}$, 
	$x \mapsto \mathbf{x}^{3g-3}(x) = x_1(x) + ... + x_{3g-3}$, 
	be a holomorphic map 
	such that for all $x\in U$ and $r \in \{1, ..., 3g-3\}$,
	$x_r(x) \neq x$ and $\mathbf{x}^{3g-3}(x)$ is $Q$-generic.
	Then there exists a unique family 
	of meromorphic quadratic differentials 
	$\left\{ q^{(i)}_{z, x, \mathbf{x}} \right\}_{x \in U}$ 
	holomorphically parameterized by $U$ such that,
	for each $x \in U$,
	$q^{(i)}_{z, x, \mathbf{x}^{3g-3}}$ satisfies conditions (a), (b) in \Cref{prop-merom-q-existence-1}, 
	and furthermore takes the local form 
		\begin{align}
			\frac{dz^2}{( z - z(x))^i} + R^{(i)}_{z, x, \mathbf{x}^{3g-3}} (z) dz^2,
			\label{normalised-q-1}
		\end{align}
	where $R^{(i)}_{z, x, \mathbf{x}^{3g-3}} (z)$ is the restriction 
	of a holomorphic function $R^{(i)}(z_1, z_2)$ defined on $(U, z) \times (U, z)$
	to the slice $\{ z_1 = z(x), z_2 \equiv z \}$. \\ %
	2. Let $\mathbf{x}^{3g-4}: U \rightarrow C^{[3g-4]}$, 
	$x \mapsto \mathbf{x}^{3g-4}(x) = x_1(x) + ... + x_{3g-4}$, 
	be a holomorphic map 
	such that for all $x\in U$ and $r \in \{1, ..., 3g-4\}$,
	$x_r(x) \neq x$ and $x + \mathbf{x}^{3g-4}(x)$ is $Q$-generic.
	Then there exists a unique family of meromorphic quadratic differentials  
	$\left\{ q^{(i)}_{0, z, x, \mathbf{x}^{3g-4}} \right\}_{x \in U}$ 
	holomorphically parameterized by $U$ such that,
	for each $x \in U$,
	$q^{(i)}_{0, z, x, \mathbf{x}^{3g-4}}$ satisfies conditions (a), (b) in \Cref{prop-merom-q-existence-1}, 
	and furthermore takes the local form 
		\begin{align}
			\frac{dz^2}{\left( z - z(x) \right)^i} + R^{(i)}_{0, z, x, \mathbf{x}^{3g-4}}(z) dz^2,
			\label{normalised-q-2}
		\end{align}
	where $R^{(i)}_{0, z, x, \mathbf{x}^{3g-4}} (z)$ is the restriction 
	of a function $R^{(i)}_0 (z_1, z_2)$, which is defined on $(U, z_1) \times (U, z_2)$
	and vanishes at the diagonal $\{ z_1 = z_2 \}$,
	to the slice $\{ z_1 = z(x), z_2 \equiv z \}$.
\end{proposition}
\begin{proof}%
{%
See \cite[Proposition~6.2]{dinh-thesis}.%
}%
\end{proof}

\paragraph{Some technical results}
In the following we will be interested in families of quadratic differentials of certain forms.
The idea, based on \cite{iwasaki91,iwasaki92}, is to work with families of ``building blocks'' of quadratic differentials and analyze the limit of these families as the poles collide.

Let $(U,z)$ be a coordinated neighborhood of $x_0 \in C$
and $ \mathbf{x'} = x_3 + ... + x_{d}$ for $d \leq 3g-3$ a reduced effective divisor on $C$ 
such that, for all $x_1, x_2 \in U$, $x_1 + x_2 + \mathbf{x'}$ is $Q$-generic.
For each $r\in \{3, ..., d\}$, let $\mathbf{x_r'} \coloneqq \mathbf{x'} - x_r$.
Then, refining $U$ if necessary, 
we can choose some effective divisors $\mathbf{w} = w_0 + w_{d+1} + ... + w_{3g-3}$ on $C$ with support distinct from $U \cup \{x_3, ..., x_d \}$
such that,
for all $x_1$, $x_2 \in U$ and $r\in \{3, ..., d\}$,
$x_1 + \mathbf{x'} + \mathbf{w}$ and $x_1 + x_2 + \mathbf{x_r'} + \mathbf{w}$ are $Q$-generic.

Now, for each $u\in z(U)$,
let $x_\pm(u)$ be points in $U$ having coordinates $\pm u$ respectively.
We use the following short-hand notations for the unique families of quadratic differentials from \Cref{building-blocks-2},
\begin{align}\label{building-blocks-0}
	&q^{(i)}_{u, \pm} \coloneqq q^{(i)}_{0, z, x_\pm, \mathbf{w} + \mathbf{x'}} 
	\overset{ \text{on } (U,z)}{\equiv} \frac{dz^2}{(z \mp u)^i} + R^{(i)}_{u, \pm}(z) dz^2,
	&i \in \{1,2\}, 
\end{align}
where $R^{(i)}_{u, \pm}(z)$ is the restriction 
to the slice $\{(u, z) \}$ of a function $R^{(i)}_0 (z_1, z)$
which is holomorphic in both variables 
and vanishes at the diagonal $\{z = u \}$.

\begin{lemma}\label{lemma-linear-system-right-side}
	Fix $i \in \{1,2\}$. \\
	1. The functions $\frac{1}{u} R^{(i)}_{u, +}(-u)$ and $-\frac{1}{u} R^{(i)}_{u,-}(u)$ are holomorphic on $U$ with variable $u = z(x_+)$. 
	Furthermore, their evaluations at $u = 0$ are equal. \\
	2. For each fixed $z$, the limit 
	\begin{align}
		\check{R}^{(i)} (z) \coloneqq \underset{u \rightarrow 0}{\lim} \frac{1}{u}\left( R^{(i)}_{u,+} (z) - R^{(i)}_{u,-}(z) \right)
		\label{shift-after-collision}
	\end{align}
	exists. 
	These limits define a holomorphic function on $U$ with variable $z$.
\end{lemma}
{%
The proof is given in \Cref{app-technical-proofs}.
}%

In the following, we analyze the limit when two simple apparent singularities collide.
As two simple apparent singularities approach each other, 
each contributes and blows-up the coefficient $q_{z,0}$ in the Laurent expansion in \eqref{SL-operator_tail-1} at the other one.
The constraint \eqref{SL-operator_residue-constraint-1} implies that  their accessory parameters will also blow-up.
We will show that, if these respective accessory parameters blow-up in a particular way,
the limit upon collision will be either an oper with a double apparent singularity at the collision site,
or an oper with two less apparent singularities.

\subsection{Setup}
\subsubsection{Conditions on the collision site.}
For $2 \leq d \leq 3g-3$,
let $x_0$ and $\mathbf{x}' = x_3 + ... + x_d$
be reduced effective divisors on $C$ 
such that $x_0 \nleq \mathbf{x}'$ and
\begin{align}\label{condition-neighborhood}
	2x_0 + \mathbf{x'} \text{ is } Q\text{-generic.}
\end{align} 
A consequence of this condition is that 
for all pairs of distinct points $x_1$ and $x_2$ in a sufficiently small neighborhood $U$ of $x_0$,
we have
\begin{align}\label{condition-neighborhood-2}
	x_1 + x_2 + \mathbf{x'} \text{ is } Q\text{-generic.}
\end{align} 
Condition \eqref{condition-neighborhood-2} is equivalent to requiring that, given a basis $\mathbf{q} = \left( q_1, ..., q_{3g-3} \right)$ of $H^0\left(K_C^2 \right)$,
the $d \times (3g-3)$ matrix $\mathbf{q_{r,k}} \coloneqq q_k( z_r(x_r))^{1 \leq k \leq 3g-3}_{1 \leq r \leq d}$,
defined by evaluating $q_k$ at $x_r$ using some local coordinate $z_r$,
is of maximal rank $d$.

This condition is satisfied by a generic choices of $x_0$ and $\mathbf{x'}$.

\paragraph{Determinant with simple zero at $x_0$.}
Consider the case where $d = 3g-3$, and $x_0$, $\mathbf{x'}$ satisfy condition \eqref{condition-neighborhood}.
Then upon choosing local coordinates and a basis $\mathbf{q} = (q_k)$ of $H^0(K^2)$,
any $x_1$, $x_2 \in U$
define a $(3g-3)\times (3g-3)$ non-degenerate matrix $\mathbf{q_{r,k}}(x_1, x_2)$.

We will in the following evaluate $q_k$ at $x_1$ and $x_2$ using the same local coordinate $z$ on $U$.
Furthermore, we will make the following choices of $x_1$ and $x_2$.
W.l.o.g., suppose $z(x_0) = 0$.
Then for each $u \in z(U)$, let $x_1(u)$ and $x_2(u)$ be points in $U$ that have coordinates $\pm u$ respectively;
in particular, $x_0 = x_1(0) = x_2(0)$.
Then a choice of a basis $\mathbf{q} = (q_k)$ of $H^0(K^2)$ and local coordinates $z_3$, ..., $z_{3g-3}$ for $x_3$, ..., $x_{3g-3}$
together define a function $\det(\mathbf{q_{r,k}})(u): U \cong z(U) \rightarrow \C$.
This function is holomorphic in $u$ and vanishes only at $u=0$,
since at $u = 0$ the matrix $\mathbf{q_{r,k}}$ is degenerate.
The following proposition shows that if we use a certain family $(q_{k, u})_{k = 1}^{3g-3}$ of basis of $H^0(K^2)$ that also varies holomorphically in $u$, 
the zero at $u = 0$ of the corresponding determinant is simple.

\begin{lemma}\label{lemma-basis-of-quadratic-differentials}
	Let $x_0$ and $\mathbf{x'} = x_3 + ... + x_{3g-3}$ satisfy \eqref{condition-neighborhood},
	and $(U,z)$, $x_1(u)$, $x_2(u)$ be constructed as above.
	Then there exists a family $\mathbf{q}_{u} \coloneqq \{ (q_{k, u})_{k = 1}^{3g-3} \}_{u \in z(U)}$ of basis of $H^0(K^2)$ that is parameterized by $U$ and is such that the function
	$\det (\mathbf{q_{r,k}})(u)$,
	where $\mathbf{q_{r,k}} = q_{k, u}(z_r(x_r))$,
	has a simple zero at $u = 0$.
\end{lemma}
{%
The proof is given in \Cref{app-technical-proofs}.
The family constructed there is of the form $\mathbf{q}_u = (q_1, q_2, q_{3,u}, \dots, q_{3g-3,u})$, where $q_1$ and $q_2$ do not depend on $u$ and vanish at $\mathbf{x'}$, $q_1$ has a simple zero at $x_0$, $q_2$ does not vanish at $x_0$, and, for $3 \leq k \leq 3g-3$, $q_{k,u}$ vanishes at $x_1(u)$, $x_2(u)$ and satisfies $q_{k,u}(z_r(x_r)) = \delta_{kr}$ for $3 \leq r \leq 3g-3$.
}%
{%
For $d < 3g-3$, the same construction applies after replacing $\mathbf{x'}$ by $\mathbf{x'} + w_{d+1} + \dots + w_{3g-3}$, where $w_{d+1}, \dots, w_{3g-3}$ are auxiliary points such that $2x_0 + \mathbf{x'} + w_{d+1} + \dots + w_{3g-3}$ is $Q$-generic.
}%

\subsubsection{Families of meromorphic quadratic differentials.}
Given $x_0$ and $\mathbf{x'}$ satisfying condition \eqref{condition-neighborhood}, 
let $(U,z)$, $(U_3, z_3)$, ..., $(U_d, z_d)$ be some respective coordinate neighborhoods of $x_0$, $x_3$, ..., $x_d$, 
which do not intersect each other and are subordinate to a fixed holomorphic projective structure.
W.l.o.g., let $z(x_0) = 0$; 
for each $u\in z(U)$, let $x_\pm (u) \in U$ be defined by $z(x_\pm) = \pm u$.
Let $\mathbf{x_r'} \coloneqq \mathbf{x}' - x_r$.

Let $\nu_\pm(u)$, $\nu_3(u)$, ..., $\nu_{d}(u)$ be holomorphic functions from $U \setminus \{x_0\}$ to $\mathbb{C}$.
We want to characterize a family of opers $\{ \mathcal{D}_u \}_{u \in z(U), u\neq 0}$ holomorphically parameterized by $U \setminus \{x_0\}$,
where $\mathcal{D}_u$ has simple apparent singularities at $x_\pm(u)$, $x_3$, ..., $x_d$
and respective accessory parameters $\nu_\pm(u)$, $\nu_3(u)$, ..., $\nu_d(u)$
w.r.t. local coordinates $z$, $z_3$, ..., $z_d$.
We do this by writing $\mathcal{D}_u = \mathcal{D}_{pr} + q_u$,
where $\mathcal{D}_{pr}$ is an oper defined by the chosen holomorphic projective structure, and $q_u$ is a meromorphic quadratic differential built in terms of the ``building blocks''.
One set of the ``building blocks'' comes from \eqref{building-blocks-0},
some $u \rightarrow 0$ limits of which are discussed in \Cref{lemma-linear-system-right-side}.
We also use the following short-hand notations 
\begin{align*}
	&q^{(i)}_{u, x_r} \coloneqq q^{(i)}_{0, z_r, x_r, x_+ + x_- + \mathbf{x_r'} + \mathbf{w} - w_0},
	&3 \leq r \leq d,
\end{align*}
for ``normalized'' quadratic differentials that have simple/double pole at $x_r$
and vanishes at 
$$x_+ + x_-  + x_3 + ... + x_{r-1} + x_{r+1} + ... + x_d + w_{d+1} + ... + w_{3g-3}.$$ 
We also denote the $u \rightarrow 0$ limits 
of these families by 
\begin{align}\label{common-building-block-limit}
	&q^{(i)}_{0} \coloneqq q^{(i)}_{0, z, x_0, \mathbf{w} + \mathbf{x'}} 
	\overset{ \text{on } (U,z)}{\equiv} \frac{dz^2}{z^i} + R^{(i)}_{0}(z) dz^2,
	&q^{(i)}_{0, x_r} \coloneqq q^{(i)}_{0, z_r, x_r, 2x_0 + \mathbf{x_r'} + \mathbf{w} - w_0}.
\end{align}

\begin{remark}
	We have defined $q^{(i)}_{u, \pm}$ upon choosing a divisor $\textbf{w} = w_0 + w_{d+1} + ... + w_{3g-3}$.
	Choosing a different choice of $\mathbf{w}$ amounts to translating the definitions of $q^{(i)}_{u, \pm}$ and $q^{(i)}_{u, x_r}$ 
	by a holomorphic quadratic differential that vanishes at $\mathbf{x'}$ and $x_+ + x_- + \mathbf{x_r'}$ respectively.
\end{remark}

Now, to construct $q_u$ from the ``building blocks'', for each nonzero $u \in z(U)$, let
\begin{align}
	q_u = q^{(0)}_u 
	+ \sum_{s\in\{ \pm\}} \nu_{s}(u) q^{(1)}_{u, s}
	- \frac{3}{4} \sum_{s\in \{\pm\}} q^{(2)}_{u, s}
	+ \sum_{r=3}^{d} \nu_r(u) q_{u, x_r}^{(1)} - \frac{3}{4} \sum_{r=3}^{d} q^{(2)}_{u, x_r},
	\label{q-u-setup}
\end{align}
where $q^{(0)}_u$ is the holomorphic quadratic differential 
that makes $q_u$ satisfy condition \eqref{SL-operator_residue-constraint-1},
which now due to the construction of $q^{(i)}_{u, \pm}$ and $q^{(i)}_{u, x_r}$ takes the form
\begin{subequations}\label{linear-system-total}
	\begin{align}
		q^{(0)}_u(z) \mid_{z=\pm u} &=  -\nu_\pm(u)^2 - \nu_\mp(u) q^{(1)}_{u,\mp}(z)\mid_{z=\pm u} + \frac{3}{4} q^{(2)}_{u, \mp}(z)\mid_{z=\pm u} \eqqcolon C_\pm(u),
		 \label{linear-eqn-1} \\
		q_{u}^{(0)} (z_r) \mid_{z_r = z_r(x_r)} 
		&= - \nu_r(u)^2, & \text{ for } 3 \leq r \leq d. \label{linear-eqn-far}
	\end{align}
\end{subequations}
The functions $\nu_\pm(u)$, $\nu_3(u)$, ..., $\nu_{d}(u)$ define via these constraints \eqref{linear-system-total}
for each $u \neq 0$ a non-homogeneous linear system in $H^0(K_C^2)$.
Due to condition \eqref{condition-neighborhood}, the homogeneous linear system is of maximal rank $d$.
It follows that, for each $u\neq 0$, 
the non-homogeneous linear system \eqref{linear-system-total} defines via its solutions a $(3g-3-d)$-dimensional subspace in $H^0(K^2)$, 
and hence a $(3g-3-d)$-dimensional family of meromorphic quadratic differentials $q_u$ with the appropriate Laurent tails at $x_\pm$, $x_3$, ..., $x_d$ defined by $\nu_\pm(u)$, $\nu_3(u)$, ..., $\nu_{d}(u)$.

We can represent the linear system \eqref{linear-system-total} more explicitly 
by using the family $\mathbf{q}_{u} = \{ (q_{k, u})_{k = 1}^{3g-3} \}_{u \in z(U)}$ of basis of $H^0(K^2)$ %
constructed in \Cref{lemma-basis-of-quadratic-differentials}.
For $u\neq 0$,
let $q^{(0)}_u = \sum_{k=1}^{3g-3} E_k(u) q_{k, u}$.
Then \eqref{linear-system-total} decouples into 
a $2\times 2$ linear system representing \eqref{linear-eqn-1}, 
\begin{subequations}\label{linear-system-decoupled}
	\begin{equation}\label{linear-system-decoupled-1}
		\begin{pmatrix}
			&q_1(z)\mid_{z=u} & q_2(z)\mid_{z=u}  \\
			&q_1(z)\mid_{z=-u} & q_2(z)\mid_{z=-u} 
		\end{pmatrix}
		\begin{pmatrix}
			E_1(u) \\
			E_2(u) 		
		\end{pmatrix}
		=  \begin{pmatrix}
			C_+(u) \\
			C_-(u) 
		\end{pmatrix},
	\end{equation}
and a trivial $(d-2) \times (d-2)$ one representing \eqref{linear-eqn-far},
\begin{equation}\label{linear-system-decoupled-2}
	E_{r}(u) = -\nu_r(u)^2 \text{ for } 3 \leq r \leq d.
\end{equation}
\end{subequations}
For $d < 3g-3$, the expansion coefficients $E_r(u)$ with $r$ in the range $d < r \leq 3g-3$ are not constrained and parameterize the $(3g-3-d)$-dimensional 
subspace in $H^0(K^2)$ consisting of solutions to \eqref{linear-system-total}.

\paragraph{Limits of the solutions to the linear systems.}
We will be interested in 
families of meromorphic quadratic differentials $\{q_u\}_{u \in z(U), u\neq 0}$ 
that have well-defined limits $q_0 \coloneqq \underset{u \rightarrow 0}{\lim} q_u$ 
with the Laurent tails at the poles having the appropriate forms to make these poles apparent singularities.
In other words, we are interested in families $\{\mathcal{D}_u\}_{u \in z(U)}$  of opers parameterized by $U$,
with $\mathcal{D}_u$ for $u \neq 0$ having simple apparent singularities at $x_\pm$, $x_3$, ..., $x_d$ 
and $\mathcal{D}_0$ being the limit as $x_+$ and $x_-$ collide.

More specifically, we will study families determined by $\nu_3(u)$, ..., $\nu_d(u)$ being holomorphic for all $u \in z(U)$,
and $\nu_\pm(u)$ holomorphic at $u \neq 0$. 
On $U$, the quadratic differential $q_u$ then takes the form
\begin{equation}\label{eqn-qu-expansion}
	q_u(z) = \frac{-3/4}{(z-u)^2} + \frac{-3/4}{(z+u)^2} + \frac{\nu_+(u)}{z-u} + \frac{\nu_-(u)}{z+u} + g_u(z),
\end{equation}
where $g_u(z)$ is holomorphic in $z$; for the families of Form (B) in \Cref{prop-bb-limit-3}, it is holomorphic in both $u$ and $z$ (\Cref{rem-gu-holomorphic}).
We consider accessory parameters $\nu_\pm(u)$ with Laurent expansions at $u = 0$ of one of the following two forms: Form (B), for ``bubbling'' (\Cref{prop-bb-limit-3} and \Cref{thm-bubble}),
\begin{equation}\label{eqn-form-B}
	\nu_\pm(u) = \pm \frac{3}{4u} \pm \nu' u + \mathcal{O}(u^2),
\end{equation}
and Form (D), for ``doubling'' (\Cref{prop-double-as}),
\begin{equation}\label{eqn-form-D}
	\nu_\pm(u) = \mp \frac{1}{4u} + \nu_0^\mathcal{D} \pm \nu' u + \mathcal{O}(u^2),
\end{equation}
where $\nu'$ and $\nu_0^\mathcal{D}$ are constants.

\begin{lemma}\label{lemma-limit-linear-system}
	Suppose $\nu_3(u)$, ..., $\nu_d(u): U \cong z(U) \rightarrow \C$ are holomorphic functions,
	and $\nu_\pm(u): U \cong z(U) \rightarrow \C$ are meromorphic functions with simple poles and Laurent expansions at $u = 0$ of either the form \eqref{eqn-form-B} or \eqref{eqn-form-D}.
	Then the induced non-homogeneous linear system \eqref{linear-system-total} limits to a degenerate $d \times (3g-3)$ linear system of rank $d-1$. 
\end{lemma}
{%
The proof is given in \Cref{app-technical-proofs}.
In the limit, $C_+(u)$ and $C_-(u)$ in \eqref{linear-eqn-1} tend to the same value $C_0$ given in \eqref{limit-RHS}, and the limit linear system \eqref{linear-system-total-limit} reads $q^{(0)}_0(0) = C_0$, $q^{(0)}_0(z_r(x_r)) = -\nu_r(0)^2$ for $3 \leq r \leq d$.
}%

\begin{lemma}\label{lemma-tune-limit-hol-q}
	Let $x_0$, $x_3$, ..., $x_{d}$ satisfy condition \eqref{condition-neighborhood},
	$\nu_0\in \C$, and $q \in H^0(K^2)$ a holomorphic quadratic differential.
	Then there exist holomorphic functions $\nu_3(u)$, ..., $\nu_{d}(u): z(U) \rightarrow \C$ 
	and $\nu_\pm(u): z(U\setminus \{x_0\}) \rightarrow \C$ with Laurent expansions at $u =0$ of the forms \eqref{eqn-form-B} or \eqref{eqn-form-D} (with $\nu_0$ as the coefficient of order $u^0$),
	and a family of holomorphic quadratic differentials $\{q^{(0)}_u\}_{u\in z(U)}$ such that
	\begin{enumerate}[(i)]
		\item $q^{(0)}_u$ solves the linear system \eqref{linear-system-total} defined by $\nu_\pm(u)$, $\nu_3(u)$, ..., $\nu_d(u)$ for $u\neq 0$;
		\item $q^{(0)}_0 = q$ solves the corresponding limit linear system \eqref{linear-system-total-limit}.
	\end{enumerate}
\end{lemma}
{%
The proof is given in \Cref{app-technical-proofs}.
}%

\subsection{Double apparent singularity as the limit}

In this subsection we will analyze the limit $\underset{u \rightarrow 0}{\lim} q_u$ 
if $\nu_\pm(u)$ take the form \eqref{eqn-form-D}.

\begin{lemma}\label{lemma-db-limit-1}
	Let $x_0$, $x_3$, ..., $x_d$ be distinct points on $C$, 
	$(U,z)$ a coordinate neighborhood of $x_0$ with $z(x_0) = 0$,
	$\nu_{\pm}(u) = \mp \frac{1}{4u} + \nu_0 \pm \nu' u + \mathcal{O}(u^2)$ holomorphic functions on $U\setminus \{x_0 \}$,
	and $q^{(i)}_{u, \pm}$ defined as in \eqref{building-blocks-0} for $u \in z\left( U \right)$, $u \neq 0$.
	Then the family of meromorphic quadratic differentials	
	\begin{align} \label{family-merom-q-db}
		&\sum_{s\in \{\pm\}} \nu_s(u) q^{(1)}_{u, s} - \frac{3}{4} \sum_{s\in \{\pm\}} q^{(2)}_{u, s},
	\end{align}
	which is parameterized by $U \setminus \{x_0\}$, 
	extends to a family parameterized by $U$.
	The meromorphic quadratic differential $q^{(2)}_{x_0}$ 
	corresponding to the extension to $x_0$
	is holomorphic on $C \setminus \{x_0\}$,
	vanishes at $x_3$, ..., $x_{d}$,
	and on $U$ takes the form
	\begin{align}\label{q_x_0-local-form}
		q^{(2)}_{x_0} 
		\equiv q^{(2)}_{x_0}(\nu_0) 
		= \left[ - \frac{2 }{z^2}  + \frac{2 \nu_0 }{z} 
		+ 2\nu_0  R^{(1)}_{0}(z)
		-\frac{1}{4} \check{R}^{(1)}(z) 
		- \frac{3}{2}R^{(2)}_0(z) \right] dz^2,
	\end{align}
	where $\check{R}^{(1)}(z)$ is defined in part 2 of \Cref{lemma-linear-system-right-side}
	and $R^{(i)}_0(z)$ defined in \eqref{common-building-block-limit}.
\end{lemma}
{%
The proof is given in \Cref{app-technical-proofs}.
}%

In the following, we fix a holomorphic projective structure 
and a corresponding oper $\mathcal{D}_{pr}$ (which has no apparent singularity).
Recall again that any other oper can be written as $\mathcal{D}_{pr} + q$ 
where $q$ is a meromorphic quadratic differential having double poles with appropriate Laurent tails in coordinates subordinate to the chosen holomorphic projective structure.

\begin{proposition}\label{prop-db-limit-2}
	Let $x_0$, $x_3$, ..., $x_{d}$ be distinct points on $C$ such that $2x_0 + x_3 + ... + x_d$ is $Q$-generic.
	Let $z$ be a coordinate on $U$ subordinate to the chosen holomorphic projective structure, with $z(x_0) = 0$.
	Let $\nu_3(u)$, ..., $\nu_d(u)$ be holomorphic functions on $U$,
	$\nu_{\pm}(u) = \mp \frac{1}{4u} + \nu_0 \pm \nu' u + \mathcal{O}(u^2)$ holomorphic functions on $U\setminus \{x_0 \}$,
	and $ \{\mathcal{D}_u = \mathcal{D}_{pr} + q_u \}_{u \in z(U), u \neq 0}$ the corresponding family of opers parameterized by $U\setminus \{x_0\}$
	where $q_u$ is defined as in \eqref{q-u-setup},
	with $\{ q^{(0)}_u \}_{u \in z(U), u \neq 0}$ the corresponding family of holomorphic quadratic differentials.
	
	If there exists a holomorphic quadratic differential $q^{(0)}_0$ with  $q^{(0)}_0 = \underset{u \rightarrow 0}{\lim} q^{(0)}_u$, 
	then $\{ \mathcal{D}_u \}_{u \in z(U), u\neq 0}$ extends to a family of opers parameterized by $U$.
	The oper corresponding to the extension to $x_0$ is $\mathcal{D}_0 = \mathcal{D}_{pr} + q_0$, where
	\begin{align}\label{db-limit}
		q_0 \overset{\text{on } (U,z)}{=} {q^{(0)}_0} + q^{(2)}_{x_0}(\nu_0) 
		+  \sum_{r=3}^{d} \nu_r(0) q_{0, x_r}^{(1)} - \frac{3}{4} \sum_{r=3}^{d} q^{(2)}_{0, x_r}.
	\end{align}
	(Recall the definition of $q^{(2)}_{0, x_r}$ in \eqref{common-building-block-limit}.)
	In particular, the apparent singularities of $\mathcal{D}_0$ consist of
	a double apparent singularity at $x_0$ with accessory parameter $\nu_0$,
	and simple apparent singularities at $x_3$, ..., $x_d$ with respective accessory parameters $\nu_3(0)$, ..., $\nu_d(0)$.
\end{proposition}
\begin{proof}
	The existence of $\mathcal{D}_0 = \mathcal{D}_{pr} + q_0$ with $q_0$ of the form \eqref{db-limit} follows from
	the definition \eqref{q-u-setup} of $q_u$
	and \Cref{lemma-db-limit-1}.
	Let $q^{reg}_0(z)$ be the regular part on $U$ of $q_0(z)$, i.e.
	\begin{align*}
		&q^{reg}_0(z) \coloneqq q_0(z) - \left(- \frac{2 }{z^2}   + \frac{2 \nu_0 }{z}  \right) dz^2 \nonumber \\
		= &\hspace{3pt} q^{(0)}_0(z)
		+ \left[ 2\nu_0  R^{(1)}_{0}(z)
		-\frac{1}{4} \check{R}^{(1)}(z) 
		- \frac{3}{2}R^{(2)}_0(z) \right]  dz^2
		+  \sum_{r=3}^{d} \nu_r(0) q_{0, x_r}^{(1)} (z) - \frac{3}{4} \sum_{r=3}^{d} q^{(2)}_{0, x_r} (z).
	\end{align*}
	To prove the proposition, it remains to show that the coefficients of the Laurent tail of $q_0(z)$ satisfy the condition making $x_0$ a double apparent singularity, 
	i.e. for $q^{reg}_0(z) =  \sum_{k\geq 0} q_{0, k} z^k$ we have
	\begin{align}\label{db-condition-limit}
		\nu_0^3 + q_{0, 0} \nu_0 + \frac{q_{0,1}}{2} = 0.
	\end{align}
	To this end, for each $u\neq 0$, let $q^{reg}_u(z)$ be the regular part on $U$ of $q_u(z)$, i.e.
	\begin{align*}
		q^{reg}_u(z) 
		&\coloneqq q_u(z) -
		\left( - \frac{3}{4(z-u)^2} - \frac{3}{4(z+u)^2}
		+ \frac{\nu_+(u)}{z-u} + \frac{\nu_-(u)}{z+u} \right)  \nonumber \\
		&= q^{(0)}_u (z)
		+ \sum_{s \in \{\pm\}} \nu_s(u) R^{(1)}_{u,s}(z) 
		- \frac{3}{4} \sum_{s \in \{\pm\}} R^{(2)}_{u,s} (z) 
		+ \sum_{r=3}^{d} \nu_r(u) q_{u, x_r}^{(1)}(z) - \frac{3}{4} \sum_{r=3}^{d} q^{(2)}_{u, x_r} (z).
	\end{align*}
	As a function of both $u$ and $z$, the leading term in $u$ of $q^{reg}_u(z)$ 
	comes from the leading terms in $\nu_\pm(u) R^{(1)}_{u, \pm}(z)$ and is equal to $-(R^{(1)}_{u,+}(z) - R^{(1)}_{u,-}(z) )/( 4u)$.
	By part 2 of \Cref{lemma-linear-system-right-side}, 
	this has a well-defined $u\rightarrow 0$ limit.
	Hence $q^{reg}_u(z)$ has a well-defined $u\rightarrow 0$ limit,
	and in fact $q^{reg}_u(z)  \overset{ u \rightarrow 0}{\longrightarrow} q^{reg}_0(z)$. 
	
	Now, this enables us to expand $q^{reg}_u(z) = \sum_{m, n \geq 0} q_{m,n} u^m z^n$,
	where letting $m=0$ gives the coefficients $q_{0, n}$ of the Taylor expansion $q^{reg}_0(z) =  \sum_{n \geq 0} q_{0, n} z^n$.
	Plugging this expansions of $q^{reg}_0(z)$ and $\nu_{\pm}(u)$ into condition \eqref{SL-operator_residue-constraint-1} for $z = \pm u$, 
	\begin{align*}
		&\nu_\pm(u)^2 - \frac{3}{4} \frac{1}{(2u)^2} + \frac{\nu_\mp(u)}{(\pm 2 u)} + q^{reg}_u(\pm u) = 0,
		&u \neq 0,
	\end{align*}
	we obtain a series of constraints by order of $u$ starting from the order $u^{-2}$.
	The constraints up to order $u^{-1}$ are automatically satisfied by the leading terms of $\nu_\pm(u)$.
	Solving the constraints at order $u^0$ and $u^1$, 
	we obtain \eqref{db-condition-limit}.
\end{proof}

\begin{proposition} \label{prop-double-as}
	Let $\mathcal{D}$ be a branched projective structure with $\mathrm{div}(\mathcal{D}) = 2x_0 + x_3 + ... + x_d$ being $Q$-generic and $d \leq 3g-3$.
	Then there exists a coordinate neighborhood $(U, z)$ centred at $x_0$ 
	and a family of branched projective structures $\{\mathcal{D}_u\}_{u\in z(U)}$ parametrized by $U$ such that
	\begin{enumerate}[(i)]
		\item $\mathcal{D}_0 = \mathcal{D}$;
		\item for $u\neq 0$, $\mathrm{div}(\mathcal{D}_u) = x_+ + x_- + x_3 + \dots x_d$ with $x_\pm \in U$ and $z(x_\pm) = \pm u$;
		\item for $u\neq 0$, the accessory parameters $\nu_\pm(u)$ of $x_\pm$ w.r.t. the local coordinate $z$ have Laurent expansions as functions of $u$ of Form (D) \eqref{eqn-form-D},
		where $2\nu_0^\mathcal{D}$ is the accessory parameter of the double apparent singularity $x_0$ of $\mathcal{D}$.
	\end{enumerate}
	Furthermore, this family defines via taking monodromy a holomorphic map $U \rightarrow \mathrm{Hom}(\pi_1, \mathrm{PSL}_2)$, which is injective for $d < 2g-2$.     
\end{proposition}
\begin{proof}
	Fix a holomorphic projective structure with corresponding oper $\mathcal{D}_{pr}$, 
	and choose coordinate neighborhoods $(U_r, z_r)$ of $x_r$ for $3 \leq r \leq d$ 
	and $(U,z)$ of $x_0$ subordinate to this holomorphic projective structure,
	where $U$ satisfies condition \eqref{condition-neighborhood}.
	We can write $\mathcal{D} = \mathcal{D}_{pr} + q_{\mathcal{D}}$, where $q_\mathcal{D}$ is a meromorphic quadratic differential
	with local expression $-\frac{3}{4(z-z_r(x_r))^2} dz_r^2 + \frac{\nu^\mathcal{D}_r}{z-z_r(x_r)} dz_r^2 + ...$ on $U_{r}$, $3 \leq r \leq d$,
	and $-\frac{2}{z^2} dz^2 + \frac{2 \nu^\mathcal{D}_0}{z} dz^2 + ...$ on $U$.
	Then there exists a unique holomorphic quadratic differential $q^{(0)}_0 (\mathcal{D})$ such that $q_{\mathcal{D}}$ can be written in the form \eqref{db-limit}, i.e.
	\begin{align*}
		q_{\mathcal{D}} = 
		q^{(0)}_0 (\mathcal{D}) + q^{(2)}_{x_0}(\nu^\mathcal{D}_0)
		+  \sum_{r=3}^{d} \nu^\mathcal{D}_r q_{0, x_r}^{(1)} - \frac{3}{4} \sum_{r=3}^{d} q^{(2)}_{0, x_r}.
	\end{align*}
	By \Cref{lemma-tune-limit-hol-q}, there exist holomorphic functions $\nu_3(u)$, ..., $\nu_{d}(u): z(U) \rightarrow \C$, 
	$\nu_\pm(u): z(U\setminus \{x_0\}) \rightarrow \C$ with Laurent expansions at $u =0$ of the forms \eqref{eqn-form-D} 
	(in particular, with $\nu_0^\mathcal{D}$ as the coefficient of order $u^0$),
	and a family $\{q^{(0)}_u\}_{u\in z(U)}$ of holomorphic quadratic differentials with $\underset{u \rightarrow 0}{\lim} q^{(0)}_u = q^{(0)}_0 (\mathcal{D})$ such that 
	\begin{align*}
		&q_u = q^{(0)}_u 
		+ \sum_{s\in\{ \pm\}} \nu_{s}(u) q^{(1)}_{u, s}
		- \frac{3}{4} \sum_{s\in \{\pm\}} q^{(2)}_{u, s}
		+ \sum_{r=3}^{d} \nu_r(u) q_{u, x_r}^{(1)} - \frac{3}{4} \sum_{r=3}^{d} q^{(2)}_{u, x_r},
		&u \neq 0,
	\end{align*}
	is a meromorphic quadratic differential with simple apparent singularities at $x_3$, ..., $x_d$ and $x_\pm \in U$ with $z(x_\pm) = \pm u$.
	By \Cref{prop-db-limit-2}, $q_0 \coloneqq \underset{u \rightarrow 0}{\lim} q_u = q_\mathcal{D}$, 
	and hence $\{\mathcal{D}_u \coloneqq \mathcal{D}_{pr} + q_u \}_{u\in z(U)}$ defines a family we seek.
	
	To see that taking the projective monodromy representation 
	of $\mathcal{D}_u$ defines a holomorphic map $U \rightarrow \Hom(\pi_1, \mathrm{PSL}_2)$,
	consider a coordinate neighborhood $(V, w)$,
	where $V$ is distinct from $U\cup \{x_3, ..., x_d\}$ and $w$ is subordinate to the chosen holomorphic projective structure.
	On $(V, w)$, $\mathcal{D}_u$ takes the form $\pa_w^2 + q_u(w)$ 
	where $q_u(w)$ is a function holomorphic on $U \times V$.
	By standard results on differential equations that vary holomorphically with respect to deformation parameters \cite{ilyashenko-yakovenko},
	the local solutions to $\pa_w^2 + q_u(w)$ are holomorphic functions on $U \times V$.
	Analytically continuing the ratio of two such linearly independent solutions defines the projective monodromy representation of $\mathcal{D}_u$, 
	which is now holomorphic in $u$.
	By
{%
\cite[Corollary~5.4]{dinh-thesis}%
}%
, if the number of apparent singularities is less than $2g-2$ to start with, 
	this holomorphic map is injective.
\end{proof}

\subsection{Reduction of the number of apparent singularities}

In this subsection we will analyze the limit $\underset{u \rightarrow 0}{\lim} q_u$ 
if $\nu_\pm(u)$ take the form \eqref{eqn-form-B}.

\begin{lemma}\label{lemma-bb-limit-1}
	Let $x_0$, $x_3$, ..., $x_d$ be distinct points on $C$, 
	$(U,z)$ a coordinate neighborhood of $x_0$ with $z(x_0) = 0$,
	$\nu_{\pm}(u) = \pm \frac{3}{4u} \pm \nu' u + \mathcal{O}(u^2)$ holomorphic functions on $U \setminus \{x_0\}$,
	and $q^{(i)}_{u, \pm}$ defined as in \eqref{building-blocks-0} for $u \in z\left( U \right)$, $u \neq 0$.
	Then the family of meromorphic quadratic differentials	
	\begin{align} \label{family-merom-q-bb}
		\sum_{s\in \{\pm\}} \nu_s(u) q^{(1)}_{u, s} - \frac{3}{4} \sum_{s\in \{\pm\}} q^{(2)}_{u, s},
	\end{align}
	which is parameterized by $U \setminus \{x_0\}$, 
	extends to a family parameterized by $U$.
	The quadratic differential $\Delta q^{(0)}_{x_0}$ 
	corresponding to the extension to $x_0$
	is holomorphic on $C$,
	vanishes at $x_3$, ..., $x_{d}$,
	and on $U$ takes the form
	\begin{align}\label{delta-q}
		\Delta q^{(0)}_{x_0} = \left[ \frac{3}{4} \check{R}^{(1)}(z) - \frac{3}{2} R^{(2)}_0(z) \right] dz^2.
	\end{align}
\end{lemma}

{%
The proof is given in \Cref{app-technical-proofs}.
}%

Fix a holomorphic projective structure, and denote by $\mathcal{D}_{pr}$ the corresponding oper that has no apparent singularity. 
In the following we use coordinates subordinate to the chosen holomorphic projective structure.

\begin{proposition}\label{prop-bb-limit-2}
	Let $x_0$, $x_3$, ..., $x_{d}$ be distinct points on $C$ such that $2x_0 + x_3 + ... + x_d$ is $Q$-generic.
	Let $z$ be a coordinate on $U$ subordinate to the chosen holomorphic projective structure, with $z(x_0) = 0$.
	Let $\nu_3(u)$, ..., $\nu_d(u)$ be holomorphic functions on $U$,
	$\nu_{\pm}(u) = \pm \frac{3}{4u} \pm \nu' u + \mathcal{O}(u^2)$ holomorphic functions on $U\setminus \{x_0\}$,
	and $ \{\mathcal{D}_u = \mathcal{D}_{pr} + q_u \}_{u \in z(U), u \neq 0}$ the corresponding family of opers parameterized by $U\setminus \{x_0\}$
	where $q_u$ is defined as in \eqref{q-u-setup},
	with $\{ q^{(0)}_u \}_{u \in z(U), u \neq 0}$ the corresponding family of holomorphic quadratic differentials.
	
	If there exists a holomorphic quadratic differential $q^{(0)}_0$ with  $q^{(0)}_0 = \underset{u \rightarrow 0}{\lim} q^{(0)}_u$, 
	then $\{ \mathcal{D}_u \}_{u \in z(U), u\neq 0}$ extends to a family of opers parameterized by $U$.
	The oper corresponding to the extension to $x_0$ is $\mathcal{D}_0 = \mathcal{D}_{pr} + q_0$, where
	\begin{align}\label{bb-limit}
		q_0 \overset{\text{on } (U,z)}{=} {q^{(0)}_0} + \Delta q^{(0)}_{x_0}
		+  \sum_{r=3}^{d} \nu_r(0) q_{0, x_r}^{(1)} - \frac{3}{4} \sum_{r=3}^{d} q^{(2)}_{0, x_r},
	\end{align}
	where $\Delta q^{(0)}_{x_0}$ is defined as in \eqref{delta-q}.
	In particular, the apparent singularities of $\mathcal{D}_0$ are simple and located at $x_3$, ..., $x_d$ with accessory parameters $\nu_3(0)$, ..., $\nu_d(0)$.
\end{proposition}
\begin{proof}
	The proof follows directly from \Cref{lemma-bb-limit-1}.
{%
	The apparency constraints \eqref{SL-operator_residue-constraint-1} at $x_3, \dots, x_d$ hold for all $u \neq 0$ and pass to the limit $u \rightarrow 0$ by continuity. In contrast to \Cref{prop-db-limit-2}, there is no constraint to check at $x_0$, since $\Delta q^{(0)}_{x_0}$ is holomorphic.
}%
\end{proof}

{%
\begin{remark}[Holomorphic dependence of $g_u$]\label{rem-gu-holomorphic}
	For each $u\neq 0$, let $q^{reg}_u(z)$ be the regular part on $U$ of $q_u(z)$, i.e.
	\begin{align*}
		q^{reg}_u(z) 
		&\coloneqq q_u(z) -
		\left( - \frac{3}{4(z-u)^2} - \frac{3}{4(z+u)^2}
		+ \frac{\nu_+(u)}{z-u} + \frac{\nu_-(u)}{z+u} \right)  \nonumber \\
		&= q^{(0)}_u (z)
		+ \sum_{s \in \{\pm\}} \nu_s(u) R^{(1)}_{u,s}(z) 
		- \frac{3}{4} \sum_{s \in \{\pm\}} R^{(2)}_{u,s} (z) 
		+ \sum_{r=3}^{d} \nu_r(u) q_{u, x_r}^{(1)}(z) - \frac{3}{4} \sum_{r=3}^{d} q^{(2)}_{u, x_r} (z).
	\end{align*}
	As a function of both $u$ and $z$, the leading term in $u$ of $q^{reg}_u(z)$ 
	comes from the leading terms in $\nu_\pm(u) R^{(1)}_{u, \pm}(z)$ and is equal to $3(R^{(1)}_{u,+}(z) - R^{(1)}_{u,-}(z) )/( 4u)$.
	By part 2 of \Cref{lemma-linear-system-right-side}, 
	this has a well-defined $u\rightarrow 0$ limit; by the expansion in its proof, it is in fact holomorphic in $(u,z)$.
	Hence $q^{reg}_u(z)$ has a well-defined $u\rightarrow 0$ limit,
	and in fact $q^{reg}_u(z)  \overset{ u \rightarrow 0}{\longrightarrow} q^{reg}_0(z)$. Since moreover $q^{(0)}_u$ extends holomorphically to $u = 0$ (the coefficients \eqref{E1-E2} have removable singularities, cf. \Cref{lemma-tune-limit-hol-q}), and $\nu_r(u)$ and the building blocks depend holomorphically on $u$, the function $g_u(z) \coloneqq q^{reg}_u(z)$ in \eqref{eqn-qu-expansion} is holomorphic in both $u$ and $z$.
\end{remark}
}%

\begin{proposition}\label{prop-bb-limit-3}
	Let $\mathcal{D}$ be a branched projective structure with $\mathrm{div}(\mathcal{D}) = x_3+ ... + x_{d}$ for $d \leq 3g-3$,
	and $x_0\in C$ such that $2x_0 + x_3 + ... x_d$ is $Q$-generic.
	Then there exists a coordinate neighborhood $(U, z)$ centred at $x_0$
	and a family of branched projective structures $\{\mathcal{D}_u\}_{u\in z(U)}$ parametrized by $U$ such that
	\begin{enumerate}[(i)]
		\item $\mathcal{D}_0 = \mathcal{D}$;
		\item for $u\neq 0$, $\mathrm{div}(\mathcal{D}_u) = x_+ + x_- + x_3 + \dots x_d$ with $x_\pm \in U$ and $z(x_\pm) = \pm u$;
		\item for $u\neq 0$, the accessory parameters $\nu_\pm(u)$ of $x_\pm$ w.r.t. the local coordinate $z$ have Laurent expansions as functions of $u$ of Form (B) \eqref{eqn-form-B}.
	\end{enumerate}
	Furthermore, this family defines via taking monodromy a holomorphic map $U \rightarrow \Hom(\pi_1, \mathrm{PSL}_2)$, which is injective for $d < 2g-2$.    
\end{proposition}
\begin{proof}
	The proof is the similar to that of \Cref{prop-double-as}.
\end{proof}

\subsection{Crossing Simpson's partial oper strata}
\begin{proposition}\label{prop-family-dR}
	Let $\{\mathcal{D}_u\}_{u \in z(U)}$ be a family as in \Cref{prop-bb-limit-3} with $d$ even, 
	and suppose that the projective monodromy of $\mathcal{D}_0$ is irreducible.
	Then, after shrinking $U$, there exist a holomorphic family $\{[E_u, \nabla_u]\}_{u \in z(U)}$ in $\mathcal{M}_{dR}$
	and line subbundles $L_u \hookrightarrow E_u$ with $\mathcal{D}_u = \mathcal{D}_{(L_u \hookrightarrow E_u, \nabla_u)}$.
	In particular, for the closed strata \eqref{eqn-closed-strata},
	\[ [E_u, \nabla_u] \in \mathcal{M}^{\leq d}_{dR} \quad \text{for all } u \in z(U), \qquad \qquad [E_0, \nabla_0] \in \mathcal{M}^{\leq d-2}_{dR}. \]
\end{proposition}
\begin{proof}
	As the number of apparent singularities is even throughout the family $\{\mathcal{D}_u\}$, 
	we can lift the projective monodromy $\rho_u$ of $\mathcal{D}_u$ 
	to a holomorphic family $\check{\rho}_u$ of $\SL_2(\C)$-representations.
	Since $\rho_0$ is irreducible and irreducibility is an open condition, $\check{\rho}_u$ is irreducible for all $u \in z(U)$ after shrinking $U$.
	Since the Riemann--Hilbert correspondence is biholomorphic on irreducible representations \cite{simpson-moduli2}, 
	the flat connections $[E_u, \nabla_u]$ with monodromy $\check{\rho}_u$ form a holomorphic family. 
	Since $(E_u, \nabla_u)$ realizes the monodromy of $\mathcal{D}_u$, the oper $\mathcal{D}_u$ is induced by a line subbundle $L_u \hookrightarrow E_u$, i.e.\ $\mathcal{D}_u = \mathcal{D}_{(L_u \hookrightarrow E_u, \nabla_u)}$ (cf. \Cref{subsec-projective-structures}).
	Note that 
	\[ \deg(L_u^{-2}K) = \deg \mathrm{div}(\mathcal{D}_u) = 
	\begin{cases} 
		d, & u \neq 0, \\ 
		d-2, & u = 0. 
	\end{cases} \]
\end{proof}

The following corollary gives a geometric intuition: 
one can define a family that starts from a partial oper stratum and crosses to a smaller stratum by colliding apparent singularities of Form (B) \eqref{eqn-form-B}.
Its proof follows immediately from \Cref{prop-rank2-limits}(ii), since 
opers with fewer than $2g-2$ apparent singularities are induced by triples of the form $(L_E \hookrightarrow E, \nabla)$
where $E$ is strictly unstable with unique destabilizing line subbundle $L_E$.
\begin{corollary}\label{cor-crossing-strata}
	Given $d < 2g-2$, the family $[E_u, \nabla_u]$ of \Cref{prop-family-dR} satisfies
	\[ [E_u, \nabla_u] \in (N_d)^+_1 \quad \text{for } u \neq 0, \qquad \qquad \text{ and } \qquad [E_0, \nabla_0] \in (N_{d-2})^+_1 \subset \overline{(N_d)^+_1}. \]
	Here the last inclusion is Simpson's, recalled in \Cref{subsubsec-natural-stratification}.
\end{corollary}

For $d \geq 2g-2$, the subbundle $L_u$ is not destabilizing, and \Cref{prop-family-dR} does not exclude that $E_u$ admits another line subbundle of degree $\deg L_u + 1$ for all $u$, i.e.\ that $[E_u, \nabla_u] \in \mathcal{M}^{\leq d-2}_{dR}$ for all $u$.

\section{Relation to the isomonodromic bubbling surgery}
In this section, we let $\{ \mathcal{D}_u\}$ be a family of branched projective structures as in \Cref{prop-bb-limit-3}, 
with accessory parameters $\nu_\pm(u)$ of the colliding apparent singularities $x_\pm$ of Form (B) \eqref{eqn-form-B}.
We recall that the limit $\mathcal{D} = \mathcal{D}_0$ has two  apparent singularities fewer with $\mathrm{div}(\mathcal{D}) = x_3 + \dots + x_d$. 
The main result of this section is the proof that for all sufficiently small nonzero $u$, $\mathcal{D}_u$ is the output of a bubbling operation that creates the apparent singularities $x_\pm$

\subsection{Bubbling operation and twin paths}
In this section, we use the terminology \textit{branched projective structures} in place of opers with apparent singularities, 
as this notion does not presuppose a fixed complex structure and hence is more readily incorporated in the literature 
on isomonodromic bubbling \cite{CalsamigliaDeroinFrancaviglia2014}.

\paragraph{Developing maps and developed paths}
Let $\sigma$ be a branched projective structure on a smooth surface $\Sigma$.
Let $A_\sigma \colon \widetilde{\Sigma} \rightarrow \mathbb{P}^1$ be its developing map on the universal cover defined via analytic continuation of the ratio $w_\alpha$ of two independent local solutions of its oper (\Cref{subsubsec-projective-structures}).
By construction, $A_\sigma$ is equivariant with respect to the induced projective monodromy, i.e. $A_\sigma(\gamma \cdot \tilde{x}) = \rho(\gamma) A_\sigma(\tilde{x})$ for $\rho \colon \pi_1 \rightarrow \mathrm{PSL}_2(\C)$, and is unique up to post-composition with $\mathrm{PSL}_2(\C)$.

By construction, an apparent singularity $x$ defines a critical point of $A_\sigma$.
For example, near a lift $\tilde{x}$ of $x$, $A_\sigma = M(\zeta^2)$ for a local coordinate $\zeta$ centred at $\tilde{x}$ and some $M \in \mathrm{PSL}_2(\C)$.
Given a path $\alpha$ in $\Sigma$, we say that its developed path is
$\overline{\alpha} \coloneqq A_\sigma \circ \tilde{\alpha}$ for a lift $\tilde{\alpha}$ of $\alpha$ to $\widetilde{\Sigma}$; 
this developed path is well defined up to $\mathrm{PSL}_2(\C)$.
{%
More generally, for a subset $K \subset \Sigma$ contained in a simply connected open set, we write $A_\sigma(K) \coloneqq A_\sigma(\widetilde{K})$ for a lift $\widetilde{K}$ of $K$ to $\widetilde{\Sigma}$, well defined up to $\mathrm{PSL}_2(\C)$ \cite[Definition~2.4]{CalsamigliaDeroinFrancaviglia2014}; when several such subsets appear in one statement, such as $x_\pm$, $\gamma_1$, $\gamma_2$ and $\Delta$, they are lifted together inside a lift of a single simply connected open set containing them.
}%
For example, a small loop around a regular point develops onto a loop winding once around the image of that point, 
while a small loop around a simple apparent singularity $x$ develops onto a loop winding twice around $A_\sigma(\tilde{x})$.

\paragraph{Isomonodromic bubbling}
Let $\sigma$ be a branched projective structure on a Riemann surface $C'$ with underlying smooth surface $\Sigma$,
and $\gamma \subset \Sigma$ be an embedded arc along which the developing map $A_\sigma$ of $\sigma$ is injective. 
By cutting $\Sigma$ along $\gamma$ and gluing in a copy of $\mathbb{P}^1$ along the developed image $\overline{\gamma}$ of $\gamma$, 
we obtain a branched projective structure $\sigma' = \mathrm{Bub}(\sigma, \gamma)$ that 
\begin{enumerate}[(i)]
	\item has the same monodromy representation as $\sigma$;
	\item is subordinate to a different Riemann surface $C''$ with the same underlying smooth surface;
	\item has two more apparent singularities than $\sigma$ as the result of gluing extra $2\pi$ angles to the end of $\gamma$.
\end{enumerate}
We say that $\mathrm{Bub}(\sigma, \gamma)$ is the output of \textit{bubbling} of $\sigma$ along the arc $\gamma$. 
The inverse operation $\mathrm{Bub}(\sigma, \gamma) \mapsto \sigma$, obtained by collapsing the glued sphere, is called \textit{debubbling}. 

\begin{figure}[ht]
  \centering
  \includegraphics{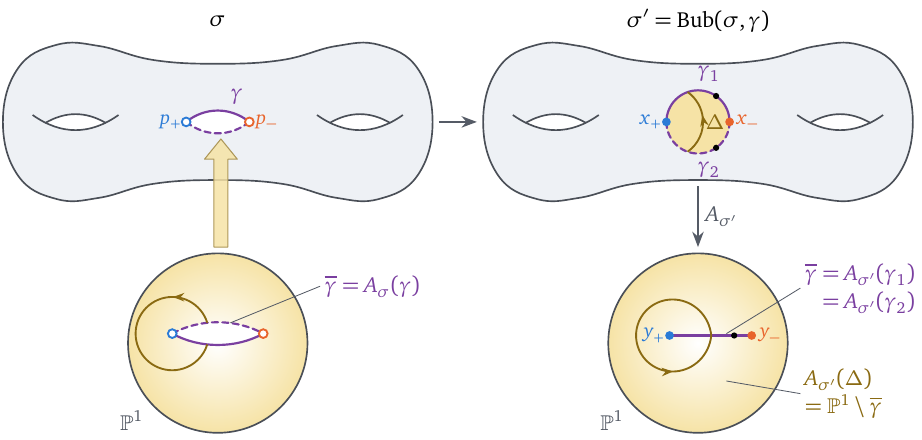}
  \caption{Bubbling of $\sigma$ along $\gamma$: 
  the complement of $\overline{\gamma}$ on $\mathbb{P}^1$ is glued in and becomes a disc bounded by the twin paths $\gamma_1$ and $\gamma_2$.
  The output is a new branched projective structure $\sigma'$ with two extra apparent singularities, 
  and generally subordinate to a different complex structure.}
  \label{fig-bubbling-surgery}
\end{figure}

\paragraph{Twin paths}
\begin{definition}[Definition 2.6 in \cite{CalsamigliaDeroinFrancaviglia2014}]
	Given a projective structure with developing map $A: \widetilde{\Sigma} \rightarrow \mathbb{P}^1$, we say that two paths $\gamma_1$ and $\gamma_2$ on $\Sigma$, both defined on the same interval $[0, T]$ are \textit{twin paths} if they overlap once developed, i.e.
	\begin{enumerate}
		\item $\gamma_1(0) = \gamma_2(0)$ is an apparent singularity;
		\item If $\alpha: [-T, T] \rightarrow \Sigma$ is given by $\alpha(t) = \gamma_1(t)$ for $t \geq 0$ and $\gamma_2(-t)$ for $t \leq 0$, then the developed $\overline{\alpha}$ of $\alpha$ is even: $\overline{\alpha}(t) = \overline{\alpha}(-t)$.
	\end{enumerate} 
	In particular, if $\gamma_1$ and $\gamma_2$ are embedded and disjoint except at $\gamma_1(0) = \gamma_2(0)$ then we call them \textit{embedded twin paths}.
\end{definition}

We will need the following configuration of twin paths to identify bubbling.
We will say that two twin paths $\gamma_1$ and $\gamma_2$ bound a bubble $\Delta$
if they are embedded and intersect precisely at two simple apparent singularities $x_+$ and $x_-$, 
and $\Delta$ is the disc bounded by $\gamma_1 \cup \gamma_2$ develops biholomorphically to $\mathbb{P}^1 \setminus \overline{\gamma}$ with $\overline{\gamma} = A(\gamma_1) = A(\gamma_2)$.
The following theorem, which identifies the output of bubbling in terms of twin paths,
follows from the proof of Theorem 5.1 in \cite{CalsamigliaDeroinFrancaviglia2014}.
\begin{theorem}\label{thm-twinpaths-bubble}
	Let $\sigma'$ be a branched projective structure on $\Sigma$, and let $\gamma_1$, $\gamma_2$ be twin paths bounding a bubble $\Delta$.
	Then $\sigma' = \mathrm{Bub}(\sigma, \gamma)$, where $\sigma$ is obtained from $\sigma'$ by collapsing $\Delta$ and identifying $\gamma_1(t)$ with $\gamma_2(t)$, and $\gamma$ is the resulting segment.
\end{theorem}
\begin{proof}
Since $\gamma_1$ and $\gamma_2$ have the same developed image, identifying $\gamma_1(t)$ with $\gamma_2(t)$ is compatible with the projective charts outside $\Delta$; at the two endpoints the cone angle drops from $4\pi$ to $2\pi$, so they become regular points of $\sigma$, and $\overline{\gamma}$ is the developed image of $\gamma$.
Gluing $\Delta$ back, a copy of $\mathbb{P}^1$ cut along $\overline{\gamma}$, along the arc $\gamma$ is the bubbling of $\sigma$ along $\gamma$.
\end{proof}

The main result of this section is the following theorem.
For simplicity, we will use interchangeably the notations for opers with apparent singularities $\mathcal{D}_u$ on $C$ 
and the corresponding branched projective structures.
\begin{theorem}\label{thm-bubble}
	Let $\{\mathcal{D}_u\}$ be a family of branched projective structures as in \Cref{prop-bb-limit-3}, with $\mathrm{div}(\mathcal{D}_u) = x_+ + x_- + x_3 + \dots + x_d$ for $u \neq 0$.
	Then there exists $\epsilon > 0$ such that for every $u$ with $0 < |u| < \epsilon$, the structure $\mathcal{D}_u$ has twin paths $\gamma_{u,1}$ and $\gamma_{u,2}$ from $x_+$ to $x_-$ bounding a bubble.
	Consequently, $\mathcal{D}_u = \mathrm{Bub}(\sigma_u, \gamma_u)$, where $\sigma_u$ is a branched projective structure with apparent singularities $x_3, \dots, x_d$ and the same monodromy as $\mathcal{D}_u$, on a Riemann surface with the same underlying smooth surface as $C$.
\end{theorem}

We present the proof of \Cref{thm-bubble} in \Cref{sect-family-blown-up,sect-existence-bubble}. 
The following is a summary of the key steps.
\begin{enumerate}[(1)]
	\item \textit{Rescaling.} Passing to the blow-up, i.e.\ to the coordinate $w = z/u$ in which the apparent singularities $x_\pm$ are fixed at $w = \pm 1$, 
	scales the quadratic differential by $u^2$ \eqref{eqn-pullback-Q-1}.
	In particular, the accessory parameters from Form (B) now take the form $u \nu_\pm(u) = \pm \frac{3}{4} + \mathcal{O}(u^2)$ \eqref{eqn-pullback-Q-2}.
	Hence the rescaled family extends holomorphically to $u = 0$, with limit \eqref{eqn-limit-P1} on the exceptional divisor $\mathbb{E}$.
	\item \textit{Developing maps.} Since the rescaled family of oper equations depends holomorphically on $u$, so do its developing maps $A_u$ on the disc $\mathbb{D}_R \subset \mathbb{E}$, and $A_0$ is, up to a M\"obius transformation, the Joukowsky function $J(w) = w + w^{-1}$ (\Cref{prop-family-developingmaps-1}).
	Specifically, $A_0$ is a limit developing map on the exceptional divisor $\mathbb{E} \cong \mathbb{P}^1$ of the blown-up surface.
	Normalizing $A_0 = J$, we have $A_u \rightarrow J$ uniformly on a suitable compact domain (\Cref{prop-family-developingmaps-2}).
	\item \textit{Constant degree 2.} By Rouch\'e's theorem, uniform convergence implies that on a suitable domain $\mathcal{W}_u$ containing $w = \pm 1$, 
	$A_u$ is a branched double cover over its image and 
	ramifies exactly at the critical points $w = \pm 1$ (\Cref{lem-degree-2}).
	\item \textit{Twin paths.} The two preimages along $A_u$ of the straight segment $\overline{\gamma}_u$ 
	connecting the critical values $A_u(\pm 1)$ consists of two arcs joining the critical points $w = \pm 1$.
	They are twin paths bounding a bubble.
	For $u \neq 0$, $w = z/u$ is just a change of coordinate on $C$, so the same holds for $\mathcal{D}_u$ in the original coordinate $z$. 
	\Cref{thm-bubble} now follows from \Cref{thm-twinpaths-bubble}.
\end{enumerate}

\subsection{Families of branched projective structures on blown-up surfaces}
\label{sect-family-blown-up}
\paragraph{Blown up surfaces}
We now consider the family $\mathcal{D}_u$ as a family of quadratic differentials on the surface $S = \mathbb{D}_\epsilon \times C$. Let $\pi_1$ be the projection to the first factor. Let 
\[ \widetilde{S} = \mathrm{Bl}_{(0, x_0)} S \overset{\beta}{\longrightarrow} S \] 
be the blow up of $S$ at $(0, x_0)$, and $\widetilde{\pi}_1 \coloneqq \pi_1 \circ \beta$.
Then 
\begin{align}
\pi_1^{-1}(u) \simeq \widetilde{\pi}_1^{-1}(u) \simeq C, \qquad \qquad u \neq 0,  
\end{align}
while the central fiber of $\widetilde{\pi}_1$ acquires in addition the exceptional divisor $\mathbb{E}$,
\begin{align}
\widetilde{\pi}_1^{-1}(0) \simeq C \cup_{x_0} \mathbb{E}, \qquad \qquad \mathbb{E} \simeq \mathbb{P}^1.
\end{align}

Given a coordinate neighborhood $(U,z)$ centered at $x_0$, 
the preimage $\beta^{-1}(\mathbb{D}_\epsilon \times U)$ can be parametrized by
\[ \beta^{-1}(\mathbb{D}_\epsilon \times U) = \left\{ \left( (u, z), [a:b] \right) \in S \times \mathbb{P}^1  \mid a z = b u  \right\}. \]
For $a \neq 0$, consider an open subset of $\widetilde{S}$ coordinated by $(u, w)$ via 
\[ (u, w) \longmapsto \left( (u, z), [a:b] \right) = \left(  (u, uw), [1, w]  \right) \subset \widetilde{S}, \]
i.e. we let $z = wu$ and $b = wa$. 
We see explicitly that 
\( \beta^{-1}\left( \{u \} \times U \right) \simeq U \)
for $u \neq 0$,  and for $u = 0$, 
\[ \mathbb{P}^1 \simeq \mathbb{E} = \{w \in \mathbb{C}\} \cup \{w = \infty\}  \]
where $\{ w = \infty \}$ corresponds to the node $\mathbb{E} \cap C$.

Note that, for nonzero $u$, the relation $w = z/u$ is a M\"obius transform. In the coordinates $(u, w)$, for $u\neq 0$, the apparent singularities $z = \pm u$ of the family $\mathcal{D}_u$ are fixed at $w = \pm 1$.
See \Cref{fig-blowup} for an illustration.

\begin{figure}[ht]
  \centering
  \includegraphics[width=\textwidth]{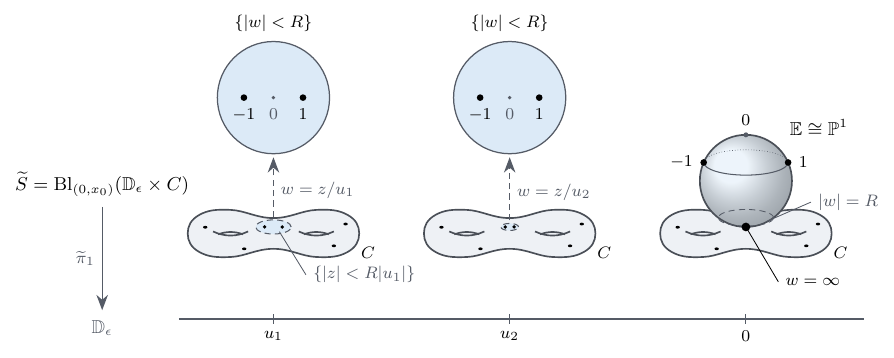}
  \caption{The blown-up family $\widetilde{\pi}_1 \colon \widetilde{S} \rightarrow \mathbb{D}_\epsilon$.
For $u \neq 0$, $\widetilde{\pi}_1^{-1}(u) \cong C$; the small disc $\{ |w| = |z/u| < R \}$ shrinks as $u \rightarrow 0$. 
The apparent singularities $z = \pm u$ are fixed at coordinates $w= \pm 1$ in the coordinates $(u, w)$.
At $u = 0$, $\widetilde{\pi}_1^{-1}(u) \cong C \cup \mathbb{E}$:
the node corresponds to $w = 0$ and $w \in \mathbb{C}$ parametrizes an open dense subset of the exceptional divisor $\mathbb{E} \cong \mathbb{P}^1$.}
  \label{fig-blowup}
\end{figure}

\paragraph{Pulling back families of branched projective structures}
Let $\{ \mathcal{D}_u \}_{u \in U}$ be a family of branched projective structures as in  \Cref{prop-bb-limit-3}, and let $\mathcal{Q}$ be the relative quadratic differential on $S$ that parametrizes $\mathcal{D}_u$, 
i.e. $\mathcal{Q} \mid_{\{u\} \times C} = q_u$ and $\mathcal{D}_u = \partial_z^2 + q_u(z)$ in local coordinates $z$ subordinate to a fixed reference holomorphic projective structure.
On the chart $(u, w)$ and for nonzero $u$, the local form of $\widetilde{\mathcal{Q}} \coloneqq \beta^\ast \mathcal{Q}$ is defined by the M\"obius transform $z \rightarrow w = z/u$,
\begin{equation}\label{eqn-pullback-Q-1}
	{\mathcal{Q}}(u, w) = u^2 q_u(uw).
\end{equation}

\begin{proposition}
	Let $\{\mathcal{D}_u\}$ be the family of projective structures as in \Cref{prop-bb-limit-3}, with accessory parameters $\nu_\pm(u)$ of Form B \eqref{eqn-form-B} 
	and $\mathcal{D} = \partial^2 + q_0$ as the $u \rightarrow 0$ limit. 
	Let $\mathcal{Q}$ be the corresponding relative quadratic differential on $S$.
	Then for nonzero $u$ and in local chart $(u, w)$, 
	\begin{equation}\label{eqn-pullback-Q-2}
		\beta^\ast \mathcal{Q}(u, w) =  \left( \frac{-3/4}{(w-1)^2} + \frac{-3/4}{(w+1)^2} + \frac{u \nu_+(u)}{w-1} + \frac{u \nu_-(u)}{w+1} + u^2 g_u(uw) \right) dw^2, 
		\qquad u\neq 0.
	\end{equation}
	The extension of $\beta^\ast \mathcal{Q}$ to $\pi^{-1}(0) \simeq C \cup_{x_0} \mathbb{E}$ restricts to $q_0$ on $C$ and
	\begin{equation}\label{eqn-limit-P1}
		\widetilde{\mathcal{Q}}_0 \mid_{\mathbb{E}} (w) = \frac{-3}{(w-1)^2 (w+1)^2} dw^2
	\end{equation}
	on the exceptional divisor $\mathbb{E} \simeq \mathbb{P}^1$.
\end{proposition}

\begin{proof}
	The local form \eqref{eqn-pullback-Q-2} follows by plugging \eqref{eqn-qu-expansion} into \eqref{eqn-pullback-Q-1} and substituting $z = uw$.
{%
Here $g_u(z)$ is holomorphic in $(u, z)$ by \Cref{rem-gu-holomorphic}.
}%
	Since Form (B) implies
	\[ u \nu_\pm(u) = \pm \frac{3}{4} \pm u^2 \nu' + \mathcal{O}(u^3), \]
	we can compute the restriction of limit $\underset{u \rightarrow 0}{\lim} \mathcal{Q}(u, w)$ to the open dense subset $\{ w \in \mathbb{C} \} \subset \mathbb{E}$ directly by taking the $u \rightarrow 0$ limit of the expression in \eqref{eqn-pullback-Q-2}. 
	This yields
	\[ \frac{-3}{4} \left[ \frac{1}{(w-1)^2} + \frac{1}{(w+1)^2} \right] 
	+ \frac{3}{4} \left[ \frac{1}{w-1} - \frac{1}{w+1} \right] \]
	which is equal to \eqref{eqn-limit-P1}
\end{proof}

\paragraph{Joukowsky function as the limit}
Let us now consider the Joukowsky function on $\mathbb{E}$ defined by 
\[ J(w) = w + w^{-1}\]

The following properties of $J(w)$ follows immediately by direct computation.
\begin{proposition}\label{prop-Joukowsky}
	\begin{enumerate}[(i)]
		\item
$J: \mathbb{P}^1 \rightarrow \mathbb{P}^1$ has degree $2$: for $\xi \in \mathbb{C}$, $J(w) = \xi$ if and only if $w^2 - \xi w + 1 = 0$, and $J^{-1}(\infty) = \{0, \infty\}$.
		\item The critical points of $J(w)$ are $w = \pm 1$, and the corresponding critical values are $J(\pm 1) = \pm 2$.
		\item $J(\mathbb{S}^1) = [-2, 2]$ where $\mathbb{S}^1$ is the unit circle.
		\item The restriction of $J$ to $\mathrm{int} (\mathbb{D}_1)$ is a biholomorphism to $\mathbb{P}^1 \setminus [-2, 2]$.
	\end{enumerate}
\end{proposition}

The role of the Joukowsky function is that it is a developing map of the limiting oper equation \eqref{eqn-limit-P1}.

\begin{proposition}
	$J(w)$ is a developing map of the differential equation defined by \eqref{eqn-limit-P1},
	\[ \left( \partial_w^2 - \frac{3}{(w-1)^2 (w+1)^2} \right) f(w) = 0.  \]
	In other words, there exists a basis of solutions $f_1(w)$ and $f_2(w)$ with $J = f_1/f_2$.
\end{proposition}
\begin{proof}
We have $J' = 1 - w^{-2} = \frac{w^2-1}{w^2}$, $J'' = \frac{2}{w^3}$ and $J''' = -\frac{6}{w^4}$, so that
\[ \frac{J''}{J'} = \frac{2}{w(w^2-1)}, \qquad \frac{J'''}{J'} = -\frac{6}{w^2(w^2-1)}. \]
Hence
\[ \{ J, w \} = \frac{J'''}{J'} - \frac{3}{2} \left( \frac{J''}{J'} \right)^2 
= -\frac{6}{w^2(w^2-1)} - \frac{6}{w^2(w^2-1)^2} 
= -\frac{6}{(w-1)^2(w+1)^2}, \]
i.e.\ $\frac{1}{2}\{J, w\}$ is the potential of \eqref{eqn-limit-P1}.
\end{proof}

Since the family of oper differential equations defined by $\mathcal{Q}(u, w)$ depends on $u$ holomorphically, 
one can solve for their solutions and obtain a holomorphic family of holomorphic solutions in a domain 
away from the fixed apparent singularities at $w = \pm 1$. 
Taking ratio of these solutions defines a holomorphic family of developing maps, which via analytic continuation admit simple critical points at $w = \pm 1$.
The following result is immediate.

\begin{proposition}\label{prop-family-developingmaps-1}
Let $\{\mathcal{D}_u\}$ be the family of projective structures as in \Cref{prop-bb-limit-3}. 
Fix $R > 1$, and consider $\mathbb{D}_R = \{ \lvert w \rvert < R \}  \subset \mathbb{E}$.
Then there exist $\epsilon > 0$ and a holomorphic map 
\[ \mathbf{A}: \mathbb{D}_\epsilon \times \mathbb{D}_R \longrightarrow \mathbb{P}^1, \qquad \qquad (u, w) \longmapsto A_u(w) \]
such that for all $u \in \mathbb{D}_\epsilon$, $A_u$ is a developing map of the restriction of $\beta^\ast \mathcal{D}_u$ to $\mathbb{D}_R$.
In particular, $A_u$ simply ramifies at $w = \pm 1$ and nowhere else on $\mathbb{D}_R$, and $A_0$ is M\"obius equivalent to the Joukowsky function $J = w + w^{-1}$.
\end{proposition}

\subsection{Proof of existence of bubble for nonzero $u$}
\label{sect-existence-bubble}
Let us now fix $M > 2$, and let
\[ \mathcal{A} \coloneqq J^{-1}(\mathbb{D}_M). \]
It follows from straightforward computation using polar coordinates $w = r e^{i\theta}$ that $\mathcal{A}$ contains the unit circle, is homeomorphic to an annulus and can be parametrized by 
\[ \mathcal{A} = \{ \rho(\theta)^{-1} < r < \rho(\theta)  \} \]
where $\rho(\theta) > 1$ satisfies
\[ \rho(\theta)^2 = \frac{1}{2}\left( c(\theta) + \sqrt{c(\theta)^2-4} \right) \qquad \text{ and }
\qquad  c(\theta) = M^2 - 2\cos 2\theta.  \] 
Let $\mathcal{B} \coloneqq \{ r < \rho \}$ and $\mathcal{H} \coloneqq \{ r < \rho^{-1} \}$.
Note that $\overline{\mathcal{B}} \setminus \mathcal{A} = \mathcal{H} \sqcup \partial \mathcal{A}$ where the boundary $\partial \mathcal{A}$ has two connected components, i.e. 
\[ \partial \mathcal{A} = J^{-1}(\partial \mathbb{D}_M) = \{ r = \rho^{-1} \} \sqcup \{ r = \rho \}.  \]
Fix $M' \in (2, M)$. For $u \in (0, \epsilon)$ in \Cref{prop-family-developingmaps-1}, let $\mathcal{W}_u \coloneqq A_u^{-1}(\mathbb{D}_{M'}) \cap \mathcal{B}$.
These regions are illustrated in \Cref{fig-regions}.

\begin{figure}[ht]
  \centering
  \includegraphics[width=\textwidth]{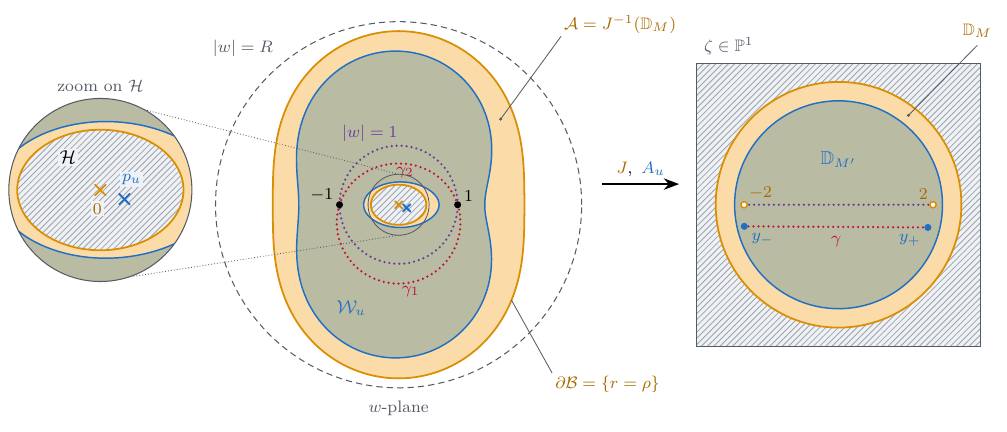}
  \caption{
The domains of \Cref{sect-existence-bubble} for one fixed small $u$. 
The domains and their boundaries defined by $J$ and $A_u$ are drawn in orange and blue respectively.
Twin paths and their developed image via $J$ and $A_u$ are drawn in purple and red respectively.
The poles of $J$ at $w=0$ and of $A_u$ at $w=p_u$ are contained in $\mathcal{H}$.}
  \label{fig-regions}
\end{figure}

The following lemma follows from straightforward computation and helps compare the domain $\mathbb{D}_R$ in \Cref{prop-family-developingmaps-1} 
with the annulus $\mathcal{A} = J^{-1}(\mathbb{D}_M)$.

\begin{lemma}\label{lem-B-in-DR}
	For $R > 0$, if $R - R^{-1} > M$ then $\overline{\mathcal{B}} \subset \mathbb{D}_R$.
\end{lemma}
\begin{proof}
	Note that $R > 2$ since $R - R^{-1} > M > 2$.
	For $|w| \geq R$ we have
	\[ |J(w)| = |w + w^{-1}| \geq |w| - |w|^{-1} \geq R - R^{-1} > M,\]
	since $t \mapsto t - t^{-1}$ is increasing.
	Hence $J^{-1}(\overline{\mathbb{D}}_M) \subset \mathbb{D}_R$, and in particular $\overline{\mathcal{A}} \subset \mathbb{D}_R$.
	Since $c(\theta) \geq M^2 - 2 > 2$, we have $\rho > 1$, and so $\overline{\mathcal{H}} \subset \overline{\mathbb{D}}_1 \subset \mathbb{D}_R$.
	Therefore $\overline{\mathcal{B}} = \overline{\mathcal{A}} \cup \overline{\mathcal{H}} \subset \mathbb{D}_R$.
\end{proof}

\paragraph{Uniform convergence}
Observe that the central developing map $J$ has a pole at $w = 0 \in \mathcal{H} \subset \overline{\mathcal{B}}$.
In the normalized family of developing maps in the following result, we will set $A_0 = J$,
and hence $A_u$ also has a pole in $\overline{\mathcal{B}}$ for sufficiently small $u$.
To measure the distance between $A_u$ and $J$ on $\overline{\mathcal{B}}$, we will hence need to work with the chordal metric on $\mathbb{P}^1$,
i.e. the Euclidean distance in $\mathbb{R}^3$ between stereographic images,
\[ d(a, b) = \frac{2|a-b|}{\sqrt{(1+|a|^2)(1+|b|^2)}}, 
\qquad d(a, \infty) = \frac{2}{\sqrt{1+|a|^2}}. \]
Note that $(\mathbb{P}^1, d)$ is a compact metric space inducing the usual topology.

\begin{proposition}\label{prop-family-developingmaps-2} 
There exists a normalized family $\mathbf{A}$ as in \Cref{prop-family-developingmaps-1} with $A_0 = J$ and such that $A_u$ converges uniformly to $J$ as $u \rightarrow 0$ on $\overline{\mathcal{B}}$ with respect to the chordal metric. 
\end{proposition}
\begin{proof}
	Choose $R$ in \Cref{prop-family-developingmaps-1} with $R - R^{-1} > M$, so that $\overline{\mathcal{B}} \subset \mathbb{D}_R$ by \Cref{lem-B-in-DR}.
	Let $T$ be the M\"obius transformation with $T \circ A_0 = J$.
	Since developing maps are defined up to post-composition with M\"obius transformations, $(u, w) \mapsto T \circ A_u(w)$ is again a family as in \Cref{prop-family-developingmaps-1};
	replacing $A_u$ by $T \circ A_u$, we may assume $A_0 = J$.
	The map $\mathbf{A}$ is continuous on the compact set $\overline{\mathbb{D}}_{\epsilon/2} \times \overline{\mathcal{B}}$, hence uniformly continuous there with respect to $d$ by the Heine--Cantor theorem \cite[Theorem~4.19]{Rudin76}.
	In other words, for every $\eta > 0$ there is $\tau > 0$ with $d(A_u(w), J(w)) = d(A_u(w), A_0(w)) < \eta$ for all $|u| < \tau$ and $w \in \overline{\mathcal{B}}$.
\end{proof}

We will use Rouch\'e's theorem (\Cref{thm-Rouche}), which needs Euclidean estimates instead, in the proof of the main \Cref{thm-bubble}. 
The following \Cref{lem-gap} translates chordal estimates into Euclidean ones.

\begin{lemma}\label{lem-gap}
	Let $0 \leq s < r$, and $a, b \in \mathbb{P}^1$ with $|a| \leq s$ and $|b| \geq r$. Then $d(a, b) \geq d(s, r)$.
\end{lemma}
{%
\begin{proof}
	Put $x = |a|$ and $y = |b|$, so that $x \leq s < r \leq y$.
	If $b = \infty$, then $d(a, b) = 2/\sqrt{1+x^2} \geq 2/\sqrt{1+s^2} > d(s, r)$, since $r - s < \sqrt{1+r^2}$.
	Otherwise $|a - b| \geq y - x > 0$, and $(1+x^2)(1+y^2) = (1+xy)^2 + (y-x)^2$ gives
	\[ d(a, b) \geq \frac{2(y-x)}{\sqrt{(1+xy)^2 + (y-x)^2}} = h\big(\tau(x,y)\big), \qquad 
	\tau(x, y) \coloneqq \frac{y-x}{1+xy}, \qquad h(\tau) \coloneqq \frac{2}{\sqrt{1+\tau^{-2}}}. \]
	Since $h$ is increasing and $d(s, r) = h(\tau(s, r))$, it remains to note that $\tau$ decreases in $x$ and increases in $y$:
	\[ \tau(x, y) - \tau(s, y) = \frac{(s-x)(1+y^2)}{(1+xy)(1+sy)} \geq 0, \qquad 
	\tau(s, y) - \tau(s, r) = \frac{(y-r)(1+s^2)}{(1+sy)(1+sr)} \geq 0. \]
\end{proof}
}%

\paragraph{From uniform convergence to constant degree 2}
We recall the relevant form of Rouch\'e theorem before applying it in the proof of \Cref{lem-degree-2}.
\begin{theorem}\label{thm-Rouche} (Rouch\'e) 
	Let $\Omega \subset \mathbb{C}$ be a regular domain, and $f$, $g$ holomorphic on an open neighborhood of $\overline{\Omega}$ with $|f-g| < |f|$ on $\partial \Omega$.
	Then the number of zeroes of $f$ on $\Omega$, counted with multiplicity, is equal to that of $g$.
\end{theorem}

\begin{lemma}\label{lem-degree-2}
Let $\mathbf{A}$ be as in \Cref{prop-family-developingmaps-2}. Then there exists $\epsilon$ such that for all $u \in \mathbb{D}_\epsilon$, 
\[A_u: \mathcal{W}_u \rightarrow \mathbb{D}_{M'} \]  
is a branched double cover ramified exactly at $w = \pm 1$.
\end{lemma}
\begin{proof}
Let $E(u) \coloneqq \underset{w \in \overline{\mathcal{B}}}{\sup} d\left( A_u(w), J(w) \right)$. 
By uniform convergence on $\overline{\mathcal{B}}$ (cf. \Cref{prop-family-developingmaps-2}), there exists $\epsilon$ such that for all $u \in \mathbb{D}_\epsilon$ we have 
\begin{equation}\label{eqn-E(u)-bound}
	E(u) < \min \left\{d(M', M), d(M, M+1), \frac{M-M'}{\kappa} \right\}, 
\qquad \kappa = \frac{1}{2}\left( 1 + (M+1)^2 \right).
\end{equation}
For a fixed $u \in \mathbb{D}_\epsilon$, we claim that
\begin{enumerate}[(a)]
	\item $|A_u| > M'$ on $\overline{\mathcal{B}} \setminus \mathcal{A} = \mathcal{H} \sqcup \partial \mathcal{A}$;
	\item $|A_u| < M+1$ on $\overline{\mathcal{A}}$;
	\item $\underset{\overline{\mathcal{A}}}{\sup}|A_u - J| < M-M'$.
\end{enumerate}
To see (a), suppose that there exists $w \in \overline{\mathcal{B}} \setminus \mathcal{A}$ with $|A_u(w)| \leq M'$. 
Since $|J(w)| \geq M$, it follows from \Cref{lem-gap} that 
$d\left( A_u(w), J(w) \right) \geq d(M', M)$, which contradicts the bound \eqref{eqn-E(u)-bound}.
To see (b), suppose that there exists $w \in \overline{\mathcal{A}}$ with $|A_u(w)| \geq M+1$. Since $|J(w)| \leq M$, it follows again from \Cref{lem-gap} that $d\left( A_u(w), J(w) \right) \geq d(M, M+1)$, which contradicts the bound \eqref{eqn-E(u)-bound}.
To see (c), note that one can write
\[ |A_u - J| = \frac{1}{2} \sqrt{(1+|A_u|^2) (1+ |J|^2)} \; d(A_u, J). \]
It follows from (b) that $|A_u - J| \leq \kappa \, d(A_u, J)$ on $\overline{\mathcal{A}}$. 
Part (c) now follows from the fact that we have chosen $\epsilon$ such that $d(A_u(w), J(w)) \leq E(u) < (M-M')/\kappa$ for all $w \in \overline{\mathcal{A}}$.  

Observe that it follows from (a) that $\mathcal{W}_u \subset \overline{\mathcal{W}_u} \subset \mathcal{A}$. 
Properness of the map $A_u: \mathcal{W}_u \rightarrow \mathbb{D}_{M'}$ follows from (a) and the compactness of $\overline{\mathcal{B}}$.
Indeed, let $K \subset \mathbb{D}_{M'}$ be compact. 
By (a), we have $A_u^{-1}(K) \cap \mathcal{W}_u = A_u^{-1}(K) \cap \overline{\mathcal{B}}$, which is a closed subset of $\overline{\mathcal{B}}$ and hence compact.
To compute the degree of $A_u$, for $\xi \in \mathbb{D}_{M'}$, let $f = J - \xi$ and $g = A_u - \xi$. On $\partial \mathcal{A}$ where $|J| = M$, we have
\[ |f - g| = |A_u - J| < M - M' < M - |\xi| = |J| - |\xi| \leq |J - \xi| = |f|. \]
By (b), $A_u$ has no poles near $\overline{\mathcal{A}}$, so $f$ and $g$ are holomorphic on a neighborhood of $\overline{\mathcal{A}}$.
We therefore can apply Rouch\'e Theorem \ref{thm-Rouche} to conclude that the number of preimages of $\xi$ via $A_u$ is equal to that via $J$ on $\mathcal{A}$, which is $2$ by \Cref{prop-Joukowsky}(i).
Since $\mathcal{W}_u = A_u^{-1}(\mathbb{D}_{M'}) \cap \mathcal{B} = A_u^{-1}(\mathbb{D}_{M'}) \cap \mathcal{A}$ by (a), the preimages of $\xi \in \mathbb{D}_{M'}$ in $\mathcal{A}$ are contained in $\mathcal{W}_u$. 
Finally, by continuity of $\mathbf{A}$ and since the branch locus of $A_0 = J$ is $J(\pm 1) = \pm 2$, 
we can further shrink $\epsilon$ if necessary to have $|A_u(\pm 1)| < M'$.
Then the ramification loci $\{ w = \pm 1 \} \subset \mathcal{A}$ are contained in $\mathcal{W}_u$. 
By construction of the family (\Cref{prop-family-developingmaps-1}), $A_u$ ramifies at $\pm 1$ and nowhere else.
\end{proof}

We now give the proof of existence of twin paths and hence of bubble for $u \neq 0$.

\vspace{5pt}
\noindent	
\textit{Proof of \Cref{thm-bubble}}
We consider the holomorphic family of developing maps in \Cref{prop-family-developingmaps-2} with $\epsilon$ given in \Cref{lem-degree-2}.
Fix nonzero $u \in \mathbb{D}_\epsilon$, and let $y_{u, \pm} \coloneqq A_u(\pm 1) \in \mathbb{D}_{M'}$ be the branch points of $A_u$.
Let $\overline{\gamma}_u$ be the straight segment connecting $y_{u, \pm}$: 
by construction, we have $\overline{\gamma}_u \subset \mathbb{D}_{M'}$.
Then the preimage $A_u^{-1}(\overline{\gamma}_u)$ consists of exactly two segments $\gamma_{u,1}$ and $\gamma_{u,2}$ that connect the critical points $w = \pm 1$
and develop to the same image $A_u(\gamma_{u,1}) = A_u(\gamma_{u,2}) = \overline{\gamma}_u$.

Let $\Delta_u$ be the bounded component of $\mathbb{C} \setminus (\gamma_{u,1} \cup \gamma_{u,2})$.
Then $\overline{\Delta_u} \subset \mathcal{B}$ and $A_u(\Delta_u) \subset \mathbb{P}^1 \setminus \overline{\gamma}_u$.
The map $A_u: \Delta_u \rightarrow \mathbb{P}^1 \setminus \overline{\gamma}_u$ is proper and a local homeomorphism, 
since $w = \pm 1$ are the only critical points of $A_u$ on $\mathbb{D}_R$. 
This means $\Delta_u$ is a covering of $\mathbb{P}^1 \setminus \overline{\gamma}_u$. 
Since both of these domains are connected and simply connected, $A_u$ must map $\Delta_u$ biholomorphically onto $\mathbb{P}^1 \setminus \overline{\gamma}_u$. 

We now recall that the coordinates change $z \rightarrow w = z/u$ is a M\" obius transform for nonzero $u$. 
In other words, $A_u(z/u)$ is a developing map of $\mathcal{D}_u$. 
Hence the images of $\gamma_{u,1}$ and $\gamma_{u,2}$ under the inverse M\" obius transform are twin paths connecting the apparent singularities $x_\pm$ of $\mathcal{D}_u$, bounding a disc that develops biholomorphically onto $\mathbb{P}^1 \setminus \overline{\gamma}_u$. 
Let $\sigma_u$ be the projective structure obtained from $\mathcal{D}_u$ by collapsing the disc $\Delta_u$, 
and $\gamma_u$ be the segment on $\sigma_u$ along which the two twin paths are glued, which develops onto $\overline{\gamma}_u$.
By \Cref{thm-twinpaths-bubble}, we have $\mathcal{D}_u = \mathrm{Bub}(\sigma_u, \gamma_u)$.
\qed

\appendix
{%
\section{Proofs of technical lemmas}\label{app-technical-proofs}
This appendix contains the proofs of the technical lemmas in \Cref{sect-degeneration-apparent-singularities}.
}%

\paragraph{%
{%
Proof of \Cref{lemma-linear-system-right-side}.%
}%
}
\begin{proof}
	1. We need to show that the $u\rightarrow 0$ limit of $\frac{1}{u} R^{(i)}_{u,+}(-u)$ and $-\frac{1}{u} R^{(i)}_{u,-}(u)$ exist and are equal.
	Consider the Taylor expansion of $R^{(i)}_{0}(z_1, z) = \sum_{m, n \geq 0} R^{(i)}_{m,n} z_1^m z^n$ at its zero $(z_1, z) = (0,0)$.
	Since $R^{(i)}_{0,0} = 0$, 
	\begin{align*}
		R^{(i)}_{u,-}(u) &\coloneqq R^{(i)}_{0}(-u, u) = - \left( R_{1,0} - R_{0,1} \right) u + \mathcal{O}(u^2), \\
		R^{(i)}_{u,+}(-u) &\coloneqq R^{(i)}_{0}(u, -u) =  \left( R_{1,0} - R_{0,1} \right) u + \mathcal{O}(u^2).
	\end{align*}
	It follows that
	\begin{align}
		\underset{u \rightarrow 0}{\lim} \frac{1}{u} R^{(i)}_{u,+}(-u) 
		= - \underset{u \rightarrow 0}{\lim} \frac{1}{u} R^{(i)}_{u,-}(u).
	\end{align}
	2. For each fixed $z$, it follows from the Taylor expansion $R^{(i)}_{0}(z_1, z) = \sum_{m, n \geq 0} R^{(i)}_{m,n} z_1^m z^n$ that, 
	\begin{align*}
		R^{(i)}_{u,+}(z) - R^{(i)}_{u,-}(z) 
		= 2u \sum_{n\geq 0} R^{(i)}_{1,n}z^n + 2u^3 \sum_{n\geq 0} R^{(i)}_{3,n} z^n + 2u^5 \sum_{n\geq 0} R^{(i)}_{5,n} z^n + ...
	\end{align*}
	is a holomorphic function on $(U,u)$ that vanishes at $u = 0$; we can calculate
	\begin{align}
		\underset{u \rightarrow 0}{\lim} \frac{1}{u}\left( R^{(i)}_{u,+}(z) - R^{(i)}_{u,-}(z) \right) = 2 \sum_{n\geq 0} R^{(i)}_{1,n}z^n.
	\end{align}
\end{proof}

\paragraph{%
{%
Proof of \Cref{lemma-basis-of-quadratic-differentials}.%
}%
}
\begin{proof}
	For each $k$ in the range $3 \leq k \leq 3g-3$ and each $u \in z(U)$,
	there exists a unique up to scaling quadratic differential 
	that vanishes at $x_1(u)$, $x_2(u)$ and all $x_{r}$ with $r \neq k$.
	By our assumption on $x_0$ and $\mathbf{x'}$,
	this quadratic differential does not vanish at $x_k$.
	By scaling, for $3 \leq k \leq 3g-3$, we define a unique quadratic differential $q_{k,u}$ with $q_{k, u} \left( z_r(x_r) \right) = \delta_{kr}$.
	We now define $q_1 \equiv q_{1,u}$ and $q_2 \equiv q_{2,u}$ independent of $u$ as follows.
	Let $q_1$ be a (unique up to scaling) holomorphic quadratic differential that vanishes at $x_0 + \mathbf{x'}$.
	By our assumption on $x_0$ and $\mathbf{x'}$,
	its zero at $x_0$ is simple.
	For some $x_0' \in U$ with $x_0' \neq x_0$,
	let $q_2$ be a (unique up to scaling) holomorphic quadratic differential that vanishes at $x_0' + \mathbf{x'}$.
	By our assumption on $x_0$ and $\mathbf{x'}$, $q_2$ does not vanish at $x_0$ and hence is linearly independent from $q_1$.
	Then for each $u \in z(U)$, let $ \left( q_{k, u} \right) \coloneqq \left( q_1, q_2, q_{3, u}, ..., q_{3g-3, u} \right)$: this forms a basis of $H^0(K^2)$ for each $u\in z(U)$.
	It follows that the function
	\begin{align}
		\det (\mathbf{q_{r,k}})(u)
		&= \begin{vmatrix}
			&q_1(z(x_1)) &  &q_2(z(x_1)) & & 0  \\
			&q_1(z(x_2)) &  &q_2(z(x_2)) & & 0  \\
			& 0 &  & 0 & & I_{(3g-5) \times (3g-5)} 
		\end{vmatrix} \nonumber \\
		&= q_1(z(x_1))  q_2(z(x_2)) - q_1(z(x_2)) q_2(z(x_1)) \nonumber \\
		&= 2 q_{1}'(z(x_0)) q_{2}(z(x_0)) u + \mathcal{O}(u^2)
	\end{align}
	has a simple zero at $u=0$.  
\end{proof}

\paragraph{%
{%
Proof of \Cref{lemma-limit-linear-system}.%
}%
}
\begin{proof}
	This can be checked explicitly by plugging \eqref{eqn-form-B} and \eqref{eqn-form-D} in \eqref{linear-eqn-1}.
	Specifically, the coefficients at the orders $u^{-1}$ and $u^0$ of $\nu_\pm(u)$ ensure that $C_\pm(u)$ are regular at $u=0$, 
	while the coefficients at the order $u^1$ of $\nu_\pm(u)$ 
	together with part 1 of \Cref{lemma-linear-system-right-side} ensure that 
	$C_+(0) = C_-(0) = C_0$, where
	\begin{align}\label{limit-RHS}
		C_0 \coloneqq 
		\begin{cases}
			 -\nu' - \frac{3}{4} \underset{u \rightarrow 0}{\lim} \frac{1 }{u} R^{(1)}_{u,-}(u) 
			 &\text{ for } \nu_\pm(u) = \pm \frac{3}{4u} \pm \nu' u + \mathcal{O}(u^2), \\
			 -\nu_0^2 + \nu' - \frac{1}{4} \underset{u \rightarrow 0}{\lim} \frac{1}{u} R^{(1)}_{u, -}(u)
			 &\text{ for } \nu_\pm(u) = \mp \frac{1}{4u} + \nu_0 \pm \nu' u + \mathcal{O}(u^2).
		\end{cases}
	\end{align}
	Then the linear system \eqref{linear-system-total} limits to
	\begin{align}\label{linear-system-total-limit}
			&q_u^{(0)}(0) 
			= C_0, 
			&q_{u}^{(0)} (z_r(x_r)) 
			= - \nu_r(0)^2, 
	\end{align}
	for $3 \leq r \leq d$.
	Note that the dimension of $Q_{x_0 + \mathbf{x'}}$
	must have the minimal value $3g- 2 - d$
	since otherwise condition \eqref{condition-neighborhood} will be violated.
	This yields the rank $d-1$ of the system \eqref{linear-system-total-limit}.
\end{proof}

\paragraph{%
{%
Proof of \Cref{lemma-tune-limit-hol-q}.%
}%
}
\begin{proof}
We will make use of the representation \eqref{linear-system-decoupled} of \eqref{linear-system-total}
using a family $\mathbf{q}_{u} = \{ (q_{k, u})_{k = 1}^{3g-3} \}_{u \in z(U)}$ of basis of $H^0(K^2)$ 
constructed in \Cref{lemma-basis-of-quadratic-differentials}.
Let us use the basis corresponding to $u=0$ and expand $q = \sum_{k=1}^{3g-3} E^0_k q_{k, 0}$.
We need to show that there exist $\nu_\pm(u)$ of the prescribed forms such that the solutions to \eqref{linear-system-decoupled-1},
\begin{align}\label{E1-E2}
	&E_1(u) = \frac{C_+(u) q_2(-u) -  C_-(u) q_2(u)}{q_1(u) q_2(-u) -  q_1(-u) q_2(u)},
	&E_2(u) = \frac{ q_1(u) C_-(u)  -  q_1(-u) C_+(u)  }{q_1(u) q_2(-u) -  q_1(-u) q_2(u)},
\end{align}
where $C_\pm(u)$ are defined as in \eqref{linear-system-total},
have well-defined $u \rightarrow 0$ limits which are equal to $E_1^0$ and $E_2^0$ respectively.
To this end, denote by $N_1(u)$ and $N_2(u)$ the respective numerators in \eqref{E1-E2}.
Since $\underset{u\rightarrow 0}{\lim} C_\pm(u) = C_0$ (cf. \eqref{limit-RHS}), both $N_1(u)$ and $N_2(u)$ are holomorphic and vanish at $u = 0$.
Since the denominators of \eqref{E1-E2} have simple zero at $u=0$ by \Cref{lemma-basis-of-quadratic-differentials},
the $u\rightarrow 0$ limits of $E_1(u)$ and $E_2(u)$ are well-defined by L'Hôpital's rule.
These limits are determined by the coefficient at order $u^1$ of $N_1(u)$ and $N_2(u)$.
With $R^{(i)}_{u, \pm}(z) = \sum_{m,n} R^{(i)}_{m,n} (\pm u)^m z^n$, one can compute explicitly these coefficients to be
\begin{align*}
	N_1(u) &= \frac{q_2(0)}{2}\left[ -4 (\nu_{+,2} + \nu_{-,2}) 
	+ 3 \left( R^{(2)}_{0,1} - R^{(2)}_{1,0} + R^{(1)}_{2,0} - R^{(1)}_{1,1} + R^{(1)}_{0,2} \right) \right] u \nonumber \\
	& \hspace{6mm} + \frac{q_2'(0)}{2} \left( 4\nu_1 + 3 R^{(1)}_{1,0} - 3 R^{(1)}_{0,1} \right)u + \mathcal{O}(u^2),	 \nonumber \\
	N_2(u)  &= - \frac{q_1'(0)}{2} \left[ 4\nu_1 + 3 R^{(1)}_{1,0} - 3 R^{(1)}_{0,1} \right] u + \mathcal{O}(u^2)
\end{align*}
for the ansatz $\nu_\pm(u) = \pm \frac{3}{4u} \pm \nu_1 u + \nu_{\pm, 2}u^2 + \mathcal{O}(u^3)$,
and 
\begin{align*}
	N_1(u) &= -\frac{q_2(0)}{2} \left[ 8\nu_0 \nu_1 
	+ 4\nu_0 \left( R^{(1)}_{0,1} - R^{(1)}_{1,0} \right) 
	+ R^{(1)}_{2,0} - R^{(1)}_{1,1} + R^{(1)}_{2,0} 
	+ 3 R^{(2)}_{1,0} - 3 R^{(2)}_{0,1}  \right]u \nonumber \\
	&\hspace{6mm}
	+ \frac{q_2'(0)}{2} \left[ 4\nu_0^2 - 4\nu_1 + R^{(1)}_{0,1} - R^{(1)}_{1,0} \right]u +\mathcal{O}(u^2) \nonumber \\
	N_2(u)  &=  \frac{q_1'(0)}{2} \left( -4 \nu_0^2 + 4 \nu_1 - R^{(1)}_{0,1} + R^{(1)}_{1,0}  \right) u + \mathcal{O}(u^2)
\end{align*}
for the ansatz $\nu_\pm(u) = \mp \frac{1}{4u} + \nu_0 \pm \nu_1 u + \nu_{\pm, 2}u^2 + \mathcal{O}(u^3)$.

Since $q_2(0)$ and $q_1'(0)$ are non-zero, we can tune 
$(\nu_1, \nu_{\pm, 2})$ in the first case
and $(\nu_0, \nu_1)$ in the second case
to tune the coefficients of order $u^1$ of $N_1(u)$ and $N_2(u)$ as we want.
Hence there exist holomorphic functions $\nu_\pm(u): U\setminus \{x_0\} \rightarrow \C$ with the Laurent expansion at $u = 0$ of the form \eqref{eqn-form-B} or \eqref{eqn-form-D},
so that $\underset{u \rightarrow 0}{\lim} E_1(u) = E_1^0$ and $\underset{u \rightarrow 0}{\lim} E_2(u) = E_2^0$.
For $3 \leq r \leq d$, let $\nu_r \in \C$ be such that $E^0_{r} = -\nu_r^2$, and define $\nu_r(u) = \nu_r$ for all $u\in z(U)$.
The tuple $(E_1(u), E_2(u), E_3(u), ..., E_{3g-3}(u) ) = (E_1(u), E_2(u), E^0_3, ..., E^0_{3g-3})$ then defines for each $u \neq 0$ a holomorphic quadratic differential $q^{(0)}_u \coloneqq \sum_{k=1}^{3g-3} E_k(u) q_{k,u}$
that solves the linear system \eqref{linear-system-decoupled} 
defined by $(\nu_\pm(u), \nu_3(u), ..., \nu_d(u))$,
with $\underset{u \rightarrow 0}{\lim} q^{(0)}_u = q$ a solution to the corresponding limit linear system \eqref{linear-system-total-limit}.
\end{proof}

\paragraph{%
{%
Proof of \Cref{lemma-db-limit-1}.%
}%
}
\begin{proof}
	We will show explicitly that,
	as a multi-variable function with variables $u$ and $z$,
	\eqref{family-merom-q-db} limits to \eqref{q_x_0-local-form} as $u \rightarrow 0$.
	To this end, it suffices to show that given any $z \neq 0$,
	the $u \rightarrow 0$ limit of the evaluation of \eqref{family-merom-q-db} at $z$ 
	is the evaluation of \eqref{q_x_0-local-form} at $z$.	
	Plugging $\nu_{\pm}(u) = \mp \frac{1}{4u} + \nu_0 \pm \nu' u + \mathcal{O}(u^2)$ into \eqref{family-merom-q-db}
	yields
	\begin{align}\label{plug-in-db}
		\frac{-2z^2 - u^2}{(z-u)^2 (z+u)^2} 
		+ \frac{2\nu_0 z}{(z-u)(z+u)}
		- \frac{R^{(1)}_{u,+}(z) - R^{(1)}_{u,-}(z) }{ 4u }
		+ \nu_0 \left( R^{(1)}_{u,+}(z) + R^{(1)}_{u,-}(z) \right) & \nonumber \\
		 - \frac{3}{4} \left( R^{(2)}_{u,+}(z) + R^{(2)}_{u,-}(z) \right)
		+ \mathcal{O}(u)(z) &.
	\end{align}
	Here we have ignored the factor $dz^2$
	and denoted by $\mathcal{O}(u)(z)$ a function with variables $u$ and $z$ such that,
	given any $z' \neq 0$,
	the function $F_{z'}(u) \coloneqq \mathcal{O}(u)(z) \mid_{z = z'}$ is holomorphic on $\{ u \mid 0 \leq |u| \leq |z'|/2 \}$
	with $F_{z'}(u) \sim \mathcal{O}(u)$.
	It follows that $\underset{u \rightarrow 0}{\lim} \mathcal{O}(u)(z) = \mathcal{O}(u)(z)\mid_{u = 0}$, 
	which a priori defines a function in $z$,
	is the zero function. 
	By a similar argument, we can take the $u \rightarrow 0$ limit of the other terms in \eqref{plug-in-db} explicitly 
	and obtain $q^{(2)}_{x_0}$ as in \eqref{q_x_0-local-form} as the limit.
	
	To show that $q^{(2)}_{x_0}$ is holomorphic on $C \setminus \{x_0\}$, it suffices to show holomorphicity on $C\setminus U$.
	Given $p\in C\setminus U$ and a coordinate neighborhood $(V,w)$ of $p$, 
	let $q^{(i)}_{u,\pm}(w) dw^2$ be the local form of $q^{(i)}_{u,\pm}$ on $V$.
	We want to show that the evaluation of \eqref{family-merom-q-db} at $p$, which defines the function
	\begin{align}\label{evaluation-at-p}
		\sum_{s\in \{\pm\}} \nu_s(u) q^{(1)}_{u, s}(w(p)) - \frac{3}{4} \sum_{s\in \{\pm\}} q^{(2)}_{u, s}(w(p))
		= \frac{1}{4u}\left( -q^{(1)}_{u,+}(w(p)) + q^{(1)}_{u,-} (w(p))  \right) + \mathcal{O}(1)
	\end{align}
	in $u$ which is holomorphic at $u \neq 0$
	has a well-defined $u \rightarrow 0$ limit.
	Since $q^{(i)}_{u,\pm} \overset{u \rightarrow 0}{\rightarrow} q^{(i)}_0$, 
	which is holomorphic on $C \setminus \{x_0\}$,
	it suffices to show that the leading term $\frac{1}{4u}\left( -q^{(1)}_{u,+}(w(p)) + q^{(1)}_{u,-} (w(p))  \right)$
	of \eqref{evaluation-at-p} has a well-defined $u \rightarrow 0$ limit.
	This is achieved by an argument similar to the proof of part 2 of \Cref{lemma-linear-system-right-side}.
	Furthermore, for $p \in \{x_3, ..., x_d\}$, since \eqref{evaluation-at-p} vanishes for $u\neq 0$, it is identically zero for all $u \in z(U)$.	
\end{proof}

\paragraph{%
{%
Proof of \Cref{lemma-bb-limit-1}.%
}%
}
{%
\begin{proof}
	The proof is the same as that of \Cref{lemma-db-limit-1}, with $\nu_\pm(u)$ of Form (B) in place of Form (D).
	The only difference is the polar part on $U$: plugging $\nu_\pm(u) = \pm \frac{3}{4u} + \mathcal{O}(u)$ into \eqref{family-merom-q-bb} gives, up to $dz^2$,
	\[ -\frac{3}{4}\left[ \frac{1}{(z-u)^2} + \frac{1}{(z+u)^2} \right] + \frac{3}{4u}\left[ \frac{1}{z-u} - \frac{1}{z+u} \right] + \mathcal{O}(u)
	= - \frac{3u^2}{(z^2-u^2)^2} + \mathcal{O}(u), \]
	which tends to $0$ for every fixed $z \neq 0$, instead of to $-2/z^2$ as in \eqref{plug-in-db}.
	Hence the limit \eqref{delta-q} has no pole at $x_0$, and it is holomorphic on $U$ by \Cref{lemma-linear-system-right-side}.
\end{proof}
}%

\printbibliography
	
\end{document}